\documentclass[11pt,reqno]{amsart}
\usepackage[margin=2.5cm]{geometry}
\allowdisplaybreaks
\usepackage{mlmodern}

\usepackage[english]{babel}
\usepackage[latin1]{inputenc}
\usepackage[T1]{fontenc}
\usepackage{amssymb,amsmath,amstext,amsfonts,amsthm}
\usepackage{mathrsfs}
\usepackage{mathtools}
\usepackage{graphicx}
\usepackage{caption}
\usepackage[export]{adjustbox}
\usepackage{tikz-cd} 
\usetikzlibrary{matrix,arrows,arrows.meta,decorations.pathmorphing}
\usepackage[all]{xy}
\usepackage{booktabs} 
\usepackage{array}
\newcolumntype{P}[1]{>{\centering\arraybackslash}p{#1}}
\usepackage{subfig}
\usepackage{caption}
\usepackage{xcolor}
\usepackage{verbatim}
\usepackage{hhline}
\usepackage{relsize}
\usepackage[colorlinks=true, urlcolor=blue, linkcolor=blue, citecolor=blue]{hyperref}
\usepackage{multicol}
\usepackage{overpic}
\usepackage[noend]{algorithmic}

\usepackage{calc,ifthen,alltt,tabularx,capt-of}
\usepackage[linesnumbered,longend,ruled,vlined]{algorithm2e}
\usepackage{nicematrix}

\numberwithin{equation}{section}
\theoremstyle{plain}
\newtheorem*{theorem*}{Theorem}
\newtheorem*{Example*}{Example}
\newtheorem{theorem}{Theorem}
\numberwithin{theorem}{section}
\newtheorem{proposition}[theorem]{Proposition}
\newtheorem{lemma}[theorem]{Lemma}
\newtheorem{corollary}[theorem]{Corollary}

\theoremstyle{definition}
\newtheorem{definition}[theorem]{Definition}
\newtheorem{remark}[theorem]{Remark}
\newtheorem{example}[theorem]{Example}

\newcommand{\A}{\mathbb{A}}
\newcommand{\N}{\mathbb{N}}
\renewcommand{\P}{\mathbb{P}}
\newcommand{\Z}{\mathbb{Z}}
\newcommand{\R}{\mathbb{R}}
\newcommand{\C}{\mathbb{C}}
\newcommand{\T}{\mathbb{T}}

\newcommand{\V}{\mathbb{V}}

\newcommand{\BB}{\mathcal{B}}
\newcommand{\II}{\mathcal{I}}

\newcommand{\CC}{\mathcal{C}}
\newcommand{\VV}{\mathcal{V}}
\newcommand{\LL}{\mathcal{L}}
\newcommand{\EE}{\mathcal{E}}
\newcommand{\FF}{\mathcal{F}}
\newcommand{\OO}{\mathcal{O}}

\newcommand{\PP}{\mathcal{P}}

\newcommand{\TT}{\mathcal{T}}

\newcommand{\K}{\mathbb{K}}

\def\bd{{\boldsymbol{d}}}
\def\bm{{\boldsymbol{m}}}
\def\bn{{\boldsymbol{n}}}
\def\bs{{\boldsymbol{s}}}
\def\bone{{\boldsymbol{1}}}

\DeclareMathOperator{\Trop}{Trop}

\newcommand{\pr}{{\rm pr}}

\newcommand{\codim}{\operatorname{codim}}
\newcommand{\cont}{\operatorname{Cont}}
\newcommand{\conv}{\operatorname{conv}}
\newcommand{\defect}{\operatorname{def}}
\newcommand{\DC}{\operatorname{DC}}
\newcommand{\DL}{\operatorname{DL}}
\newcommand{\DD}{\operatorname{DD}}
\newcommand{\gDD}{\operatorname{gDD}}

\newcommand{\rank}{\operatorname{rank}}
\newcommand{\image}{\operatorname{im}}
\newcommand{\Vol}{\operatorname{Vol}}
\newcommand{\rowspan}{\operatorname{rowspan}}
\newcommand{\vertices}{\operatorname{vert}}

\newcommand{\mT}{\mathsmaller{\mathsf{T}}}
\newcommand{\mR}{\mathsmaller{\mathbb{R}}}

\newcommand{\mBW}{\mathsmaller{\mathrm{BW}}}
\newcommand{\mED}{\mathsmaller{\mathrm{ED}}}

\def\stackbelow#1#2{\underset{\displaystyle\overset{\displaystyle\shortparallel}{#2}}{#1}}

\makeatletter
\def\l@subsection{\@tocline{2}{0pt}{2.5pc}{5pc}{}}
\makeatother

\newcommand{\nocontentsline}[3]{}
\let\origcontentsline\addcontentsline
\newcommand\stoptoc{\let\addcontentsline\nocontentsline}
\newcommand\resumetoc{\let\addcontentsline\origcontentsline}

\title[Osculating Geometry and Higher-Order Distance Loci]{Osculating Geometry and Higher-Order Distance Loci}

\author{Sandra Di Rocco}
\author{Kemal Rose}
\address{Department of Mathematics, KTH Royal Institute of Technology, SE-100 44 Stockholm, Sweden}
\email{dirocco@kth.se}
\email{kemalr@kth.se}

\author{Luca Sodomaco}
\address{Max Planck Institute for Mathematics in the Sciences, Leipzig, Germany}
\email{luca.sodomaco@mis.mpg.de}

\begin{document}

\begin{abstract}
We discuss the problem of optimizing the distance function from a given point, subject to polynomial constraints. A key algebraic invariant that governs its complexity is the {\em Euclidean distance degree}, which pertains to first-order tangency.

We focus on the {\em data locus} of points possessing at least one critical point of the distance function that is normal to a higher-order osculating space. We study the {\em higher-order distance degree} of a morphism as an intersection-theoretic invariant involving jet bundles and higher-order polar loci. Our approach builds on foundational definitions and results developed by Piene, particularly regarding higher-order polar loci. We give closed formulas for generic maps, Veronese embeddings, and toric embeddings. We place particular emphasis on the Bombieri-Weyl metric, revealing that the chosen metric profoundly influences both the degree and birationality of the higher-order projection maps. Additionally, we introduce a tropical framework that represents these degrees as stable intersections with Bergman fans, facilitating effective combinatorial computation in toric settings.
\end{abstract}

\maketitle

{
\hypersetup{linkcolor=black}
\tableofcontents
}

\section{Introduction}

\subsection*{Motivation and Background}
Metric algebraic geometry investigates the geometry of algebraic varieties through notions of distance and approximation \cite{breiding2024metric}. At its core lies the problem of identifying points on a variety that are closest to a given point in ambient space, typically with respect to the Euclidean norm.
A key measure of the intrinsic complexity of these problems is the {\em Euclidean distance degree (EDD)}. This algebraic invariant counts the number of complex critical points of the squared Euclidean distance function when restricted to a given variety. Introduced in the foundational work \cite{DHOST}, the EDD is central in the theory of optimization over algebraic varieties and has catalyzed recent advances at the interface of geometry and computation.

The classical EDD reflects first-order geometry: critical points are governed by tangent spaces and orthogonality conditions between the variety and the direction vector from a data point. However, many applications across data science, signal processing, and numerical optimization require a more refined notion of proximity, one that accounts for curvature, torsion, or higher-order contact. See Example \ref{example: 4-view variety} for a concrete example in Computer Vision.
In manifold learning \cite{mordohai2010dimensionality,fefferman2016testing,fefferman2018fitting,amari2024adversarial,kiani2025hardness}, for instance, trajectory planning \cite{BrysonHo,zha2002optimal}, nonlinear regression, and geometric data fitting \cite{cazals2005estimating,seber1989nonlinear}, it is not sufficient to minimize the Euclidean distance merely; how a model surface curves toward the data is crucial for both interpretability and accuracy. Such problems require critical configurations that exhibit higher-order contact, modeled via osculating spaces rather than tangents alone.
This motivates the study of algebraic invariants that generalize Euclidean distance degrees and reflect higher-order tangency. Foundational work in this direction was developed by Piene in \cite{piene2022higher}, where she introduces and studies higher-order polar loci and reciprocal polar loci in a broad algebro-geometric setting. Her results are instrumental to the developments presented in this paper. We focus on the case of nonsingular varieties, which often allows for local computations and more streamlined arguments. In particular, in Sections \ref{sec: polar}, \ref{sec: higher-order normal bundles}, and \ref{sec: osculating eigenvectors}, we adopt and adapt some of Piene's definitions and results within our framework. Further related recent contributions include \cite{brandt2024voronoi,horobet2024critical,breiding2025critical}.

\subsection*{Geometric Setup and State of the Art.} We start by fixing an $(n+1)$-dimensional real vector space $V_\mR$ and a positive definite quadratic form $q\colon V_\mR\to\R$. We denote by $Q$ the quadric hypersurface defined by $q.$ For any data point $u\in V_\mR$, we denote by $d_u$ the distance function from $u$, defined by $d_u(v)\coloneqq\sqrt{q(u-v)}$ for all $v\in V_\mR$. For any fixed subset $S\subseteq V_\mR$ such that $u\notin S$, it is natural to study the set of points $x\in S$ attaining a local minimum of $d_u$ restricted to $S$. In this paper, we restrict ourselves to subsets $S$ that are affine cones over real algebraic varieties. As a prototypical example, we focus on affine cones over projective toric varieties \cite{Higher_duality_and_toric}, and in particular varieties of rank at most one symmetric tensors, seen as homogeneous polynomials \cite{banach1938uber}. More precisely, we let $Y_\mR$ be the image of a real morphism $f\colon X_\mR\to\P(V_\mR)=\P_{\mR}^n$ and $S=C(Y_\mR)$ be the affine cone over $Y_\mR$. Firstly, despite being naturally a problem defined over the real numbers, the algebraic study of the locus of critical points of $d_u$ requires working over the field of complex numbers. Secondly, we observe that a similar study may be carried out for more general affine varieties that are not cones. However, we restrict ourselves to affine cones over projective varieties to apply classical results of intersection theory in projective space. Motivated by these facts, we consider the complex vector space $V\coloneqq V_\mR\otimes\C$, the Zariski closure $Y$ of $Y_\mR$ in $\P(V)=\P^n$, and its affine cone $C(Y)\subseteq V$. In particular, we will always work with morphisms $f\colon X\to\P^n$, where $Y=f(X)$. Additionally, since the squared distance function $d_u^2$ is polynomial, we consider it as a polynomial function $d_u^2\colon V\to\C$. We emphasize that the real quadratic form $q$ induces only a bilinear symmetric form over $V$ and not a Hermitian inner product. We denote by $Q$ the nonsingular quadric hypersurface in $\P^n$ defined by the quadratic form $q$.

A generic complex data point typically has a nonsingular critical point for the distance function restricted to a reduced variety. Introducing additional conditions on the nature of critical points leads to richer algebraic structures known as {\em data loci}. For example, {\em conditional ED data loci} encode data points with at least one critical point on a prescribed subvariety \cite{horobet2017data,horobet2022data,dirocco2024Relative}. Other well-studied data loci are {\em $\varepsilon$-offsets}, representing points with a critical point at distance $\varepsilon$, {\em bisector hypersurfaces}, defined by data having two equidistant critical points \cite{horobet2019offset,ottaviani2020distance}, and {\em Voronoi cells}, consisting of all data points closest to a fixed point lying on a given variety \cite{cifuentes2022voronoi}.

In this work, we introduce and study the {\em higher-order distance degree $\DD_k(f,Q)$} and {\em higher-order distance loci} $\DL_k(f,Q)$ of a morphism $f\colon X\to\P^n$ with respect to a fixed nonsingular quadric hypersurface $Q\subseteq\P^n$. This invariant describes the algebraic complexity for $f$ to satisfy a prescribed osculating condition. In the following, we summarize the main steps that motivate our definition.

For a fixed integer $k \ge 1$ and a point $p \in X$, the {\em $k$th osculating space} $\T_p^k(f) \subseteq \P^n$ is defined as the projective span of all partial derivatives of the components of $f$ at $p$ up to order $k$, equivalently, the image of the $k$th jet map at $p$. This space captures all infinitesimal directions of the morphism $f$ at $p$ up to order $k$ and is the smallest projective linear subspace with this property.  For $k = 1$, the osculating space $\T_p^1(f)$ coincides with the projective tangent space $\T_p(f)$, which has dimension $m = \dim X$ at nonsingular points. If $f$ is an embedding, the tangent space has constant dimension by definition. However, for $k \ge 2$, the dimension of the osculating spaces $\T_p^k(f)$ may vary. We say that a morphism $f\colon X \to \P^n$ is {\em globally $k$-osculating} if the dimension of $\T_p^k(f)$ is constant across all points of $X$ (see Definition~\ref{def: globally k-osculating}).

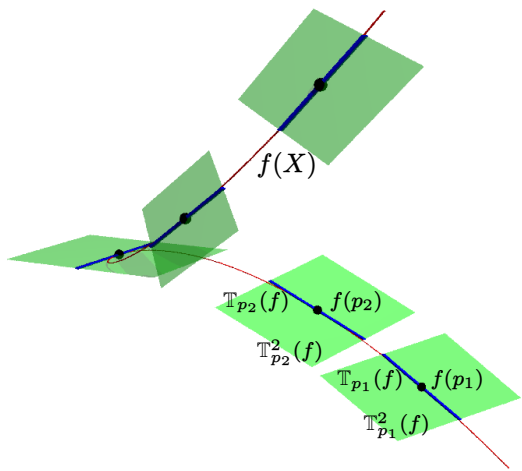
\begin{figure}[ht]
\centering
\begin{overpic}[width=0.45\textwidth]{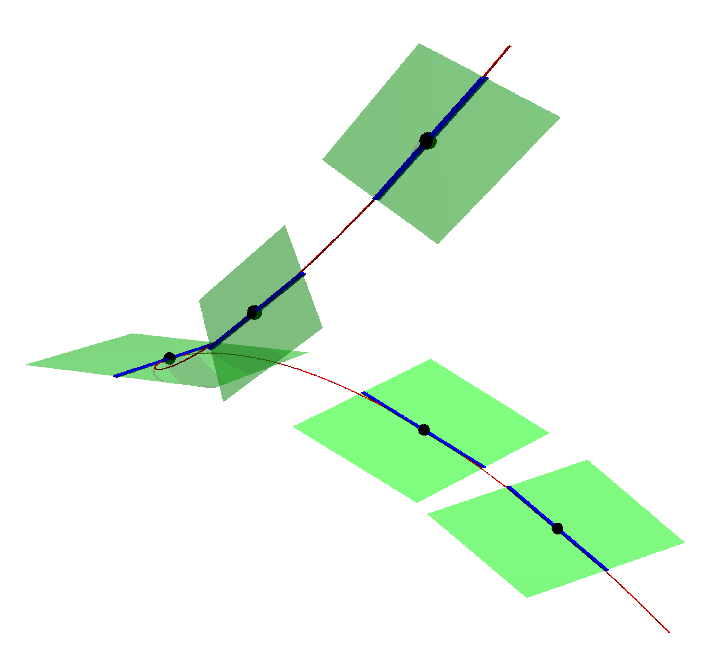}
\put (79,20) {\scriptsize{$f(p_1)$}}
\put (62.3,20) {\scriptsize{$\T_{p_1}(f)$}}
\put (67,12) {\scriptsize{$\T_{p_1}^2(f)$}}
\put (61,34) {\scriptsize{$f(p_2)$}}
\put (42,34) {\scriptsize{$\T_{p_2}(f)$}}
\put (48,25) {\scriptsize{$\T_{p_2}^2(f)$}}
\put (48,58) {\small{$f(X)$}}
\end{overpic}
\caption{Tangent lines and osculating planes of the twisted cubic image of the Veronese embedding $f=\nu_1^3\colon\P^1\hookrightarrow\P^3$ in the affine chart $\{u_0=1\}\cong\R^3$.}\label{fig: osculating spaces}
\end{figure}

In particular, a globally $1$-osculating parametrization is an immersion. 
Globally $k$-osculating morphisms are particularly well-suited to geometric analysis, as they admit a natural system of invariants known as {\em higher-order polar classes} $p_{k,i}(f)$ for $i \in \{0, \dots, m\},$ related to characteristic classes. Introduced in more generality by Piene in \cite{piene2022higher}, these classes are subvarieties of codimension $i$ in $X$ that encode the geometric behavior of the $k$th-order osculating spaces. In Sections~\ref{sec: osculating} and~\ref{sec: polar}, we recall the geometric framework of osculating functions and the main definitions and results on higher-order polar geometry from \cite{piene2022higher}, which forms the foundation for our main results.

Section \ref{sec: higher-order normal bundles} recalls the notion of higher-order normality with respect to a quadric $Q$. Given a morphism $f\colon X \to \P^n$ and a point $u \in V$, we say that $[v] = f(p) \in f(X)$ is {\em critical of order $k$} for the squared distance function $d_u^2$ if
\[
[\nabla d_u^2(v)] \in \N_p^k(f,Q) \coloneqq \langle \T_p^k(f)^\perp, f(p) \rangle\,,
\]
where $\langle S \rangle$ denotes the projective span of a subset $S \subseteq \P^n$, and $\T_p^k(f)^\perp$ is the {\em polar subspace} of the $k$th-order osculating space $\T_p^k(f)$ with respect to the quadric $Q$, see \eqref{eq: perp}. We refer to $\N_p^k(f,Q)$ as the {\em $k$th-order normal space} of $f$ at $p$.

The central object of study is the {\em $k$th-order distance locus} $\DL_k(f,Q)$, defined as the projection to the second factor $\P^n$ of the incidence variety
\[
\DC_k(f,Q) \coloneqq \overline{\left\{([v], [u]) \in \P^n \times \P^n \,\middle|\, [v] \in f(X) \text{ is critical of order } k \text{ for } d_u^2\right\}}\subseteq\P^n \times \P^n\,.
\]

For $k = 1$, this construction recovers the classical {\em projective ED correspondence} $\DC(f,Q)\coloneqq\DC_1(f,Q)$ and the {\em Euclidean distance degree} $\DD(f,Q)\coloneqq\DD_1(f,Q)$, defined as the degree of the surjective projection $\pr_2|_{\DC(f,Q)}\colon\DC(f,Q)\to\P^n=\DL_1(f,Q)$; see \cite[Theorem 4.4]{DHOST}.
The description of higher-order distance loci $\DL_k(f,Q)$ becomes increasingly intricate for $k \geq 2$: indeed, for generic $[u] \in \P^n$, a critical point of order $k$ may not exist on $f(X)$. Hence, unlike the case $k=1$, the data locus $\DL_k(f,Q)$ forms a proper algebraic subvariety of $\P^n$, reflecting the rarity and structural complexity of higher-order critical interactions.

To quantify its complexity, we consider the surjective morphism $\varphi_{2,k}\colon\DC_k(f,Q)\to\DL_k(f,Q)$ induced by $\pr_2$ and we define the {\em $k$th-order distance degree} of $(f,Q)$ as
\[
\DD_k(f,Q) \coloneqq \deg\DL_k(f,Q) \cdot \deg\varphi_{2,k}\,,
\]
as formalized in Definition~\ref{def: kth order ED degree}. In particular, when the pair $(f,Q)$ is in {\em general $k$-osculating position} (see Definition~\ref{def: general k osculating position}), the morphism $\varphi_{2,k}$ is generically finite over its image and $\DD_k(f,Q)$ attains its maximum value, called the {\em generic $k$th-order distance degree of $f$} and denoted by $\gDD_k(f)$ (see Definition \ref{def:generic DD}). In particular, this invariant is independent of $Q$.

\subsection*{Main Results.} 

The main result of Section \ref{sec: osculating eigenvectors} provides a closed formula for the $k$th-order distance degree of a generic polynomial map. Its proof is given in Section \ref{proof2}.

\begin{theorem}\label{thm: generic ED degree general polynomial map}
Let $f_0,\dots,f_n$ be $n+1$ generic polynomials in $\C[x_0,\dots,x_m]_d$ and let $f\colon\P^m\to\P^n$ be the associated morphism. Consider a nonnegative integer $k\le d$ and assume that $n>\binom{m+k}{k}-1$.
Then $f$ has generic $k$-osculating dimension $\binom{m+k}{k}-1$. Furthermore, for any nonsingular quadric hypersurface $Q\subseteq\P^n$ that intersects $f(\P^m)$ transversally, the morphism $\varphi_{2,k}$ is birational and
\[
\DD_k(f,Q) = \deg\DL_k(f,Q) = \sum_{i=0}^m \binom{\binom{m+k}{k}}{i}(d-k)^id^{m-i}\,.
\]
\end{theorem}

This result relies on an interesting interplay between generic and special metric properties. We begin by exploring the connection between higher-order osculating optimization and a concrete, widely studied problem in data approximation: the so-called best (symmetric) rank-one tensor approximation \cite{banach1938uber,eckart1936approximation,desilva2008tensor,ottaviani2014exact}.
In Definition \ref{def: kth order eigenpoint}, we introduce {\em higher-order eigenvectors} of symmetric tensors, a generalization of the notions introduced in \cite{lim2005singular,Qi05}. In Proposition \ref{prop: critical points BW distance}, we relate them to higher-order critical points of the nearest-point problem to the Veronese variety with respect to the Bombieri-Weyl quadratic form $Q_\mBW$, see Definition \ref{def: BW inner product}.
Furthermore, we compute the higher-order distance degrees of Veronese embeddings with respect to the Bombieri-Weyl quadratic form, and examine how the geometry of the higher-order distance loci $\DL_k(f,Q)$ varies with the choice of metric. In Section \ref{proof4}, we prove the following result.

\begin{theorem}\label{thm: projection data locus Veronese BW is either birational or of degree 2}
Let $\nu_m^d\colon\P^m\hookrightarrow\P(\C[x_0,\dots,x_m]_d)$ be the degree-$d$ Veronese embedding of $\P^m$, and equip $\R[x_0,\dots,x_m]_d$ with the Bombieri-Weyl inner product. For all $k\le d$,
\begin{enumerate}
    \item If $(m,k)=(1,d-1)$, then $\deg\varphi_{2,d-1}=2$ and $\deg\DL_{d-1}(\nu_1^d,Q_\mBW)=d-1$.
    \item If $(m,k)\neq (1,d-1)$, then $\varphi_{2,k}$ is birational and $\DD_k(\nu_m^d,Q_\mBW)=\deg\DL_k(\nu_m^d,Q_\mBW)$ equals the coefficient of the monomial $h_1^mh_2^{\binom{m+k}{k}-m-1}$ in the expansion of
    \begin{equation}\label{eq: kth order ED degree BW}
    (-1)^{\binom{m+k}{k}-1}\sum_{j=0}^\infty(-1)^j(dh_1+h_2+k(d-k)h_1^2+kh_1h_2)^j \in \frac{\Z[h_1,h_2]}{\langle h_1^{m+1},h_2^{\binom{m+d}{d}}\rangle}\,.
    \end{equation}
\end{enumerate}
Furthermore, for all $k\le d$, if $(\nu_m^d,Q)$ is in general $k$-osculating position, then $\varphi_{2,k}$ is birational and
\begin{equation}\label{eq: generic degree data locus Veronese}
\deg\DL_k(\nu_m^d,Q) = \gDD_k(\nu_m^d) = \sum_{i=0}^m \binom{\binom{m+k}{k}}{i}(d-k)^id^{m-i}\,.
\end{equation}
\end{theorem}

When $k=1$, Equation \eqref{eq: kth order ED degree BW} specializes to the distance degree $\DD(\nu_m^d,Q_\mBW)$ computed in \cite[Corollary 8.7]{DHOST}, a consequence of \cite[Theorem 5.6]{cartwright2013number}, see also \cite[Theorem 3.4]{oeding2013eigenvectors} and \cite[Theorem 12]{friedland2014number} for related formulas in a more general setting.

Another goal of this paper is to derive closed formulas for higher-order distance degrees. This is closely connected to the study of characteristic classes of the image bundle associated with the morphism $j_k$ in \eqref{eq: morphism jk}. In general, studying these invariants is technically involved. In Sections~\ref{sec: formulas for k-regular embeddings} and \ref{sec: tropical}, we address this challenge using two different approaches.

In Section~\ref{sec: formulas for k-regular embeddings}, we focus on the case of {\em $k$-regular embeddings}, that is, embeddings $f \colon X \hookrightarrow \P^n$ whose $k$th-order osculating spaces attain the maximal possible dimension globally. For such morphisms, we derive explicit formulas that are particularly well-suited to computational applications. We demonstrate their use in the context of embeddings of varieties of dimension at most three, with special focus on toric embeddings. 

In Section~\ref{sec: tropical}, we remove the assumption of $k$-regularity and use a tropical approach, 
which is suited to both combinatorial and computational perspectives.
We study the tropicalization of the higher-order conormal variety $\Trop W_k(f,Q)$ associated with a morphism $f$. This viewpoint offers a powerful framework for analyzing the behavior of higher-order critical loci under degenerations and for computing higher-order polar degrees combinatorially. We show that higher-order polar degrees can be expressed in terms of stable intersections between $\Trop W_k(f,Q)$ and Bergman fans $\operatorname{Berg}(M^{n-j+1}_{n+1}\times M^{j+m_k-m+2}_{n+1})$ corresponding to uniform matroids, obtaining the following tropical formula, whose proof is given in Section \ref{proof5}.

\begin{theorem}\label{theorem: tropical description of higher-order multidegrees}
Let $f$ be a monomial $k$-osculating embedding of generic $k$-osculating dimension $m_k$.
The generic $k$th-order distance degree of $f$ is the sum of the degrees
\[
\gDD_k(f) = \sum_{j = 0}^{ m+n-m_k-1 }
\int \Trop W_k(f,Q) \cdot \operatorname{Berg}(M^{n-j+1}_{n+1}\times M^{j+m_k-m+2}_{n+1})\,.
\]
\end{theorem}

When $f$ is a toric embedding, we further make these results effective and give a combinatorial description of $\Trop W_k(f,Q)$ in Proposition \ref{prop: description of conormal variety as Y_{C,D}}.
We accompany our theoretical findings with an easy-to-use \verb|Julia|-based software implementation, which enables the computation of higher-order distance degrees and polar degrees in this toric setting, similar to results from \cite{SturmfelsHelmer}.

Finally in Section~\ref{sec: affine} we discuss the case of affine morphisms. We adapt the higher-order distance construction to the affine setting and illustrate how the resulting loci and degrees can be computed in concrete examples. This also provides additional evidence for possible applications, see Proposition~\ref{prop: hessian}.

In the following, we present an example that provides a clear illustration of how we can explicitly compute and visualize all our constructions, including osculating spaces, higher-order normal bundles, polar loci, and distance degrees. Additionally, it connects our intersection-theoretic framework with issues related to low-rank tensor approximation, which we revisit in Section \ref{sec: osculating eigenvectors}.

\subsection*{Illustrative Example} Consider the space $V=(\C^2)^*=\C[x_0,x_1]_1$ of linear forms in $x=(x_0,x_1)$, $X=\P(V)=\P^1$, and the morphism $f\colon \P^1\to\P(S^3V)\cong\P^3$ defined by $f([\ell])\coloneqq[\ell^3]$ for every $\ell\in V$. Its image $f(\P^1)$ is a nonsingular twisted cubic in $\P^3$. In coordinates we write $\ell(x)=t_0x_0+t_1x_1$, hence $\ell^3(x)=\sum_{i=0}^3\binom{3}{i}(t_0x_0)^{3-i}(t_1x_1)^i$ and we define $f([t_0 : t_1])\coloneqq[t_0^3 : t_0^2 t_1 : t_0 t_1^2 : t_1^3]$.
In particular, $f$ corresponds to the Veronese embedding $\nu_1^3\colon\P^1\hookrightarrow\P^3$, and its image consists of all cubic binary forms of rank at most one.
Working in the affine patch $\{t_0\neq 0\}$ and using the local coordinate $t=\frac{t_1}{t_0}$, the map $f$ reads as $t\mapsto(1,t,t^2,t^3)$. Given $p=[1:t]$, we consider the matrix
\[
A_p^{(2)}(f) =
\begin{pNiceMatrix}[last-col=5]
1&t&t^2&t^3&f\\[4pt]
0&1&2\,t&3\,t^2&\frac{\partial f}{\partial t}\\[4pt]
0&0&1&3\,t&\frac{1}{2}\frac{\partial^2 f}{\partial t^2}
\end{pNiceMatrix}
\]
whose rows correspond to all partial derivatives of order at most $2$ of the components of $f$. On the one hand $\T_p^2(f)=\P(\rowspan A_p^{(2)}(f))$. One verifies that $\dim \rowspan A_p^{(2)}(f)=3$, hence $\dim \T_p^2(f) = 2$. In Figure \ref{fig: osculating spaces}, we display the tangent and second-order osculating spaces of $f$ at some points.
On the other hand, the right kernel of $A_p^{(2)}(f)$ is one-dimensional and is generated by the vector $\eta=(-t^3, 3\,t^2, -3\,t, 1)^\mT$.

Consider a quadric hypersurface $Q\subseteq\P^3$ defined by a positive-definite quadratic form in $S^3\R^2$ and let $M_Q$ be the associated positive-definite symmetric matrix. Then the $2$nd-order distance locus $\DL_2(f,Q)$ is equal to the union of the normal spaces
\[
\N_p^2(f,Q)=\langle\T_p^2(f)^\perp,f(p)\rangle=\langle[M_Q\eta],f(p)\rangle
\]
for every $p\in\P^1$, in particular it is a projective surface in $\P^3$. Applying Proposition \ref{prop: kth order ED degree sum of kth order polar degrees} and Corollary \ref{corol: generic kth ED degree Veronese}, we have $\gDD_2(f) = \mu_{2,0}(f)+\mu_{2,1}(f) = 3+3 = 6$.

Firstly, let $Q_\mED$ be the standard Euclidean quadric of equation $\sum_{i=0}^3 u_i^2=0$, hence $M_{Q_\mED}$ is the identity matrix. Observe that $Q_\mED\cap f(X)$ is transversal. We verified in \verb|Macaulay2| \cite{GS} that $\DL_2(f,Q_\mED)$ is a surface of degree $6=\gDD_2(f)$ (displayed in light blue on the left-hand side of Figure \ref{fig: EDD 2nd order}) and that $\varphi_{2,2}$ is a birational morphism.
If instead we consider $Q_\mBW\colon\sum_{i=0}^3 \binom{3}{i}u_i^2=0$, corresponding to the {\em Bombieri-Weyl} inner product in $S^3\R^2$, then $Q_\mBW\cap f(X)$ is not transversal and two interesting properties hold: firstly $\DD_2(f,Q_\mBW)=4<6=\gDD_2(f)$, equivalently there is a ``$2$nd-order ED defect'' equal to $2$, a behaviour similar to the classical ED defects studied in \cite{maxim2020defect}. Secondly, the morphism $\varphi_{2,2}$ is not birational, indeed $\DL_2(f,Q_\mBW)$ is the nonsingular quadric surface of equation $u_1^2-u_0u_2+u_2^2-u_1u_3 = 0$ (displayed in yellow on the right-hand side of Figure \ref{fig: EDD 2nd order}), where $[u_0:u_1:u_2:u_3]$ are homogeneous coordinates in $\P^3$, while $\deg\varphi_{2,2}=2$. The smaller value of $\DD_2(f,Q_\mBW)$ follows because each $2$nd-order normal line $\N_p^2(f,Q_\mBW)$ meets the twisted cubic $f(X)$ at another point $f(p')$ such that $p=[v]$, $p'=[v']$, and the vectors $v,v'$ are orthogonal, implying that $\N_p^2(f,Q_\mBW)=\N_{p'}^2(f,Q_\mBW)$. We explain this graphically on the right-hand side of Figure \ref{fig: EDD 2nd order}, in particular with the points $p=[1:1]$ and $p'=[1:-1]$.
The \verb|Macaulay2| code to compute these data loci is available in \cite{github}.

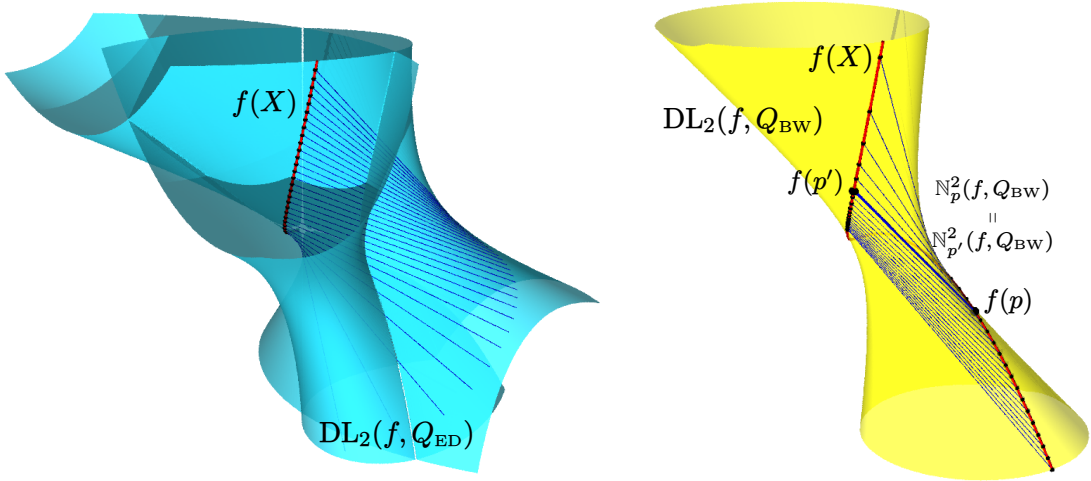
\begin{figure}[ht]
\centering
\begin{overpic}[width=0.5\textwidth]{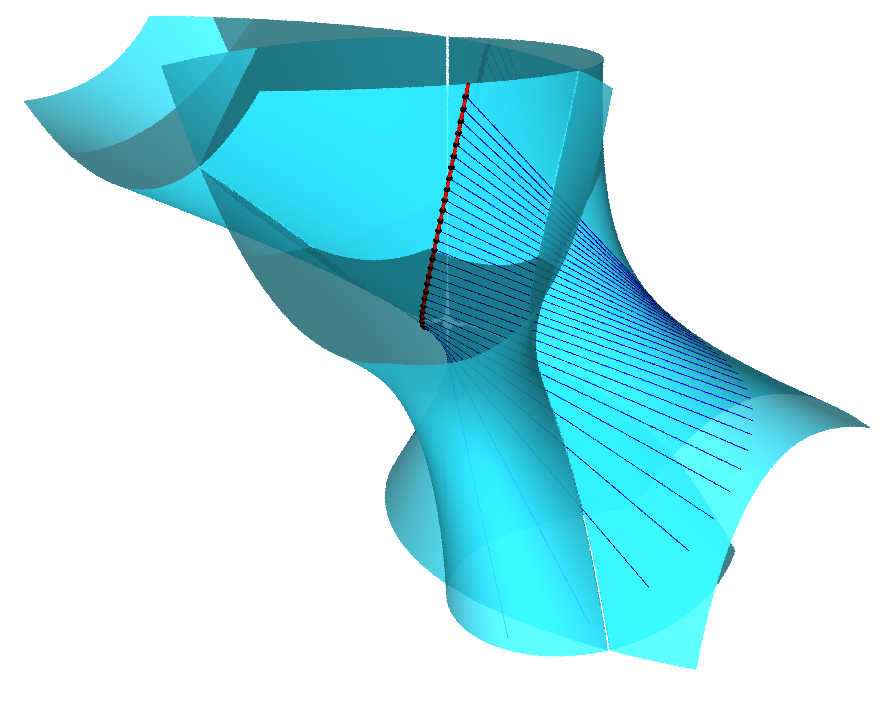}
\put (53,10) {$\DL_2(f,Q_\mED)$}
\put (39,63) {$f(X)$}
\end{overpic}
\begin{overpic}[width=0.4\textwidth]{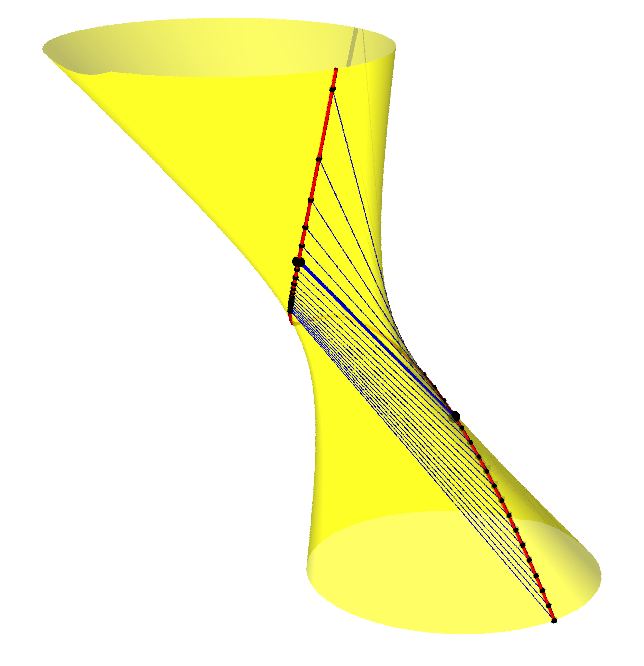}
\put (70,38) {\small{$f(p)$}}
\put (32,61) {\small{$f(p')$}}
\put (36,85) {$f(X)$}
\put (8,73) {$\DL_2(f,Q_\mBW)$}
\put (60,60) {\scriptsize{$\stackbelow{\N_p^2(f,Q_\mBW)}{\N_{p'}^2(f,Q_\mBW)}$}}
\end{overpic}
\caption{Second-order distance loci of the Veronese embedding $f=\nu_1^3\colon\P^1\hookrightarrow\P^3$, with respect to the standard (left) and Bombieri-Weyl (right) quadratic form in $S^3\R^2$, displayed in the real affine chart $\{u_0=1\}$.}\label{fig: EDD 2nd order}
\end{figure}

\section*{Acknowledgements}

We gratefully acknowledge the support from the Swedish Research Council (Vetenskapsr{\aa}det), the Verg Foundation, the Brummer \& Partners Math Data Lab, and the Knut and Alice Wallenberg Foundation. We are grateful to Ragni Piene for her valuable feedback, which has helped us improve both the accuracy and clarity of this paper.

\section{Osculating functions}\label{sec: osculating}

{\bf Notation.} Let $V$ be an $(n+1)$-dimensional vector space over $\C$, or more generally over $\K,$ an algebraically closed field of characteristic zero. We denote by $\P^n$ the projective space $\P(V)$, while $(\P^n)^\vee$ denotes the dual projective space of hyperplanes in $\P^n$. Given a multi-index $\alpha=(\alpha_1,\dots,\alpha_m)\in\N^m$, we define $|\alpha|\coloneqq \alpha_1+\cdots+\alpha_m$, and for any vector $x=(x_1,\dots,x_m)$, we call $x^\alpha=x_1^{\alpha_1}\cdots x_m^{\alpha_m}$.

Let $X$ be a nonsingular irreducible variety over $\K$ and let $f\colon X\to\P^n$ be a morphism.
It is determined by the subset $V=f^*(H^0(\P^n,\OO_{\P^n}(1))\subseteq H^0(X,\OO_X(1))$ where $\OO_X(1)\coloneqq f^*(\OO_{\P^n}(1)).$ In particular $f(p)=(\sigma_0(p),\dots,\sigma_n(p))$ for a given basis $(\sigma_0,\dots,\sigma_n)$ of $V.$ 
If $V=H^0(X,\OO_X(1))$ the morphism $f$ is said to be linearly normal.
Let $p\in X$ and $\sigma\in V$. Recall that any $\sigma\in V\subseteq H^0(X,\OO_X(1))$ can be represented in a neighborhood of $p$ as a polynomial in local coordinates $(x_1,\dots,x_m)$.
The {\em $k$th jet} of $\sigma$ at $p$ is the $\binom{m+k}{k}$-tuple
\[
j_{k,p}(\sigma)\coloneqq\left(\frac{1}{|\alpha|!}\frac{\partial^{|\alpha|}\sigma}{\partial x^\alpha}(p)\right)_{\substack{\alpha\in\N^m\\0\le|\alpha|\le k}}\in \PP^k_p(\OO_X(1))\coloneqq H^0\left(X,\OO_X(1)\otimes\frac{\OO_{X,p}}{\mathfrak{m}_{p}^{k+1}}\right)\cong \K^{\binom{m+k}{k}}\,,
\]
where $\mathfrak{m}_p$ denotes the maximal ideal at $p$. The vector space $\PP^k_p(\OO_X(1))$ is the fiber at $p$ of a vector bundle $\PP^k(\OO_X(1))$, called the {\em $k$th jet bundle} of the line bundle $\OO_X(1)$. Jet sheaves, also called sheaves of principal parts, are defined as follows. Consider the two projections $\pr_i\colon X\times X\to X$ for $i=1,2.$ The principal part sheaf of order $k$ is the coherent sheaf given by:
\[
\PP^k(\OO_X(1)) \coloneqq (\pr_1)_*\left(\pr_2^*(\OO_X(1))\otimes \frac{\OO_{X\times X}}{\II_{\Delta}^{k+1}}\right)\,,
\]
where $\Delta$ denotes the diagonal in $X\times X.$  Since the variety $X$ is nonsingular these sheaves are locally free of rank $\binom{m+k}{k}$, see \cite[\S 16]{grothendieck1967elements}.

Gluing together the maps $j_{k,p}$ yields a morphism of vector bundles on $X$:
\begin{equation}\label{eq: morphism jk}
j_k\colon V\otimes\OO_X\to \PP^k(\OO_X(1))\,.
\end{equation}

Throughout the remainder of this paper, we adopt the shorthand $\PP^k(f) \coloneqq \PP^k(\OO_X(1)).$

\begin{definition}\label{def: osculating space}
For every point $p\in X$, the {\em $k$th osculating space of $f$ at $f(p)\in f(X)$} is
\[
\T_p^k(f) \coloneqq \P(\image j_{k,p})\,.
\]
\end{definition}
If $f$ is a closed embedding, by convention we identify $X$ with its image $f(X)$ and write $p$ instead of $f(p)$. Moreover, when $f$ is a closed embedding, the first osculating space of $X$ at the point $p$ coincides with the projective tangent space of $f$ at $p$:
\[
\T_p^1(f) = \T_p(f)\,.
\]
Furthermore, osculating spaces form an ascending sequence of inclusions:
\begin{equation}\label{eq: osc seq}
    \{p\} \subseteq \T_p(f)  = \T_p^1(f) \subseteq \cdots \subseteq \T_p^k(f) \subseteq \cdots \subseteq \P^n\,.
\end{equation}

For a generic $p\in X$, the rank of $j_{k,p}$ is constant. This motivates the following definition.

\begin{definition}\label{def: globally k-osculating}
Given $f\colon X\to\P^n$, let $X_k^{\circ}\subseteq X$ be the dense open subset of $X$ of points $p$ such that $\rank j_{k,p}$ is constant. We define the {\em generic $k$-osculating dimension} of $f$ as $m_k\coloneqq\rank j_{k,p}-1$ for any $p\in X_k^{\circ}$.
We say that the morphism $f$ is {\em globally $k$-osculating} if $X_k^{\circ}=X$. In this case, $m_k$ is referred to as the {\em $k$-osculating dimension} of $f$.
\end{definition}
\begin{remark}
The generic $1$-osculating dimension of a morphism $f\colon X\to\P^n$ is $m_1=m$. Indeed, the open dense subset $X_1^\circ\subseteq X$ coincides with the nonsingular locus of $X$, and $\T_p^1(f)=\T_p(f)=\P^m$ for every $p\in X_1^\circ$. Furthermore, the morphism $f$ is globally $1$-osculating if and only if it is an immersion.
\end{remark}
    
Let $(x_1,\dots,x_m)$ be a system of local coordinates around the point $p=(0,\dots,0)$. After choosing a basis $(\sigma_0,\dots,\sigma_n)$ of $V$ one can define the $\binom{m+k}{k}\times (n+1)$ matrix
\begin{equation}\label{eq: matrix A}
A^{(k)}_p(f)\coloneqq \left(\frac{1}{|\alpha|!}\frac{\partial^{|\alpha|}\sigma_\beta}{\partial x^\alpha}(p)\right)_{\substack{\alpha\in\N^m\\\ 0\le|\alpha|\le k\\0\le\beta\le n}} =
\left(
\begin{array}{c|c|c|c}
j_{k,p}(\sigma_0)^\mT & j_{k,p}(\sigma_1)^\mT & \cdots & j_{k,p}(\sigma_n)^\mT 
\end{array}
\right)
\end{equation}
obtained by concatenating horizontally the $n+1$ column vectors $j_{k,p}(\sigma_\beta)^\mT$ of length $\binom{m+k}{k}$ for every $\beta\in\{0,\dots,k\}$.
A basis of $\image j_{k,p}$ is given by $m_k+1$ linearly independent vectors in the column span of $A^{(k)}_p(f)$.
Notice that $X_k^{\circ}=\{p\in X\mid \rank A_p^{(k)}(f)=m_k+1\}$ and that this identity is independent of the choice of a basis of $V.$

\begin{definition}\label{def: k-regular}
The morphism $f\colon X \to\P^n$ is {\em $k$-regular} if $X_k^{\circ}=X$ and if the vector bundle morphism $j_k$ in \eqref{eq: morphism jk} is surjective, equivalently $m_k+1=\binom{m+k}{k}$.
\end{definition}

\begin{proposition}
Let $f\colon X\to\P^n$ be a globally $k$-osculating morphism for some $k\ge 1$. Then it is $\ell$-osculating for all $\ell\le k.$
Similarly, if $f$ is $k$-regular, then it is $\ell$-regular for all $\ell\le k$.
\end{proposition}
\begin{proof}
Assume that $f\colon X\to\P^n$ is a globally $k$-osculating morphism for some $k\ge 1$, and let $\ell\le k$. Notice that   
\[
A_p^{(k)}(f) = 
\begin{pmatrix}
A_p^{(\ell)}(f) \\ B_p(f)
\end{pmatrix}
\]
where $B_p(f)$ is an $\left(\binom{m+k}{k}-\binom{m+\ell}{\ell}\right)\times (n+1)$ matrix for every $p\in X$. Assuming that there exists a special point $p$ where $A_p^{(\ell)}$ does not have the maximal generic rank while $A_p^{(k)}$ has the maximal generic rank is impossible.
Moreover, if $A_p^{(k)}(f)$ has maximal rank, then $A_p^{(\ell)}(f)$ will also have maximal rank.
\end{proof}

\begin{remark}
A globally $1$-osculating morphism is always $1$-regular. For $k\ge 2$, this is no longer the case. See Example \ref{ex: blowup P1xP1 one point} below.
\end{remark}

\subsection{A gallery of toric examples}\label{subsec: gallery of toric examples}

For varieties endowed with a maximal dimensional torus action, i.e., toric varieties, the torus action leads to remarkably simple criteria dictated by the size of the associated polytope.

Let $f\colon X\to\P^n$ be a nonsingular, globally $k$-osculating, torus-equivariant embedding of a toric variety $X$ of dimension $m$. 
Then $f$ is a monomial map which over the algebraic torus is given by $x\mapsto (x^{a_0},\dots,x^{a_n})$ for a subset $A=\{a_0,\dots,a_n\}\subseteq\Z^m$.
Recall that $P(f)=\conv(A)\subset\R^m$ is a full-dimensional polytope whose vertices are in one-to-one correspondence with the fixed points by the torus action. In particular the variety $X$ has an affine cover $X=\bigcup_{v\in\vertices P(f)} U_v$ where $U_v\subseteq\C^m$ is a Zariski open subset containing $p_v$ for every $v\in\vertices P(f)$. Since $X$ is nonsingular, for every vertex $v$ the first lattice points along the edges containing $v$ form a lattice basis $\BB_v$ for which $(v)_{\BB_v}=0$. Moreover there is an associated affine patch $U_v\ni p_v$ and local coordinates $(x_1,\dots,x_m)$ for which $p_v=0$ and 
\[
f|_{U_v}\colon x\mapsto (x^{(a_0)_{\BB_v}},\cdots,x^{(a_n)_{\BB_v}})
\]
Let $p=\bone$ be the generic point in the torus.
For every $k\ge 1$, we denote by $A^{(k)}$ the matrix $A_{\bone}^{(k)}(f)$ defined in \eqref{eq: matrix A}. We note that the construction of these matrices is also explained in \cite[Section 2.3]{dickenstein2024interpolation}.

The matrix $A^{(k)}$ determines the generic $k$-osculating rank. More generally, for every fixed point $p_v\in X$, we denote by $A_v^{(k)}$ the matrix $A_{p_v}^{(k)}(f)$. The following result characterizes globally $k$-osculating toric embeddings. 

\begin{proposition}\label{prop: characterization of globally osculating toric varieties}
A nonsingular toric embedding $f$ is globally $k$-osculating of osculating dimension $m_k$ if and only if $\dim\T_{p_v}^k(f)=\rank A_v^{(k)}-1=\rank A^{(k)}-1$ for every fixed point $p_v$.
\end{proposition}
\begin{proof}
From the description above, it is clear that $m_k+1=\rank A^{(k)}.$ Note that if the locus of points where the $k$th osculating space has a lower dimension than $m_k$ is nonempty, the torus action ensures that it contains fixed points. It follows that to ensure global $k$th osculation is enough that the $k$th osculation dimension at each fixed point is $m_k.$
\end{proof}  

\begin{example}\label{ex: blowup P2 three points}
Consider $\pi\colon X\to\P^2,$ the blow-up of $\P^2$ at the three points and the embedding $f_1\colon X\to \P^6$, a Del Pezzo surface of degree six defined by the global sections of the line bundle $\OO_X(1)=\pi^*\OO_{\P^2}(3)-E_1-E_2-E_3$, where $E_i$ are the exceptional divisors. On an affine patch around the fixed point $p_v$, where $v$ is indicated in Figure \ref{fig: blowup P2 three points}, the monomial map is:
\[
(x_1,x_2) \mapsto (1,x_1,x_2,x_1x_2,x_1x_2^2,x_1^2x_2,x_1^2x_2^2)\,.
\]
Consider $k=2$. In this case $\rank A^{(2)}=6$, and thus the generic $2$-osculating dimension of $f_1$ is $m_2=5$, but $f_1$ is not globally $2$-osculating since $\dim\T_{p_v}^2(f_1)=3$ for every fixed point $p_v\in X$.
\hfill$\diamondsuit$

\begin{figure}
\centering
\begin{tikzpicture}[scale=1]
\draw[line width=0.5pt, dashed] (-0.5,0) -- (2.5,0);
\draw[line width=0.5pt, dashed] (-0.5,1) -- (2.5,1);
\draw[line width=0.5pt, dashed] (-0.5,2) -- (2.5,2);
\draw[line width=0.5pt, dashed] (0,-0.5) -- (0,2.5);
\draw[line width=0.5pt, dashed] (1,-0.5) -- (1,2.5);
\draw[line width=0.5pt, dashed] (2,-0.5) -- (2,2.5);
\fill[cyan!20,nearly transparent] (1,0) -- (2,0) -- (2,1) -- (1,2) -- (1,2) -- (0,2) -- (0,1);
\draw[line width=1pt] (1,0) -- (2,0);
\draw[line width=1pt] (2,1) -- (2,0);
\draw[line width=1pt] (2,1) -- (1,2);
\draw[line width=1pt] (1,2) -- (0,2);
\draw[line width=1pt] (0,2) -- (0,1);
\draw[line width=1pt] (0,1) -- (1,0);
\draw[line width=0.5pt, red, ->] (0,2) -- (0,0);
\draw[line width=0.5pt, red, ->] (0,2) -- (2,2);
\node at (-0.7,2.2) {\scriptsize{\textcolor{red}{$v=(0,0)$}}};
\node at (-0.4,1.2) {\scriptsize{$(1,0)$}};
\node at (0.6,2.2) {\scriptsize{$(0,1)$}};
\node at (0.6,1.2) {\scriptsize{$(1,1)$}};
\node at (0.6,0.2) {\scriptsize{$(2,1)$}};
\node at (1.6,1.2) {\scriptsize{$(1,2)$}};
\node at (1.6,0.2) {\scriptsize{$(2,2)$}};
\fill[red] (0,2) circle (2pt);
\fill (0,1) circle (2pt);
\fill (1,0) circle (2pt);
\fill (2,0) circle (2pt);
\fill (2,1) circle (2pt);
\fill (1,1) circle (2pt);
\fill (1,2) circle (2pt);
\end{tikzpicture}
\caption{The polygon and the vectors of exponents of the monomial map representing the embedding $f_1$ of Example \ref{ex: blowup P2 three points}.}\label{fig: blowup P2 three points}
\end{figure}
\end{example}

\begin{example}\label{ex: blowup P1xP1 one point}
Consider $\pi\colon X\to\P^1\times\P^1,$ the blow-up of $\P^1\times\P^1$ at one point and the embedding $f_2\colon X\to \P^{14}$, a surface of degree $17$ defined by the global sections of the line bundle $\OO_X(1)=\pi^*\OO_{\P^1\times\P^1}(3,3)-E$, where $E$ is the exceptional divisor. On the affine patches around the fixed points $p_v$ and $p_w$, where $v$ and $w$ are indicated in Figure \ref{fig: blowup P1xP1 one point}, the monomial maps are respectively
\begin{align*}
(x_1,x_2) &\mapsto (1,x_1,x_1^2,x_1^3,x_2,x_1x_2,x_1^2x_2,x_1^3x_2,x_2^2,x_1x_2^2,x_1^2x_2^2,x_1^3x_2^2,x_2^3,x_1x_2^3,x_1^2x_2^3)\\
(y_1,y_2) &\mapsto (1,y_1,y_1^2,y_2,y_1y_2,y_1^2y_2,y_1^3y_2,y_1y_2^2,y_1^2y_2^2,y_1^3y_2^2,y_1^4y_2^2,y_1^2y_2^3,y_1^3y_2^3,y_1^4y_2^3,y_1^5y_2^3)\,.
\end{align*}
Consider $k=2$. One verifies that $\rank A_v^{(2)}=6$ and $\rank A_w^{(2)}=5$, therefore the embedding $f_2$ is not globally $2$-osculating.\hfill$\diamondsuit$

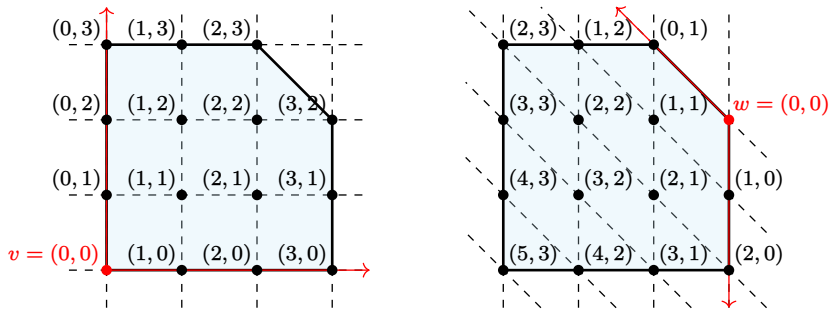
\begin{figure}
\centering
\begin{tikzpicture}[scale=1]
\draw[line width=0.5pt, dashed] (-0.5,0) -- (3.5,0);
\draw[line width=0.5pt, dashed] (-0.5,1) -- (3.5,1);
\draw[line width=0.5pt, dashed] (-0.5,2) -- (3.5,2);
\draw[line width=0.5pt, dashed] (-0.5,3) -- (3.5,3);
\draw[line width=0.5pt, dashed] (0,-0.5) -- (0,3.5);
\draw[line width=0.5pt, dashed] (1,-0.5) -- (1,3.5);
\draw[line width=0.5pt, dashed] (2,-0.5) -- (2,3.5);
\draw[line width=0.5pt, dashed] (3,-0.5) -- (3,3.5);
\fill[cyan!20, nearly transparent] (0,0) -- (3,0) -- (3,2) -- (2,3) -- (0,3) -- (0,0);
\draw[line width=1pt] (0,0) -- (3,0);
\draw[line width=1pt] (3,0) -- (3,2);
\draw[line width=1pt] (3,2) -- (2,3);
\draw[line width=1pt] (2,3) -- (0,3);
\draw[line width=1pt] (0,3) -- (0,0);
\draw[line width=0.5pt, red, ->] (0,0) -- (3.5,0);
\draw[line width=0.5pt, red, ->] (0,0) -- (0,3.5);
\node at (2.6,2.2) {\scriptsize{$(3,2)$}};
\node at (2.6,1.2) {\scriptsize{$(3,1)$}};
\node at (2.6,0.2) {\scriptsize{$(3,0)$}};
\node at (1.6,3.2) {\scriptsize{$(2,3)$}};
\node at (1.6,2.2) {\scriptsize{$(2,2)$}};
\node at (1.6,1.2) {\scriptsize{$(2,1)$}};
\node at (1.6,0.2) {\scriptsize{$(2,0)$}};
\node at (0.6,3.2) {\scriptsize{$(1,3)$}};
\node at (0.6,2.2) {\scriptsize{$(1,2)$}};
\node at (0.6,1.2) {\scriptsize{$(1,1)$}};
\node at (0.6,0.2) {\scriptsize{$(1,0)$}};
\node at (-0.4,3.2) {\scriptsize{$(0,3)$}};
\node at (-0.4,2.2) {\scriptsize{$(0,2)$}};
\node at (-0.4,1.2) {\scriptsize{$(0,1)$}};
\node at (-0.7,0.2) {\scriptsize{\textcolor{red}{$v=(0,0)$}}};
\fill (3,2) circle (2pt);
\fill (3,1) circle (2pt);
\fill (3,0) circle (2pt);
\fill (2,0) circle (2pt);
\fill (2,1) circle (2pt);
\fill (2,2) circle (2pt);
\fill (2,3) circle (2pt);
\fill (1,0) circle (2pt);
\fill (1,1) circle (2pt);
\fill (1,2) circle (2pt);
\fill (1,3) circle (2pt);
\fill[red] (0,0) circle (2pt);
\fill (0,1) circle (2pt);
\fill (0,2) circle (2pt);
\fill (0,3) circle (2pt);
\end{tikzpicture}
\hspace{1cm}
\begin{tikzpicture}[scale=1]
\draw[line width=0.5pt, dashed] (3.5,1.5) -- (1.5,3.5);
\draw[line width=0.5pt, dashed] (3.5,0.5) -- (0.5,3.5);
\draw[line width=0.5pt, dashed] (3.5,-0.5) -- (-0.5,3.5);
\draw[line width=0.5pt, dashed] (2.5,-0.5) -- (-0.5,2.5);
\draw[line width=0.5pt, dashed] (1.5,-0.5) -- (-0.5,1.5);
\draw[line width=0.5pt, dashed] (0.5,-0.5) -- (-0.5,0.5);
\draw[line width=0.5pt, dashed] (0,-0.5) -- (0,3.5);
\draw[line width=0.5pt, dashed] (1,-0.5) -- (1,3.5);
\draw[line width=0.5pt, dashed] (2,-0.5) -- (2,3.5);
\draw[line width=0.5pt, dashed] (3,-0.5) -- (3,3.5);
\fill[cyan!20, nearly transparent] (0,0) -- (3,0) -- (3,2) -- (2,3) -- (0,3) -- (0,0);
\draw[line width=1pt] (0,0) -- (3,0);
\draw[line width=1pt] (3,0) -- (3,2);
\draw[line width=1pt] (3,2) -- (2,3);
\draw[line width=1pt] (2,3) -- (0,3);
\draw[line width=1pt] (0,3) -- (0,0);
\draw[line width=0.5pt, red, ->] (3,2) -- (1.5,3.5);
\draw[line width=0.5pt, red, ->] (3,2) -- (3,-0.5);
\node at (3.7,2.2) {\scriptsize{\textcolor{red}{$w=(0,0)$}}};
\node at (3.4,1.2) {\scriptsize{$(1,0)$}};
\node at (3.4,0.2) {\scriptsize{$(2,0)$}};
\node at (2.4,3.2) {\scriptsize{$(0,1)$}};
\node at (2.4,2.2) {\scriptsize{$(1,1)$}};
\node at (2.4,1.2) {\scriptsize{$(2,1)$}};
\node at (2.4,0.2) {\scriptsize{$(3,1)$}};
\node at (1.4,3.2) {\scriptsize{$(1,2)$}};
\node at (1.4,2.2) {\scriptsize{$(2,2)$}};
\node at (1.4,1.2) {\scriptsize{$(3,2)$}};
\node at (1.4,0.2) {\scriptsize{$(4,2)$}};
\node at (0.4,3.2) {\scriptsize{$(2,3)$}};
\node at (0.4,2.2) {\scriptsize{$(3,3)$}};
\node at (0.4,1.2) {\scriptsize{$(4,3)$}};
\node at (0.4,0.2) {\scriptsize{$(5,3)$}};
\fill[red] (3,2) circle (2pt);
\fill (3,1) circle (2pt);
\fill (3,0) circle (2pt);
\fill (2,0) circle (2pt);
\fill (2,1) circle (2pt);
\fill (2,2) circle (2pt);
\fill (2,3) circle (2pt);
\fill (1,0) circle (2pt);
\fill (1,1) circle (2pt);
\fill (1,2) circle (2pt);
\fill (1,3) circle (2pt);
\fill (0,0) circle (2pt);
\fill (0,1) circle (2pt);
\fill (0,2) circle (2pt);
\fill (0,3) circle (2pt);
\end{tikzpicture}
\caption{The polygon and the vectors of exponents of the monomial maps representing the embedding $f_2$ of Example \ref{ex: blowup P1xP1 one point}.}\label{fig: blowup P1xP1 one point}
\end{figure}
\end{example}

A special class of toric varieties that are always globally $k$-osculating consists of the Segre-Veronese embeddings of a product of projective spaces. Consider two $r$-tuples of positive integers $\bm=(m_1,\dots,m_r)$ and $\bd=(d_1,\dots,d_r)$, with $m_1\le\cdots\le m_r$. Define $d\coloneqq d_1+\cdots+d_r$. Let $V_1,\dots,V_r$ be vector spaces of dimensions $m_1+1\le\cdots\le m_r+1$, and consider the product $\P^\bm=\prod_{i=1}^r\P(V_i)$. The line bundle $\OO_{\P^\bm}(\bd)$ induces an embedding
\begin{equation}\label{eq: def Segre-Veronese embedding}
    \sigma\nu_\bm^\bd\colon\P^\bn\hookrightarrow\P\left(\bigotimes_{i=1}^r\mathrm{Sym}^{d_i}V_i\right)\,,
\end{equation}
whose image $SV_\bn^\bd=\sigma\nu_\bm^\bd(\P^\bn)$ is the {\em Segre-Veronese variety}. 
When $r=1$, we use the notation $\bm=m$, $\bd=d$, $\nu_m^d\coloneqq\sigma\nu_m^d$, and $V_m^d\coloneqq SV_m^d$ is the {\em Veronese variety}.

A complete description of the osculating spaces of Segre-Veronese varieties is given in \cite[Proposition 2.5]{araujo2019defectivity}.
In particular, it is shown that, for all $\bn$ and $\bd$, and for all integer $k\ge 1$, the Segre-Veronese embedding $\sigma\nu_\bm^\bd$ is globally $k$-osculating of $k$-osculating dimension
\begin{equation}\label{eq: osculating dimension SV}
    m_k = 
    \begin{cases}
        \sum_{i=1}^k\sum_{|\bs|=i\,,\,\bs\le\bd}\binom{n_1+s_1-1}{s_1}\cdots\binom{n_r+s_r-1}{s_r} & \text{if $1\le k\le |\bd|$}\\
        \prod_{\ell=1}^r\binom{n_\ell+d_\ell}{d_\ell}-1 & \text{if $k>|\bd|$.}
    \end{cases}
\end{equation}

The following is an immediate corollary of Proposition \ref{prop: characterization of globally osculating toric varieties} and also of \cite[Theorem 4.2, Proposition 4.5]{dirocco1999generation}. It characterizes the $k$-regularity of Segre-Veronese embeddings. For additional references, see \cite{massarenti2019defectivity,catalisano2002secant,bernardi2007osculating,ballico2003secant}.
\begin{corollary}\label{corol: when Segre-Veronese k-regular}
Let $\bm$ and $\bd$ be two $r$-tuples of positive integers. The Segre-Veronese embedding $\sigma\nu_\bm^\bd$ in \eqref{eq: def Segre-Veronese embedding} is $k$-regular if and only if $k\le\min\bd$. 
\end{corollary}

\section{Higher-order polar classes}\label{sec: polar}

The definition of polar classes is classical in Algebraic Geometry, going back to Severi \cite{severi1902intersezioni} and Todd \cite{todd1937arithmetical}.
A generalization to higher-order polar classes has been recently given by Piene in \cite{piene2022higher}. We follow her definitions and results, adapting them to our setting.

\begin{definition}{\cite[Section 3]{piene2022higher}}\label{def: kth order polar classes}
Let $f\colon X\to\P^n$ be a morphism of generic $k$-osculating dimension $m_k$ and let $0\le i\le \dim X=m.$ The {\em $i$th polar class of $f$ of order $k$} is the class $p_{k,i}(f)=[P_{k,i}(f,L)]$ where
\begin{equation}\label{eq: def higher-order polar loci}
\text{$P_{k,i}(f,L)\coloneqq\overline{\{x\in X_k^{\circ}\mid\dim(\T_p^k(f)\cap L)>m_k-\codim(L)\}}$ for all $i\in\{0,\dots,\dim X=m\}$}\end{equation}
and $L\subseteq\P^n$ is any linear subspace of dimension $n-m_k+i-2.$
\end{definition}

\begin{definition}{\cite[\S 2]{Higherdual}}\label{def: higher conormal variety}
Let $p\in X_k^{\circ}$ and $H\in(\P^N)^\vee.$ Then $H$ is {\em $k$-osculating at $f(p)$} if  $H\supset \T_p^k(f).$ The $k$th {\em conormal variety} of $f\colon X\to\P^n$ is the Zariski closure $W_k(f)\coloneqq\overline{W_k^{\circ}(f)}$, where:
\begin{equation}
    W_k^{\circ}(f) \coloneqq \{(x,H)\in\P^n\times(\P^n)^\vee\mid\text{$x=f(p)$ for some $p\in X_k^{\circ}$ and $H$ is $k$-osculating at $x$}\}\,.
\end{equation}
\end{definition}

When a morphism is globally $k$-osculating, the treatment of polar varieties can be conveniently related to characteristic classes of vector bundles.

With notation as above, assume that the morphism $f\colon X\to\P^n$ is globally $k$-osculating. Then the image $\image j_k$ and the kernel $\ker j_k$ of the morphism $j_k$ in \eqref{eq: morphism jk} are also vector bundles over $X$ of ranks $m_k+1$ and $n-m_k$ respectively.
We then obtain the following short exact sequence of vector bundles over $X$:
\begin{equation}\label{eq: k-jet exact sequence}
0\to \ker j_k \hookrightarrow V\otimes\OO_X \twoheadrightarrow \image j_k \to 0\,.
\end{equation}

Moreover, higher-order polar loci are closely related to higher-order conormal and dual varieties, which have been introduced in the earlier work by Piene \cite{Higherdual}.

\begin{remark}
Since $\dim\T_p^k(f)=m_k$ for every $p\in X_k^{\circ}$, we have
\begin{equation}\label{eq: dim conormal}
\dim W_k(f)=n-1+m-m_k\,.
\end{equation}
Let the morphism $f$ be globally $k$-osculating. Then, using \eqref{eq: morphism jk},
\[
W_k^{\circ}(f)=W_k(f)=\P((\ker j_k)^\vee)\subseteq X\times(\P^n)^\vee\,.
\]
In particular, if $k=1$, then $W(f)=W_1(f)$ is the classical conormal variety of dimension $n-1$ in $\P^n\times(\P^n)^\vee$.
\end{remark}

Let $\pr_1$ and $\pr_2$ be the projections of $\P^n\times(\P^n)^\vee$ to the factors $\P^n$ and $(\P^n)^\vee$, respectively.
\begin{definition}{\cite[\S 2]{Higherdual}}\label{def: dual variety X}
The {\em dual variety of order $k$} of $f\colon X\to\P^N$ is $X_k^\vee\coloneqq \pr_2(W_k(f))$.
\end{definition}

Consider the surjective morphisms $\pi_{1,k}\colon W_k(f)\to X$ and $\pi_{2,k}\colon W_k(f)\to X_k^\vee$ induced by the projections $\pr_1$ and $\pr_2$. More precisely $\pi_{i,k}=\pr_i\circ \iota_k$ for $i=1,2$, where $\iota_k\colon W_k(f)\hookrightarrow\P^n\times(\P^n)^\vee$ is the inclusion.
Given $H\in X_k^\vee$, the {\em contact locus of $H$ of order $k$} is
\begin{equation}\label{eq: def contact locus order k}
    \cont_k(H,X)\coloneqq\pi_{1,k}(\pi_{2,k}^{-1}(H))=\overline{\{p\in X_k^{\circ}\mid H\supset \T_p^k(f)\}}\,.
\end{equation}
For a generic $H\in X_k^\vee$, the dimension of $\cont_k(H,X)$
is equal to the dimension of a generic fiber of $\pi_{2,k}$. From \eqref{eq: dim conormal} it follows that 
\[
\text{$\dim X_k^\vee=n-(m_k-m+1)-\dim\cont_k(H,X)$ for a generic hyperplane $H\in X_k^\vee$.}
\]
Generally, one expects that $\dim X_k^\vee=n-(m_k-m+1)$, or equivalently, that the generic contact locus is zero-dimensional.

\begin{definition}\label{def: kth dual defect}
The {\em dual defect of order $k$} of the  morphism $f\colon X\to\P^n$ is defined as
\begin{equation}
\text{$\defect_k(f)\coloneqq\dim\cont_k(H,X)$ for a generic $H\in(\P^n)^\vee$.}    
\end{equation}
We say that $X$ is {\em dual defective of order $k$} if $\defect_k(f)>0$. Otherwise, $X$ is {\em dual nondefective of order $k$}.
\end{definition}

\begin{example}\label{ex: multidegrees blowup P2 three points}
Consider the embedding $f_1\colon X\to\P^6$ in Example \ref{ex: blowup P2 three points}. The conormal variety $W(f_1)$ has dimension $n-1=5$. We verified that the dual variety $X^\vee\subseteq(\P^6)^\vee$ is a hypersurface of degree $12$, in particular $\defect(f_1)=\defect_1(f_1)=0$. For $k=2$, we already observed in Example \ref{ex: blowup P2 three points} that $m_2=\rank A^{(2)}-1=5$ generically, hence the identity \eqref{eq: dim conormal} gives $\dim W_2(f_1)=6-1+2-5=2$. Furthermore, we verified that $\dim X_2^\vee=2$ and $\deg X_2^\vee=6$, in particular $\defect_2(f_1)=0.$ \hfill$\diamondsuit$
\end{example}

The defect $\defect_k(f)$ is governed by nonvanishing higher-order polar classes. 
The following results are proven in \cite[Theorem 3.5]{piene2022higher}. We give a self-contained argument in the simpler case of globally osculating morphisms on nonsingular varieties.
\begin{proposition}{\cite[Theorem 3.5]{piene2022higher}}
\label{prop: compare k-order defect and number of vanishing higher polar degrees}
Let $f\colon X\to\P^n$ be a globally $k$-osculating morphism. The following holds.
\begin{enumerate}
    \item $c_i(\image j_k)=[P_{k,i}(f,L)]$ for a generic subspace $L\subseteq\P^n$ of dimension $n-m_k+i-2$.
    \item $c_i(\image j_k)=0$ for $i>m-\defect_k(f)$.
    \item $c_{m-\defect_k(f)}(\image j_k)\neq 0$ and $\deg X_k^\vee=\deg(c_{m-\defect_k(f)}(\image j_k))$.
\end{enumerate}
\end{proposition}
\begin{proof}
Recall that $\image j_k$ is globally generated by the sections of $V\otimes \OO_X$ via the vector bundle map $j_k$, and $\rank\image j_k=m_k+1$. Using \cite[\S III.3, pp. 411--414]{griffiths1978principles}, for every integer $i\ge 0$, the $i$th Chern class $c_i(\image j_k)$ is the rational class $[D(j_k(\sigma_1),\dots,j_k(\sigma_{m_k-i+2}))]$ of the degeneracy locus
\[
   D(j_k(\sigma_1),\dots,j_k(\sigma_{m_k-i+2})) \coloneqq \{p\in X\mid j_k(\sigma_1)(p)\wedge\cdots\wedge j_k(\sigma_{m_k-i+2})(p)=0\}
\]
for the generic global sections $j_k(\sigma_1),\dots,j_k(\sigma_{m_k-i+2})\in H^0(X,\image j_k).$
This is equivalent to requiring that $\dim(\T_p^k(f)\cap L)>i-2=m_k-\codim(L)$ for a generic linear space $L$ of codimension $m_k-i+2$.

Consider the $k$th dual variety $X_k^\vee$ introduced in Definition \ref{def: dual variety X}, in particular
\[
X_k^\vee=\overline{\{H\in(\P^n)^\vee \mid \text{$\T_p^k(f)\subseteq H$ for some $p\in X_k^{\circ}$}\}}\,.
\]
Note that a generic linear subspace $L$ of dimension $\codim(X_k^\vee)-1=m_k-m+\defect_k(f)$ does not meet $X_k^\vee$ and hence $j_k(\sigma)\neq 0$ for a generic linear combination $\sigma=\sum_{i=1}^{m_k-m+\defect_k(f)+1}\lambda_i\sigma_i.$ 
From the argument above it follows that $p_{k,m-\defect_k(f)+1}(f)=c_{m-\defect_k(f)+1}(\image j_k)=0.$ Similarly one sees that $p_{k,m-\defect_k(f)}(f)=c_{m-\defect_k(f)}(\image j_k)\neq 0$ and that $\deg(c_{m-\defect_k(f)}(\image j_k))=\deg X_k^\vee$.
\end{proof}

Let $H$ and $H'$ be hyperplanes in $\P^n$ and $(\P^n)^\vee$, respectively, and let $[H]$ and $[H']$ be the corresponding generating classes in the Chow rings $A^*(\P^n)$ and $A^*((\P^n)^\vee)$. We use the shorthand
\begin{equation}\label{eq: notation hyperplane classes}
h=\pr_1^*([H])=[H\times(\P^n)^\vee]\,,\quad h'=\pr_2^*([H'])=[\P^n\times H']\,.
\end{equation}
Recalling \eqref{eq: dim conormal}, the class $[W_k(f)]\in A^*(\P^n\times(\P^n)^\vee)$ can be written as
\begin{equation}\label{eq: def multidegrees conormal}
[W_k(f)] =  \sum_{i=0}^{n-m_k+m-1}\delta_{k,i}(f)h^{n-i}(h')^{m_k-m+1+i}\,.
\end{equation}
The numbers $\delta_{k,i}(f)$ are the multidegrees of $W_k(f)$.
Given $0\le i\le n-m_k+m-1$, let $L_1\subseteq\P^n$ and $L_2\subset(\P^n)^\vee$ be generic subspaces of dimensions $n-i$ and $m_k-m+1+i$, respectively. In particular $[L_1]=[H]^i$ and $[L_2]=[H']^{n-m_k+m-1-i}$. Then $\delta_{k,i}(f)$ is the cardinality of the finite intersection $W_k(f)\cap(L_1\times L_2)$, or equivalently
\begin{equation}\label{eq: interpretation multidegrees delta_i}
    \delta_{k,i}(f)=\int \pr_1^*([L_1])\cdot \pr_2^*([L_2])\cdot[W_k(f)]=\int h^i(h')^{n-m_k+m-1-i}[W_k(f)]\,.
\end{equation}

For every $0\le j\le m$, we define the {\em $k$th-order polar degree of $f$} as the degree $\mu_{k,j}(f)$ of the $j$th polar class or order $k$, $p_{k,j}(f),$ which we introduced in Definition \ref{def: kth order polar classes}.
We can now prove a higher-order version of the celebrated formula by Kleiman in \cite[Prop. (3), p. 187]{kleiman1986tangency}. In particular, we show a correspondence between the $k$th-order polar degrees $\mu_{k,j}(f)$ and the multidegrees $\delta_{k,j}(f)$ of the rational equivalence class $[W_k(f)]$. 

\begin{proposition}\label{prop: properties kth polar degrees}
Consider a morphism $f\colon X\to\P^n,$ that is globally $k$-osculating of $k$-osculating dimension $m_k<n$. For all $0\le j\le \dim X=m$, we have $\mu_{k,j}(f)=\delta_{k,m-j}(f)$.
\end{proposition}
\begin{proof}
Since the morphism $f$ is globally $k$-osculating of $k$-osculating dimension $m_k$, in particular $X$ is nonsingular and the variety $W_k(f)$ has the structure of a projective bundle over $X$, hence it is also a nonsingular variety.
Consider the surjective morphisms $\pi_{k,1}\colon W_k(f)\to X$, $\pi_{k,2}\colon W_k(f)\to X_k^\vee$. Let $L\subseteq\P^n$ be a generic projective subspace of dimension $n-m_k+j-2$. Then $L_k^\vee=L^\vee$ and $\dim L_k^\vee=n-1-\dim L=m_k-j+1$. In particular $[L_k^\vee]=[H']^{n-1-m_k+j}$, using the notation in \eqref{eq: notation hyperplane classes}.
Following \cite[Proposition 3.4]{piene2022higher} and following a similar approach in \cite[Lemma 3.24]{dirocco2024Relative}, one can show that
\begin{equation}\label{eq: property polar locus}
    P_{k,j}(f)=\pi_{k,1}(\pi_{k,2}^{-1}(L_k^\vee))\,.
\end{equation}
Let $\iota_k\colon W_k(f)\hookrightarrow\P^n\times(\P^n)^\vee$ be the inclusion. Then
\begin{align*}
\mu_{k,j}(f) &= \deg P_{k,j}(f) = \int [H]^{m-j}\cdot[P_{k,j}(f)]\\
&\stackrel{\mathclap{(*)}}{=} \int [H]^{m-j}\cdot[\pi_{1,k}(\pi_{2,k}^{-1}(L_k^\vee))] = \int [H]^{m-j}\cdot \pi_{1,k*}(\pi_{2,k}^*([L_k^\vee]))\\
&= \int [H]^{m-j}\cdot \pi_{1,k*}(\pi_{2,k}^*([H']^{n-1-m_k+j}))\stackrel{\mathclap{(**)}}{=} \int \pi_{1,k*}(\pi_{1,k}^*([H]^{m-j})\cdot\pi_{2,k}^*([H']^{n-1-m_k+j}))\\
&= \int \pi_{1,k*}(\pi_{1,k}^*([H])^{m-j}\cdot\pi_{2,k}^*([H'])^{n-1-m_k+j}) = \int \pi_{1,k*}(\iota_k^*(h^{m-j})\cdot \iota_k^*((h')^{n-1-m_k+j}))\\
&= \int \pi_{1,k*}(\iota_k^*(h^{m-j}\cdot (h')^{n-1-m_k+j})) = \int \pi_{1,k*}(h^{m-j}\cdot (h')^{n-1-m_k+j}\cdot[W_k(f)])\\
&\stackrel{\mathclap{(***)}}{=} \int h^{m-j}\cdot (h')^{n-1-m_k+j}\cdot[W_k(f)] = \delta_{k,m-j}(f)\,.
\end{align*}
In $(*)$, we applied \eqref{eq: property polar locus}. In $(**)$, we applied the general projection formula in \cite[Theorem 3.2]{fulton1998intersection}. In $(***)$, we used the fact that $h^{m-j}(h')^{n-1-m_k+j}[W_k(f)]$ is a zero-dimensional cycle, so the degree is preserved under projection.
\end{proof}

\section{Higher-order distance loci}\label{sec: higher-order normal bundles}
In this section, we introduce a {\em higher-order distance degree}, \cite{DHOST}. Furthermore, we define the {\em higher-order distance loci}. These concepts and their geometry strongly depend on the choice of an underlying metric. Our ambient space is an $(n+1)$-dimensional vector space $V$, and $x_0,\dots,x_n$ are homogeneous coordinates in $\P(V)\cong\P^n$. To introduce a notion of ``normality'', we fix a nonsingular quadric hypersurface $Q=\V(q)\subseteq\P^n$, where $q\in\mathrm{Sym}^2V^*$. The hypersurface $Q$ is the {\em isotropic quadric} associated with $q$. A standard choice is to consider $q(x) = q_\mED(x) = \sum_{i=0}^n x_i^2,$ but there is no need to restrict to the standard choice.
The quadric $Q$ induces the {\em polarity} or {\em reciprocity} map
\begin{equation}\label{eq: polarity}
\begin{matrix}
    \partial_q\colon & \P^n & \longrightarrow & (\P^n)^\vee\\
    & p=[p_0:\cdots:p_n] & \longmapsto & \left[\frac{\partial q}{\partial x_0}(p):\cdots:\frac{\partial q}{\partial x_n}(p)\right]\,.
\end{matrix}
\end{equation}

Let $L\subseteq\P^n$ be a subspace. We define
\begin{equation}\label{eq: perp}
    L^\perp \coloneqq \partial_q(L)^\vee = \left\{y\in\P^n\ \bigg|\ \text{$\sum_{i=0}^n\frac{\partial q}{\partial x_i}(p)y_i=0$ for every $p\in L$}\right\}\,.
\end{equation}
Observe that if $\dim L=r$, then $\dim L^\perp=n-r-1$. In particular, if $p=[p_0:\cdots:p_n]\in\P^n$, then $p^\perp$ is the polar hyperplane $H_p\subseteq\P^n$ of equation $\sum_{i=0}^n\frac{\partial q}{\partial x_i}(p)x_i=0$. Furthermore $(\P^n)^\perp=\emptyset$.

\begin{definition}\cite{piene2022higher}\label{def: kth order normal space}
Consider a morphism $f\colon X\to\P^n$, a nonsingular quadric hypersurface $Q\subseteq\P^n$, and let 
$p\in X_k^{\circ}.$ The {\em $k$th-order normal space of $(f,Q)$ at $p$} is 
\[
    \N_p^k(f,Q) \coloneqq \langle(\T_p^k(f))^\perp,f(p)\rangle \subseteq \P^n\,.
\]
When $k=1$, we call it the {\em normal space of $(f,Q)$ at $p$} and we denote it by $\N_p(f,Q)$.
\end{definition}

In particular, if $f\colon X\to\P^n$ is globally $1$-osculating, then the normal space of $(f,Q)$ at $p$ is $\N_p(f,Q) = \langle\T_p(f)^\perp,p\rangle$.

\begin{remark}
Given a point $p\in\P^n\setminus Q$ and the inclusion $\iota\colon \{p\}\hookrightarrow\P^n$, the normal space of $(\iota,Q)$ is $\N_p(\iota,Q)=\langle H_p, p\rangle=\P^n$. Furthermore, when $f=\mathrm{id}\colon\P^n\to\P^n$, the normal space of $(\mathrm{id},Q)$ at $p\in\P^n$ is $\N_p(\mathrm{id},Q)=\langle(\P^n)^\perp,p\rangle=\{p\}.$ The analog of the sequence \eqref{eq: osc seq} for higher-order normal spaces is
\[
\{p\}\subseteq \N_p^k(f,Q) \subseteq \cdots \subseteq \N_p^2(f,Q)\subseteq \N_p(f,Q) \subseteq \P^n\,.
\]
\end{remark}

\begin{remark}
Observe that, as $p$ varies in $X$, the dimension of $\N_p^k(f,Q)$ might change. 
Even if we assume that the map $f$ is $k$-osculating, the dimension of the span $\langle(\T_p^k(f))^\perp,p\rangle$ drops if $p\in \T_p^k(f)^\perp,$ which happens if $f(X)$ intersects $Q$ nontransversally.
\end{remark}

\begin{definition}{\cite[p. 245]{piene2022higher}}\label{def: kth distance correspondence}
Consider a morphism $f\colon X\to\P^n$ and a nonsingular quadric hypersurface $Q\subseteq\P^n$. The {\em $k$th-order distance correspondence of $(f,Q)$} is the incidence correspondence
\[
\DC_k(f,Q) \coloneqq \overline{\left\{([v],[u])\in\P^n\times\P^n \mid \text{$[v]=f(p)$ for some $p\in X_k^{\circ}$ and $[\nabla d_{u}(v)]\in \N_p^k(f,Q)$}\right\}}\,.
\]
\end{definition}

The following lemma is proven similarly to the first part of \cite[Theorem 4.1]{DHOST}.

\begin{lemma}\label{lem: higher-order proj ED correspondence is irreducible}
Consider a morphism $f\colon X\to\P^n$ and a nonsingular quadric hypersurface $Q\subseteq\P^n$. Assume that $X$ is irreducible of dimension $m$ and that $f(X)\not\subseteq Q$. Then $\DC_k(f,Q)$ is irreducible of dimension $m+n-m_k$ in $\P^n\times\P^n$. The projection of $\DC_k(f,Q)$ onto the first factor $\P^n$ is locally trivial over $f(X_k^{\circ})\setminus Q$ with fibers of dimension $n-m_k+1$.
\end{lemma}

\begin{definition}\label{def: distance locus}
Consider a morphism $f\colon X\to\P^n$ and a nonsingular quadric hypersurface $Q\subseteq\P^n$. The {\em $k$th-order distance locus of $(f,Q)$} is
\begin{equation}\label{eq: def distance locus}
\DL_k(f,Q)\coloneqq\pr_2(\DC_k(f,Q)) = \overline{ \bigcup_{p \in X_k^{\circ}} \N_p^k(f,Q)}\,,
\end{equation}
where $\pr_2$ is the projection of $\P^n\times\P^n$ onto the second factor. The {\em distance loci sequence} is:
\[
\DL_k(f,Q)\subseteq \DL_{k-1}(f,Q)\subseteq\cdots\subseteq \DL_1(f,Q)=\P^n\,.
\]
\end{definition}

In the following, we denote by $\varphi_{1,k}\colon\DC_k(f,Q)\to f(X)$ and $\varphi_{2,k}\colon\DC_k(f,Q)\to\DL_k(f,Q)$ the surjective morphisms induced by the projections $\pr_1$ and $\pr_2$ of $\P^n\times\P^n$ onto its factors.

\begin{definition}\label{def: kth order ED degree}
Consider a morphism $f\colon X\to\P^n$ and a nonsingular quadric hypersurface $Q\subseteq\P^n$. Assume that the morphism $\varphi_{2,k}$ is generically finite. The {\em $k$th-order distance degree of $(f,Q)$} is
\[
\DD_k(f,Q)\coloneqq\deg\DL_k(f,Q)\cdot\deg\varphi_{2,k}\,,
\]
where $\deg\varphi_{2,k}=\deg\varphi_{2,k}^{-1}(u)$ for a generic $u\in\DL_k(f,Q)$.
\end{definition}

For $k=1$ we use the shorthand $\DC_1(f,Q)=\DC(f,Q)$, while $\varphi_{1,1}=\varphi_1$ and $\varphi_{2,1}=\varphi_2$. In this case $\dim\DC(f,Q)=n$ and $\DL_1(f,Q)=\P^n$. 
Hence, the classical ED degree of $(f,Q)$ introduced in \cite{DHOST} is $\DD_1(f,Q)=\deg\varphi_2$. This is not generally the case for $k\ge 2$, see the running example in the introduction.

We always have the chain of inclusions $f(X)\subseteq \DL_k(f,Q)\subseteq\P^n.$ In fact, there always exists an integer $k\ge 1$ such that $f(X)=\DL_k(f,Q).$ The following lemma examines the two extreme cases: when $f(X)=\DL_k(f,Q)$ and when $\DL_k(f,Q)=\P^n.$

\begin{proposition}\label{prop: linearly normal}
Let $f\colon X\to\P^n$ be a nondegenerate globally $k$-osculating closed embedding and $Q$ a generic quadric. Then
\begin{enumerate}
    \item  There always exists a $k$ for which $f(X)=\DL_k(f,Q)$. Moreover $(X,f)=(\P^m,\nu_m^k)$ is the $k$th Veronese embedding if and only if $f(X)=\DL_k(f,Q)$ and $m_k=\binom{m+k}{k}-1.$
    \item If $\DL_k(f,Q)=\P^n$ for some $k\ge 1$ then $\DL_\ell(f,Q)=\P^n$ for every $\ell\in[k].$
\end{enumerate}
\end{proposition}
\begin{proof}
(1) Given a closed embedding $f\colon X\to\P^n$ we have $m_k=n$ for a large enough $k.$  It follows that $\rank\image j_k=n+1$ and thus $j_k$ is an isomorphism of vector bundles, imposing $\image j_k\cong\OO_X^{\oplus(n+1)}.$ This implies that the morphism $\varphi_{2,k}$ is birational, hence $\DC_k(f,Q)\cong f(X)$ and $\DL_k(f,Q)=f(X).$
If $(X,f)=(\P^m,\nu_m^k)$ is the Veronese embedding, then $f$ is $k$-regular, namely $\image j_k=\PP^k(f)\cong\OO_X^{\oplus(m_k+1)}$ with $m_k+1=\binom{m+k}{k}$, and the map $j_k$ is an isomorphism. This implies, as above, that $\varphi_{2,k}$ is birational and that its image is $f(X)$. Viceversa if $\DL_k(f,Q)=f(X)$ and $m_k+1=\binom{m+k}{k}$, then $\image j_k=\PP^k(f)=\OO_X^{\oplus(m_k+1)}$ and thus the projectivized jet bundle is a product. Applying \cite[Theorem 3.1]{dirocco2001line}, then $f$ is either the $k$th Veronese embedding $\nu_m^k$ or a morphism of an abelian variety. The last possibility can be excluded since $f$ is assumed to be a closed embedding.

(2) Assume now that $\DL_k(f,Q)=\P^n$. Then $m_k=m_{k-1}=\cdots =m,$ the morphism $\varphi_{2,k}=\cdots =\varphi_{2,1}$ is finite, and $\DL_k(f,Q)=\cdots = \DL_1(f,Q)=\P^n.$  
\end{proof}

\begin{definition}\label{def: general k osculating position}
Consider a morphism $f\colon X\to\P^n$ and a nonsingular quadric hypersurface $Q\subseteq\P^n$. The pair $(f,Q)$ is {\em in general $k$-osculating position} if $f$ is globally $k$-osculating and $f(X)$ intersects $Q$ transversally.
\end{definition}
 
If $(f,Q)$ is in general $k$-osculating position, then $\dim\N_p^k(f,Q)$ is constant at all points, in particular for every $p\in X$
\begin{equation}\label{eq: dimension normal space}
\dim\N_p^k(f,Q)=\dim\T_p^k(f)^\perp+1= n-\dim\T_p^k(f)=n-m_k\,.
\end{equation}
In the following, we recall the construction of higher-order Euclidean normal bundles, see \cite[Section 4]{piene2022higher}. The quadratic form $q$ induces an isomorphism
\[
\varphi_q\colon V\otimes\OO_X \longrightarrow V^\vee\otimes\OO_X\,.
\]
In the following, we assume that $f$ is globally $k$-osculating of $k$-osculating dimension $m_k$, in the sense of Definition \ref{def: globally k-osculating}.
Consider the $k$-jet exact sequence
\begin{equation}\label{eq: k-jet 2}
0\to \ker j_k \hookrightarrow V\otimes\OO_X \twoheadrightarrow \image j_k \to 0\,.
\end{equation}
From the inclusion $\ker j_k\hookrightarrow V\otimes\OO_X$ and the isomorphism $\varphi_q$ we get a quotient $V\otimes\OO_X\twoheadrightarrow (\ker j_k)^\vee$.
Considering also the point map $V\otimes\OO_X\twoheadrightarrow \OO_X(1)$, we get a surjection
\begin{equation}\label{eq: surjection defining kth normal bundle}
    V\otimes\OO_X\twoheadrightarrow(\ker j_k)^\vee\oplus\OO_X(1)\,.
\end{equation}
The following definition is given in more generality in \cite[Proposition 4.1]{piene2022higher}.
\begin{definition}\label{def: kth normal bundle}
Consider a morphism $f\colon X\to\P^n$ and a nonsingular quadric hypersurface $Q\subseteq\P^n$. Assume that $(f,Q)$ is in general $k$-osculating position.
The {\em $k$th-order normal bundle} of $(f,Q)$ is the image $\EE_k(f,Q)\coloneqq (\ker j_k)^\vee\oplus\OO_X(1)$ of the surjection in \eqref{eq: surjection defining kth normal bundle}.
\end{definition}

Following notation as in \cite{fulton1998intersection}, we consider the projectivized bundle $\P(\EE)$ of a vector bundle $\EE$ on $X$ as the projective bundle defined by the quotient spaces of the fibers of $\EE$. This means that $\P(\EE)_x=\P(\EE_x^\vee)$ for all $x\in X$.
Additionally, let $\OO_\EE(1)$ be the tautological line bundle of $\EE$ on $\P(\EE)$.

When $(f,Q)$ is in general $k$-osculating position, the degree $\DD_k(f,Q)$ coincides with the degree of the top higher-order reciprocal polar locus defined in \cite[\S 5]{piene2022higher}, as shown in Corollary 5.2 and Remark 5.3 of \cite{piene2022higher}. In the following proposition, we reproduce the proof of this result in the simplified setting of nonsingular varieties and globally $k$-osculating morphisms, using only the tools developed thus far. This formula is essential in the following sections to derive formulas for the degree and dimension of the data locus $\DL_k(f,Q)$.

\begin{proposition}\label{prop: kth order ED degree sum of kth order polar degrees}
Let $f\colon X \to \P^n$ be a morphism and $Q\subseteq\P^n$ a nonsingular quadric hypersurface. If $(f,Q)$ is in general $k$-osculating position, then the morphism $\varphi_{2,k}$ is generically finite over its image and
\begin{equation}\label{eq: kth order ED degree sum of kth order polar degrees}
\DD_k(f,Q) = \sum_{i=0}^m \mu_{k,i}(f)\,.
\end{equation}
where $\mu_{k,i}(f)$ are the degrees of the $k$th polar classes of $f.$
\end{proposition}
\begin{proof}
Consider a morphism $f\colon X\to\P^n$ and a nonsingular quadric hypersurface $Q\subseteq\P^n$ such that $(f,Q)$ is in general $k$-osculating position.
Let $\EE_k(f,Q)$ be the $k$th-order normal bundle of $(f,Q)$ introduced in Definition \ref{def: kth normal bundle}. Then $\EE_k(f,Q)$ is a vector bundle on $X$ of rank $n-m_k+1$. The tautological line bundle $\OO_{\EE_k(f,Q)^\vee}(1)$, whose sections are given by $H^0(X,(\ker j_k)^\vee\oplus \OO_X(1))$, is globally generated and hence defines a morphism on $\P(\EE_k(f,Q)^\vee).$ The fact that $(f,Q)$ is in general $k$-osculating position implies that $\P(\EE_k(f,Q))_p = \N_p^k(f,Q)$ for every $p\in X$, hence $\DC_k(f,Q)=\P(\EE_k(f,Q)^\vee)$, and the morphism defined by its global sections coincides with $\varphi_{2,k}.$

Since $\OO_{\EE_k(f,Q)^\vee}(1)$ is globally generated, it is said to be ``big'' if its top self-intersection is positive.
Recall that, for a vector bundle $\EE$ on $X$ of rank $r,$ if $\pi\colon\P(\EE)\to X$ denotes the projection map, then $H^0(\P(\EE),\OO_\EE(1))=H^0(X,\EE^\vee)$ and $c_1(\OO_\EE(1))^{r-1+i}=s_i(\EE)$, where $s_i(\EE)$ is the $i$th Segre class of $\EE$, see \cite{fulton1998intersection}. In the following, we denote by $s(\EE)$ the total Segre class of $\EE$. These facts and equation \eqref{eq: dimension normal space} yield the relation $c_1(\OO_{\EE_k(f,Q)^\vee}(1))^{m+n-m_k}=s_{m}(\EE_k(f,Q)^\vee)$. It follows that
\[
s(\EE^k(f,Q)^\vee) = s((\ker j_k)\oplus\OO_X(-1)) = s(\ker j_k)s(\OO_X(-1)) = c(\ker j_k)^{-1}c(\OO_X(-1))^{-1}\,.
\]
From the $k$-jet exact sequence \eqref{eq: k-jet 2} we derive the identity $c(\ker j_k)^{-1}=c(\image j_k)$, hence
\begin{equation}\label{chern}
c(\ker j_k)^{-1}c(\OO_X(-1))^{-1}=c(\image j_k)\sum_{i=0}^m c_1(\OO_X(1))^i  
\end{equation}
which implies that
\begin{equation}\label{kpolar}
c_1(\OO_{\EE_k(f,Q)^\vee}(1))^{m+n-m_k}=s_m(\EE_k(f,Q)^\vee)=\sum_{i=0}^m c_i(\image j_k)\cdot c_1(\OO_X(1))^{m-i}\,.
\end{equation}
The hypothesis of global $k$-osculation implies that $\OO_X(1)$ is very ample. Hence, the top self intersection $c_1(\OO_X(1))^{m}$ is always positive. For the same reason, the other terms $c_i(\image j_k)\cdot c_1(\OO_X(1))^{m-i}=0$ are equal to zero if $c_i(\image j_k)=0$, otherwise they are positive. Therefore, the top self intersection $c_1(\OO_{\EE_k(f,Q)^\vee}(1))^{m+n-m_k}$ of the tautological line bundle $\OO_{\EE_k(f,Q)^\vee}(1)$ is always positive. For this reason, we conclude that the morphism $\varphi_{2,k}$ is generically finite over its image and
\begin{equation}\label{eq: identity top self-intersection tautological line bundle}
c_1(\OO_{\EE_k(f,Q)^\vee}(1))^{m+n-m_k} = \deg\image\varphi_{2,k}\cdot\deg\varphi_{2,k}^{-1}(u)
\end{equation}
for a generic $u\in\image\varphi_{2,k}=\DL_k(f,Q)$. The right-hand side of \eqref{eq: identity top self-intersection tautological line bundle} is equal to $\DD_k(f,Q)$ by Definition \ref{def: kth order ED degree}.
The formula in \eqref{eq: kth order ED degree sum of kth order polar degrees} follows from \eqref{chern} and \eqref{kpolar}.
\end{proof}

This formula shows that the generic distance degree, which we denote by $\gDD(f),$ is independent of $Q$. Moreover, the formula can now be used as a definition of the generic distance degree for any morphism (not necessarily globally $k$-osculating).

\begin{definition}\label{def:generic DD}
Let $f\colon X\to \P^n$ be a morphism. The {\em generic $k$th distance degree of $f$} is:
\[
\gDD_k(f)\coloneqq\sum_{i=0}^m \mu_{k,i}(f)\,.
\]
\end{definition}

\begin{remark}\label{rem: hyperplane_k2}
It is natural to investigate the behaviour of the generic $k$th-order distance degree under geometric operations such as projections and hyperplane sections. Regarding projections, it was shown in \cite[Proposition~3.2]{piene2022higher} that the $k$th-order polar degrees of a morphism $f'$ and of its composition $f=\pi\circ f'$ with a linear projection $\pi$ are equal under suitable assumptions; see also Lemma \ref{lem: technical lemma}. This yields a result analogous to \cite[Corollary~6.1]{DHOST}.
A similar study for hyperplane sections is more involved. The case $k=1$ is addressed in \cite[Theorem~4.1]{piene1978polar}. There is no generalization to higher-order polar degrees as observed in \cite[Section 3]{piene2022higher}. For $k\ge 2$ restricting higher osculating data to a hyperplane section introduces mixed tangential-normal contributions (equivalently, higher fundamental forms enter). For $k=1$ under suitable genericity assumptions, the difference is given by the degree of the dual variety. For $k\geq 2$ one should not expect a correction governed by a single ``$k$-dual degree'' in general.
We give a closed formula in the case of surfaces and $k=2$ to illustrate the difference.

Let $f\colon S\hookrightarrow\P^n$ be a $3$-regular embedding of a nonsingular projective surface, and write $c(S)=1+c_1+c_2$ and $c(\mathcal O_S(1))=1+L$. The upcoming equations \eqref{eq: gEDD k-regular surface} and \eqref{eq: identity total chern class J_k surface} for $k=2$ yield the closed formulas
\begin{align*}
    \gDD_2(f) &= \int_S 22\,L^2-24\,c_1L+5\,c_1^2+5\,c_2\\
    c_2(\PP_2(f)) &= 15\,L^2-20\,c_1L+5\,c_1^2+5\,c_2\,.
\end{align*}
Moreover, under the hypothesis that $f$ is $3$-regular implies that (see \cite[Proposition 2.4]{lanteri1999higher})
\[
\deg(S_2^\vee)=\int_S c_2(\PP_2(f))\,.
\]
Now let $H\subset\mathbb P^n$ be a general hyperplane and $C=S\cap H$ nonsingular, and assume moreover that the induced embedding $f|_C\colon C\hookrightarrow \mathbb P^{n-1}$ is $2$-regular. Then equation \eqref{eq: gEDD k-regular curve} for $k=2$ reads
\[
\gDD_2(f|_C) = \int_C 4L-3\,c_1(C) = \int_S 7\,L^2-3\,c_1L\,,
\]
where we used adjunction for the nonsingular hyperplane section $C=S\cap H$, namely the exact sequence
$0\to T_C\to T_S|_C\to\mathcal O_C(1)\to 0$, which implies
$c_1(T_C)=(c_1(T_S)-L)|_C$ and hence
$\int_S c_1(T_S)L=\int_C c_1(T_C)+\deg L_C$. Consequently, one obtains the explicit discrepancy
\[
\gDD_2(f)-\gDD_2(f|_C)
=\int_S 15\,L^2-21\,c_1L+5\,c_1^2+5\,c_2
=\int_S c_2(\PP_2(f))-c_1L\,.
\]
Thus, even in the favourable situation of $3$-regular  embeddings, the second dual contribution
$\int_S c_2(\PP_2(f))=\deg(S_2^\vee)$ appears naturally, but does not exhaust the discrepancy: an additional term
$\int_S c_1L$ remains. Closed formulas for higher $k$ and higher dimensions quickly become combinatorially involved. 
For threefolds one can still obtain closed formulas for $k=2$ under suitable regularity assumptions:
Proposition \ref{prop: gEDD k-regular threefold} gives $\mathrm{gDD}_2(f)$ as an explicit polynomial in $L,c_1,c_2,c_3$, and combining
Proposition~\ref{prop: gEDD k-regular surface} for the hyperplane section $S=X\cap H$ with adjunction yields an explicit discrepancy
$\mathrm{gDD}_2(f)-\mathrm{gDD}_2(f|_S)$ in terms of the Chern numbers of $X$; however, the resulting expression involves several independent Chern monomials.
\end{remark}

We conclude this section with the first application of Proposition \ref{prop: kth order ED degree sum of kth order polar degrees} to Veronese embeddings. The coming result generalizes the formula in \cite[Proposition 7.10]{DHOST} for $k=1$.

\begin{corollary}\label{corol: generic kth ED degree Veronese}
Consider the Veronese embedding $\nu_m^d\colon\P^m\hookrightarrow\P^{\binom{m+d}{d}-1}$. Then for every $k\in[d]$
\begin{equation}\label{eq: explicit EDD k polar classes Veronese}
    \gDD_k(\nu_m^d) = \sum_{i=0}^m \binom{\binom{m+k}{k}}{i}(d-k)^id^{m-i}\,.
\end{equation}
\end{corollary}
\begin{proof}
The generic $k$th-order distance degree of $\nu_m^d$ is, by Proposition \ref{prop: kth order ED degree sum of kth order polar degrees}, equal to
\begin{equation}\label{eq: EDD k polar classes Veronese}
    \gDD_k(\nu_m^d) = \sum_{i=0}^m \mu_{k,i}(\nu_m^d) = \sum_{i=0}^m\int_X c_i(\image j_k)\,.
\end{equation}
Since $\nu_m^d$ is $k$-regular for all $k\in[d]$, we have $\image j_k=\PP^k(\nu_m^d)=\OO_{\P^m}(d-k)^{\oplus\binom{m+k}{k}}$. Hence, letting $h=c_1(\OO_{\P^m}(1))$, we have $c(\PP^k(\nu_m^d)) = (1+(d-k)h)^{\binom{m+k}{k}} = \sum_{i=0}^m\binom{\binom{m+k}{k}}{i}(d-k)^ih^i$, in particular
\[
\int_X c_i(\PP^k(\nu_m^d)) = \int_X \binom{\binom{m+k}{k}}{i}(d-k)^ih^i\cdot (dh)^{m-i}=\binom{\binom{m+k}{k}}{i}(d-k)^id^{m-i}\,,
\]
giving the desired identity.
\end{proof}

\section{Osculating eigenvectors of symmetric tensors}\label{sec: osculating eigenvectors}

We start by setting some preliminary notation. Given a vector $v=(v_0,\dots,v_m)\in\C^{m+1}$, we use the shorthand $v^\alpha$ for the product $v_0^{\alpha_0}\cdots v_m^{\alpha_m}$ for all $\alpha=(\alpha_0,\dots,\alpha_m)\in\N^{m+1}$. Furthermore we define $|\alpha|=\alpha_0+\cdots+\alpha_m$. For all $k\in\N$ we define the vector $v^k\coloneqq (v^\alpha)_{\alpha\in\N^{m+1},|\alpha|=k}\in\C^{\binom{m+k}{k}}$.

A symmetric tensor of order $d$ over $\C^{m+1}$ is an element of the tensor space $\mathrm{Sym}^d(\C^{m+1})^*$, which we identify with the space $\C[x_0,\dots,x_m]_d$ of homogeneous polynomials of degree $d$ in $m+1$ variables. In particular, we write an element $f\in \C[x_0,\dots,x_m]_d$ as
\begin{equation}\label{eq: write f}
f(x_0,\dots,x_m) = \sum_{|\alpha|=d}\binom{d}{\alpha}f_\alpha x^\alpha\,,
\end{equation}
where $\binom{d}{\alpha}=\frac{d!}{\alpha_0!\cdots\alpha_m!}$. In particular, we identify $f$ with the vector $(f_\alpha)_{|\alpha|=d}\in\C^{\binom{m+d}{d}}$.
The monomials $(x^\alpha)_{|\alpha|=d}$ form a basis for the Veronese embedding $\nu_m^d$. 

The geometry induced by a specific quadric hypersurface may differ from the generic behavior. In the following, we analyze the geometry of the Bombieri-Weyl metric.

\begin{definition}\label{def: BW inner product}
Given $f=(f_\alpha)_{|\alpha|=d}$ and $g=(g_\alpha)_{|\alpha|=d}$ with real coordinates, the Bombieri-Weyl inner product between $f$ and $g$ is 
\[
\langle f,g\rangle_\mBW\coloneqq \sum_{|\alpha|=d}\binom{d}{\alpha}f_\alpha g_\alpha\,.
\]    
\end{definition}

In particular, the Bombieri-Weyl norm of $f=(f_\alpha)_{|\alpha|=d}$ is the square root of the quadratic form $q_\mBW(f) \coloneqq \sum_{|\alpha|=d}\binom{d}{\alpha}f_\alpha^2$, while $\sqrt{q_\mBW(f-g)}$ is the Bombieri-Weyl distance between the polynomials $f$ and $g$. In this section, we aim to study the higher-order critical points of the Bombieri-Weyl distance minimization problem
\begin{equation}\label{eq: rank one approx problem}
   \min_{\ell\in(\R^{n+1})^*}q_\mBW(f-\ell^d)\,,\quad f\in \R[x_0,\dots,x_m]_d\,.
\end{equation}
As for the general case, we relax the assumptions that $f$ and $\ell$ have real coefficients; therefore, we study the higher-order critical points of the complex-valued polynomial function $q_\mBW(f-\ell^d)$. Let $q(x)=\sum_{i=0}^mx_i^2$ be the standard Euclidean quadratic form in $(\R^{m+1})^*$ and $S_q^m=\V(q(x)-1)\subseteq\R^{m+1}$ the $m$-dimensional unit sphere associated with $q$. Consider the map
\begin{equation}\label{eq: def lambda times power l}
    S_q^m\times \R \longrightarrow \R[x_0,\dots,x_m]_d\,,\quad (v,\lambda)\mapsto \lambda\langle v, x\rangle^d\,.
\end{equation}
Then the subset $\{\ell^d\mid \ell\in (\C^{m+1})^*\}\subseteq \C[x_0,\dots,x_m]_d$ is the Zariski closure of the image of the map in \eqref{eq: def lambda times power l}. More precisely, it equals the affine cone over the Veronese variety $V_m^d\subseteq\P(\C[x_0,\dots,x_m]_d)$ defined in \eqref{eq: def Segre-Veronese embedding}. Furthermore, two important identities coming from the definition of the Bombieri-Weyl quadratic form are
\begin{align}\label{eq: identity BW quadratic forms}
\begin{split}
    q_\mBW(\langle v, x\rangle^d) &= q(v)^d\quad\forall\,v\in\C^{m+1}\\
    \left\langle\prod_{j=1}^d\langle x,v_j\rangle,\prod_{j=1}^d\langle x,w_j\rangle\right\rangle_\mBW &= \prod_{j=1}^d\langle v_j,w_j\rangle\quad\forall\,v_j,w_j\in\C^{m+1}\,.
\end{split}
\end{align}

\begin{remark}\label{rmk: equivalence optimization problems}
The minimization problem \eqref{eq: rank one approx problem} can be rephrased as
\begin{equation}\label{eq: optimization problem phi}
    \min_{v\in S_q^m}\min_{\lambda\in\R}q_\mBW(f-\lambda\,\ell^d)\,,\quad f\in \R[x_0,\dots,x_m]_d\,,
\end{equation}
where $\ell=\ell(v)\coloneqq\langle x,v\rangle\in(\C^{m+1})^*$.
Let $\langle \ell^d\rangle$ be the line spanned by $\ell^d$. Then the projection of $f$ onto $\langle \ell^d\rangle$ is denoted by $P_{\langle \ell^d\rangle}(f)$. We also consider the orthogonal complement $\langle \ell^d\rangle^\perp$ of $\langle \ell^d\rangle$ with respect to the Bombieri-Weyl inner product. By the Pythagorean identity, we have that
\begin{equation}\label{eq: Pythagoras}
q_\mBW(f)=q_\mBW(P_{\langle \ell^d\rangle}(f))+q_\mBW(P_{\langle \ell^d\rangle^\perp}(f))\,.
\end{equation}
Furthermore, we have that $\min_{\lambda\in\R}q_\mBW(f-\lambda\,\ell^d) = q_\mBW(P_{\langle \ell^d\rangle^\perp}(f))$.
Hence, we can say that the original problem \eqref{eq: optimization problem phi} is now rephrased as
\begin{equation}\label{eq: equivalence optimization problems}
    \min_{v\in S_q^m}q_\mBW(P_{\langle \ell^d\rangle^\perp}(f)) \Longleftrightarrow \max_{v\in S_q^m}q_\mBW(P_{\langle \ell^d\rangle}(f))\,,
\end{equation}
and the equivalence is an immediate consequence of \eqref{eq: Pythagoras}.
The identity \eqref{eq: identity BW quadratic forms} allows us to write $P_{\langle \ell^d\rangle}(f) = \langle f, \ell^d\rangle_\mBW\, \ell^d\quad\forall\,v\in S_q^m$, where $q_\mBW(\ell^d)=1$ thanks to \eqref{eq: identity BW quadratic forms}. But then the second part of \eqref{eq: equivalence optimization problems} is in turn equivalent to
\begin{equation}\label{eq: simplified problem}
    \max_{v\in S_q^m}\langle f, \ell^d\rangle_\mBW\,.
\end{equation}
In particular, $\langle f, \ell^d\rangle_\mBW$ is a polynomial objective function of degree $d$ in the coordinates of $v$. In short, minimizing the Bombieri-Weyl distance between $f$ and the affine cone $CV_m^d$ is equivalent to maximizing the function $\langle f, \ell^d\rangle_\mBW$ for $v\in S_q^m$.
\end{remark}

\begin{lemma}\label{lem: osculating space of a Veronese variety}
The Veronese embedding $\nu_m^d\colon\P^m\hookrightarrow\P(\C[x_0,\dots,x_m]_d)$ is $k$-regular for all $k\le d$. Furthermore, for all $\ell\in(\C^{m+1})^*\setminus\{0\}$ and $k\le d$, we have
\begin{equation}
    \T_{[\ell^d]}^k(\nu_m^d) = \{[\ell^{d-k}g]\mid g\in \C[x_0,\dots,x_m]_k\}\,.
\end{equation}
\end{lemma}

\begin{definition}\label{def: kth order eigenpoint}
Let $f\in\C[x_0,\dots,x_m]_d$ and $k\le d$. A nonzero vector $v\in\C^{m+1}$ is a {\em normalized eigenvector of $f$ of order $k$} if $q(v)=1$ and there exists $\lambda\in\C$ such that
\begin{equation}\label{eq: normalized eigenpair}
\frac{(d-k)!}{d!}\nabla_kf(v) = \lambda\,v^k\,,
\end{equation}
where $\nabla_kf\coloneqq \left(\frac{\partial^k f}{\partial x^\alpha}\right)_{|\alpha|=k}\in\C^{\binom{m+k}{k}}$.
The value $\lambda$ is the {\em eigenvalue of $f$ of order $k$} associated with $v$. The pair $(v,\lambda)$ is called a {\em normalized eigenpair} of $f$ of order $k$.
\end{definition}

\begin{remark}\label{rmk: even odd}
If $d$ and $k$ are odd or $d$ is odd and $k$ is even, then if $(v,\lambda)$ is a solution of \eqref{eq: normalized eigenpair}, then also $(-v,-\lambda)$ is. Instead, if $d$ is even and $k$ is odd or $d$ and $k$ are even, then if $(v,\lambda)$ is a solution of \eqref{eq: normalized eigenpair}, then also $(v,-\lambda)$ is.
\end{remark}

\begin{lemma}\label{lem: identity BW inner product partial derivatives}
Let $f\in\C[x_0,\dots,x_m]_d$ and $k\le d$. For every $v\in\C^{m+1}$, we define $\ell=\ell(v)\coloneqq\langle x,v\rangle\in(\C^{m+1})^*$. For any $\alpha\in\N^{m+1}$ with $|\alpha|=k$
\[
\langle f,\ell^{d-k}x^\alpha\rangle_\mBW = \frac{(d-k)!}{d!}\frac{\partial^k f}{\partial x^\alpha}(v)\,.
\]
\end{lemma}
\begin{proof}
Writing $f$ as in \eqref{eq: write f} and applying Definition \ref{def: BW inner product}, we get
\begin{align*}
    \langle f,\ell^{d-k}x^\alpha\rangle_\mBW &= \sum_{|\gamma|=d}\sum_{|\beta|=d-k}\binom{d-k}{\beta}\left\langle\binom{d}{\gamma}x^\gamma,x^{\alpha+\beta}\right\rangle_\mBW f_\gamma v^\beta = \sum_{\substack{|\gamma|=d\\\gamma-\alpha\ge 0}}\binom{d-k}{\gamma-\alpha}f_\gamma v^{\gamma-\alpha}\\
    &= \frac{(d-k)!}{d!}\sum_{\substack{|\gamma|=d\\\gamma-\alpha\ge 0}}\binom{d}{\gamma-\alpha}f_\gamma v^{\gamma-\alpha} = \frac{(d-k)!}{d!}\sum_{\substack{|\gamma|=d\\\gamma-\alpha\ge 0}}\binom{d}{\gamma}f_\gamma\left[\prod_{j=0}^m\frac{\gamma_j!}{(\gamma_j-\alpha_j)!}v^{\gamma-\alpha}\right]\\
    &= \frac{(d-k)!}{d!}\sum_{\substack{|\gamma|=d\\\gamma-\alpha\ge 0}}\binom{d}{\gamma}f_\gamma\left[\frac{\partial^k x^\gamma}{\partial x^\alpha}(v)\right] = \frac{(d-k)!}{d!}\frac{\partial^k f}{\partial x^\alpha}(v)\,,
\end{align*}
namely the desired identity.
\end{proof}

\begin{proposition}\label{prop: critical points BW distance}
Let $f\in\C[x_0,\dots,x_m]_d$ and $k\le d$. For every $v\in\C^{m+1}$, we define $\ell=\ell(v)\coloneqq\langle x,v\rangle\in(\C^{m+1})^*$. A vector $\lambda\,\ell^d$ with $v\in S_q^m$ and $\lambda\in\C\setminus\{0\}$ is a critical point of order $k$ of the function $q_\mBW(f-\lambda\,\ell^d)$ if and only if $(v,\lambda)$ is a normalized eigenpair of $f$ of order $k$.
\end{proposition}
\begin{proof}
Consider the linear form $\ell(v)=\langle x, v\rangle\in(\C^{m+1})^*$.
The tensor $\lambda\,\ell^d$ is critical of order $k$ of $q_\mBW(f-\lambda\,\ell^d)$ if and only if $\langle f-\lambda\,\ell^d, w\rangle_\mBW=0$ for every $w\in\T_{[\ell^d]}^k(\nu_m^d)$. Using Lemma \ref{lem: osculating space of a Veronese variety} and linearity of the inner product, we can simplify this problem by imposing that
\[
\langle f-\lambda\,\ell^d,\ell^{d-k}x^\alpha\rangle_\mBW=0\quad\forall\,\alpha\in\N^{m+1}\,\ |\alpha|=k\,.
\]
For any $\alpha\in\N^{m+1}$ with $|\alpha|=k$, we compute $\langle f,\ell^{d-k}x^\alpha\rangle_\mBW$ and $\langle\ell^d,\ell^{d-k}x^\alpha\rangle_\mBW$. The former is computed in  Lemma \ref{lem: identity BW inner product partial derivatives}. The latter is computed applying Lemma \ref{lem: identity BW inner product partial derivatives} replacing $f$ with $\ell^d$:
\begin{align*}
    \langle\ell^d,\ell^{d-k}x^\alpha\rangle_\mBW &= q(\ell)^{d-k}\langle\ell^k,x^\alpha\rangle_\mBW = \sum_{|\beta|=k}\binom{d}{\alpha}^{-1}v^\beta\left\langle\binom{d}{\beta}x^\beta,\binom{d}{\alpha}x^\alpha\right\rangle_\mBW\\
    &= \sum_{|\beta|=k}\binom{d}{\alpha}^{-1}v^\beta\binom{d}{\alpha}\delta_{\beta,\alpha} = v^\alpha\,,
\end{align*}
where the first and second equalities follow from \eqref{eq: identity BW quadratic forms} and $\delta_{\beta,\alpha}\neq 0$ if and only if $\beta=\alpha$, in which case $\delta_{\beta,\beta}=1$.
Summing up, the point $\lambda\,\ell^d$ with $v\in S_q^m$ is critical of order $k$ of $q_\mBW(f-\lambda\,\ell^d)$ if and only if $\frac{(d-k)!}{d!}\frac{\partial^k f}{\partial x^\alpha}(v) = \lambda\,v^\alpha$ for all $\alpha\in\N^{m+1}$ with $|\alpha|=k$, or equivalently, if and only if $(v,\lambda)$ is a normalized eigenpair of $f$.
\end{proof}

In Proposition \ref{prop: critical points BW distance}, it is fundamental to use the Bombieri-Weyl quadratic form $q_\mBW$.
If we replace $Q_\mBW$ with a generic quadric $Q$, then $\DD_k(\nu_m^d,Q)=\gDD_k(\nu_m^d)$, and the latter is computed in Corollary \ref{corol: generic kth ED degree Veronese}. Our next goal is to compute $\DD_k(\nu_m^d,Q_\mBW)$ instead.

\begin{proposition}\label{prop: kth order ED degree BW}
Let $\nu_m^d$ be the degree-$d$ Veronese embedding of $\P^m$, and equip $\R[x_0,\dots,x_m]_d$ with the Bombieri-Weyl inner product. For all $k\le d$, the $k$th-order distance degree $\DD_k(\nu_m^d,Q_\mBW)$ equals the coefficient of the monomial $h_1^mh_2^{\binom{m+k}{k}-m-1}$ in the expansion of \eqref{eq: kth order ED degree BW}.
Furthermore, the $k$th-order distance locus of $(\nu_m^d,Q_\mBW)$ is irreducible of dimension $m+\binom{m+d}{d}-\binom{m+k}{k}$.
\end{proposition}
\begin{proof}
By Proposition \ref{prop: critical points BW distance}, the $k$th-order critical points of the Bombieri-Weyl distance function from $f$ restricted to $\nu_m^d(\P^m)$ correspond to the $k$th-order eigenvectors of $f$.
Consider the $\binom{m+k}{k}\times 2$ matrix
\begin{equation}\label{eq: matrix defining order-k eigenvectors}
M_k(x,f)\coloneqq
\begin{pmatrix}
    \nabla_k f(x) & x^k
\end{pmatrix}\,.
\end{equation}
We consider the pair $([x],[f])\in Z\coloneqq \P^m\times \P(\C[x_0,\dots,x_m]_d)$. Note that the columns of $M_k(x,f)$ have bidegrees $(k,0)$ and $(d-k,1)$ in the coordinates of $x$ and $f$, respectively. This means that the transpose of $M_k(x,f)$ defines the fiber $\varphi_{[x],[f]}$ over $([x],[f])$ of the vector bundle morphism
\begin{equation}\label{eq: morphism phi}
    \varphi\colon \EE\to\FF\,,\quad \EE \coloneqq \OO_Z^{\binom{m+k}{k}}\,,\quad \FF\coloneqq \OO_Z(k,0)\oplus\OO_Z(d-k,1)\,.
\end{equation}
The degeneracy locus $D(\varphi)\coloneqq\{z\in Z\mid \rank\varphi_z\le 1\}$ coincides with the $k$th-order distance correspondence $\DC_k(\nu_m^d,Q_\mBW)$ introduced in Definition \ref{def: kth distance correspondence}. Observe that $\DC_k(\nu_m^d,Q_\mBW)\neq\emptyset$ and
\begin{align*}
\dim \DC_k(\nu_m^d,Q_\mBW)= m+\binom{m+d}{d}-\binom{m+k}{k}= \dim Z-(\rank\EE-1)(\rank\FF-1)
\end{align*}
coincides with the expected dimension of $D(\varphi)$. Therefore the degree of $D(\varphi)$, or the degree of $\DC_k(\nu_m^d,Q_\mBW)$, can be computed applying Porteous' formula \cite[(4.2), p. 86]{arbarello1985geometry}:
\begin{equation}
\deg D(\varphi) 
= (-1)^{\binom{m+k}{k}-1}c_{\binom{m+k}{k}-1}(\EE-\FF)\,,
\end{equation}
where
\begin{align*}
c_t(\EE-\FF) &= \frac{c_t(\EE)}{c_t(\FF)} = \frac{1}{(1+kh_1)(1+(d-k)h_1+h_2)}\\
&= \frac{1}{1+dh_1+h_2+k(d-k)h_1^2+kh_1h_2}\\
&= \sum_{j=0}^\infty(-1)^j(dh_1+h_2+k(d-k)h_1^2+kh_1h_2)^j\,.
\end{align*}
The $k$th-order distance locus $\DL_k(\nu_m^d,Q_\mBW)=\varphi_{2,k}(\DC_k(\nu_m^d,Q_\mBW))$ is irreducible of dimension $m+\binom{m+d}{d}-\binom{m+k}{k}$ by Lemma \ref{lem: higher-order proj ED correspondence is irreducible}, or of codimension $\binom{m+k}{k}-m-1$. Hence $\DD_k(\nu_m^d,Q_\mBW)$ equals the coefficient of the monomial $h_1^mh_2^{\binom{m+k}{k}-m-1}$ in the previous expansion.
\end{proof}

We determine close formulas for $\DD_k(\nu_m^d,Q_\mBW)$ for small values of $m$.

\begin{corollary}\label{corol: formulas kth order ED degree small n}
For all $k\le d$, we have
\begin{align*}
    \DD_k(\nu_1^d,Q_\mBW) &= k(d-k+1)\\    
    \DD_k(\nu_2^d,Q_\mBW) &= \frac{(d-k)^2}{2}\binom{k+2}{k}^2-\frac{(d-k)(3d-5k)}{2}\binom{k+2}{k}+(d-2k)^2\,.
\end{align*}
\end{corollary}
\begin{proof}
Let $m=1$ and $k\le d$. By Proposition \ref{prop: kth order ED degree BW}, $\DD_k(\nu_1^d,Q_\mBW)$ is the coefficient of $h_1h_2^{k-1}$ in the expansion of
\begin{equation}
	(-1)^{k}\sum_{j=0}^\infty(-1)^j(dh_1+h_2+kh_1h_2)^j \in \frac{\Z[h_1,h_2]}{\langle h_1^2,h_2^{d+1}\rangle}\,.
\end{equation}
Observe that
\begin{align*}
	\sum_{j=0}^\infty(-1)^j(dh_1+h_2+kh_1h_2)^j &= \sum_{j=0}^\infty(-1)^j\sum_{i_1+i_2+i_3=j}\binom{j}{i_1,i_2,i_3}d^{i_1}k^{i_3}h_1^{i_1+i_3}h_2^{i_2+i_3}\,.
\end{align*}
To extract the coefficient of $h_1h_2^{k-1}$, we need to impose that $i_1+i_3=1$, hence $(i_1,i_3)\in\{(1,0),(0,1)\}$ and
$i_2=j-1$, therefore, from the inner sum we extract the two summands $jdh_1h_2^{j-1}+jkh_1h_2^j$. The coefficient of $h_1h_2^{k-1}$ is obtained selecting $j\in\{k-1,k\}$, hence
\[
\DD_k(\nu_1^d,Q_\mBW) = (-1)^k[(-1)^{k-1}(k-1)k+(-1)^kkd]=k(d-k+1)\,.
\]

Now consider $m=2$ and $k\le d$. In this case $\DD_k(\nu_2^d,Q_\mBW)$ is the coefficient of $h_1^2h_2^{\binom{k+2}{2}-3}$ in the expansion of
\begin{equation}
	(-1)^{\binom{k+2}{k}-1}\sum_{j=0}^\infty(-1)^j(dh_1+h_2+k(d-k)h_1^2+kh_1h_2)^j \in \frac{\Z[h_1,h_2]}{\langle h_1^3,h_2^{\binom{d+2}{d}}\rangle}\,.
\end{equation}
Observe that
\begin{align*}
	(dh_1+h_2+k(d-k)h_1^2+kh_1h_2)^j = \sum_{|i|=j}\binom{j}{i}d^{i_1}k^{i_3+i_4}(d-k)^{i_3}h_1^{i_1+2i_3+i_4}h_2^{i_2+i_4}\,.
\end{align*}
To extract the coefficient of $h_1^2h_2^{\binom{k+2}{2}-3}$, we need to assume that $i_1+2i_3+i_4=2$. This implies that $(i_1,i_3,i_4)\in\{(2,0,0),(1,0,1),(0,1,0),(0,0,2)\}$. Using $|i|=j$, we obtain $i_2=j-i_1-i_3-i_4$ and the four possible summands
\[
\frac{j(j-1)}{2}d^2h_1^2h_2^{j-2}+j(j-1)dkh_1^2h_2^{j-1}+jk(d-k)h_1^2h_2^{j-1}+\frac{j(j-1)}{2}k^2h_1^2h_2^{j}\,.
\]
Therefore, the coefficient of $h_1^2h_2^{\binom{k+2}{2}-3}$ is obtained selecting the indices $j\in\{\binom{k+2}{2}-3,\binom{k+2}{2}-2,\binom{k+2}{2}-1\}$. Calling $S\coloneqq \binom{k+2}{k}$, we obtain the value
\begin{align}\label{eq: formula EDD k Veronese PP2}
\begin{split}
\DD_k(\nu_2^d,Q_\mBW) &= \frac{(S-3)(S-4)}{2}k^2-(S-2)(S-3)dk-(S-2)k(d-k)\\
&\quad+\frac{(S-1)(S-2)}{2}d^2\\
&= \frac{(d-k)^2}{2}S^2-\frac{(d-k)(3d-5k)}{2}S+(d-2k)^2\,,
\end{split}
\end{align}
where we have already multiplied every summand by the common coefficient $(-1)^{\binom{k+2}{k}-1}$.
\end{proof}

It is interesting to compare the formulas in Corollary \ref{corol: formulas kth order ED degree small n} with
\begin{align*}
    \gDD_k(\nu_1^d) &= d+(k+1)(d-k)\\
    \gDD_k(\nu_2^d) &= \frac{(d-k)^2}{2}S^2+\frac{(d+k)(d-k)}{2}S+d^2\,,
\end{align*}
computed using Corollary \ref{corol: generic kth ED degree Veronese}. For example $\DD_2(\nu_2^3,Q_\mBW)=22<42=\gDD_2(\nu_2^3)$.

\subsection{Proof of Theorem \ref{thm: projection data locus Veronese BW is either birational or of degree 2}}\label{proof4}

The next proof computes the degrees $\deg\varphi_{2,k}$ and $\deg\DL_k(\nu_m^d,Q)$, namely the two factors of $\DD_k(f,Q)$, when either $Q=Q_\mBW$ or $(f,Q)$ is in general $k$-osculating position. It is a necessary step towards the proof of the main Theorem \ref{thm: generic ED degree general polynomial map}.

\begin{proof}[Proof of Theorem \ref{thm: projection data locus Veronese BW is either birational or of degree 2}]
By definition, the Veronese embedding $\nu_m^d\colon\P^m\hookrightarrow\P(\C[x_0,\dots,x_m]_d)$ sends the class $[\ell]\in\P^m$ of the linear form $\ell=\langle v,x\rangle$ for some $v=(v_0,\dots,v_m)\in\C^{m+1}$ to $[\ell^d]\in\P(\C[x_0,\dots,x_m]_d)$. For all $[\ell]\in\P^m$, we have the identity
\begin{equation}\label{eq: Lasker}
\T_{[\ell^d]}^k(\nu_m^d)^\perp = \left\{[f]\in\P(\C[x_0,\dots,x_m]_d)\ \bigg|\ \text{$\frac{\partial^k f}{\partial x^\alpha}(v)=0$ for all $\alpha\in\N^{m+1}$ with $|\alpha|=k$}\right\}\,,
\end{equation}
which is a generalization of a result attributed to Lasker, see \cite[\S1.2, Proposition 1]{ottaviani2013five}, that follows by Lemma \ref{lem: identity BW inner product partial derivatives}.

We consider first the case $(m,k)=(1,d-1)$. Using \eqref{eq: Lasker}, one verifies that $\T_{[\ell^d]}^k(\nu_1^d)^\perp=\{[(\ell^\perp)^{k+1}g]\in\P^d\mid g\in\C[x_0,x_1]_{d-k-1}\}$ for all $k\le d$, where $\ell^\perp = \langle v^\perp,x\rangle$ and $v^\perp=(v_1,-v_0)$. In particular $\T_{[\ell^d]}^{d-1}(\nu_1^d)^\perp=\{[(\ell^\perp)^d]\}$, hence $\N_{[\ell^d]}^{d-1}(\nu_1^d,Q_\mBW)$ is the projective subspace generated by $[\ell^d]$ and $[(\ell^\perp)^d]$. If $v$ is generic, in particular if $v_0^2+v_1^2\neq 0$, then $\N_{[\ell^d]}^{d-1}(\nu_1^d,Q_\mBW)$ is a projective line; otherwise, it coincides with the point $[\ell^d]$.

Now consider a generic point $[f]\in\DL_{d-1}(\nu_1^d,Q_\mBW)$, in particular $f=\lambda\ell^d+\mu(\ell^\perp)^d$ for some $(\lambda,\nu)\in\C^2\setminus\{(0,0)\}$ and $\ell=\langle v,x\rangle$ for some $v=(v_0,v_1)\in\C^2$ such that $v_0^2+v_1^2\neq 0$. From the previous considerations, we conclude that $f\in \N_{[\ell^d]}^{d-1}(\nu_1^d,Q_\mBW)=\N_{[(\ell^\perp)^d]}^{d-1}(\nu_1^d,Q_\mBW)$, implying that $\varphi_{1,d-1}(\varphi_{2,d-1}^{-1}([f]))=\{[\ell^d],[(\ell^\perp)^d]\}$. We conclude that $\deg\varphi_{2,d-1}=2$. Furthermore, by Corollary \ref{corol: formulas kth order ED degree small n} we have $\DD_{d-1}(\nu_1^d,Q_\mBW)=2(d-1)$, hence necessarily $\deg\DL_{d-1}(\nu_1^d,Q_\mBW)=d-1$.

Now assume that $(m,k)\neq(1,d-1)$. To prove our statement, it is enough to find a point $[f]\in\DL_k(\nu_m^d,Q_\mBW)$ such that $\varphi_{2,k}^{-1}([f])$ is zero-dimensional, reduced, and of degree one. Indeed, this would imply that $\varphi_{2,k}^{-1}([f])$ consists of a simple point for all $[f]$ in an open dense subset of $\DL_k(\nu_m^d,Q_\mBW)$, thus yielding the birationality of $\varphi_{2,k}$ over its image $\DL_k(\nu_m^d,Q_\mBW)$. Thanks to Proposition \ref{prop: critical points BW distance}, our claim is equivalent to finding a homogeneous polynomial $f\in\C[x_0,\dots,x_m]_d$ with a unique normalized eigenvector of order $k$.
Fix the standard basis vector $e_0=(1,0,\dots,0)\in\C^{m+1}$, and consider the point $[x_0^d]=[\langle e_0,x\rangle^d]\in \nu_m^d(\P^m)$. Using \eqref{eq: Lasker}, one verifies that
\[
\T_{[x_0^d]}^k(\nu_m^d)^\perp = \left\{[f]\in\P(\C[x_0,\dots,x_m]_d) \mid \deg_{x_0}(f)\le d-k-1\right\}
\]
for all $k\le d-1$, while $\T_{[x_0^d]}^d(\nu_m^d)^\perp = \emptyset$. This means that $\N_{[x_0^d]}^d(\nu_m^d,Q_\mBW)=\{[x_0^d]\}$, hence the morphism $\varphi_{2,k}$ is birational and $\DL_d(\nu_m^d,Q_\mBW)=\nu_m^d(\P^m)$. Now, assume $k\le d-1$. A generic element of $\N_{[x_0^d]}^k(\nu_m^d,Q_\mBW)$ is the class of a polynomial of the form $f=x_0^d+\sum_{i=0}^{d-k-1}x_0^{d-k-1-i}g_i$, where $g_i\in\C[x_1,\dots,x_m]_{k+1+i}$ for all $i\in\{0,\dots,d-k-1\}$. In particular, we consider the point $[f]=[x_0^d+g]\in \N_{[x_0^d]}^k(\nu_m^d,Q_\mBW)$ for some generic $g\in\C[x_1,\dots,x_m]_d$. The $k$th-order eigenvectors of $f$ are the vectors $x\in\C^{m+1}$ such that $\rank M_k(x,f)\le 1$, where $M_k(x,f)$ is defined in \eqref{eq: matrix defining order-k eigenvectors}. Let $\widetilde{x}=(x_1,\dots,x_m)$. For our choice of $f$,
\[
M_k(x,f) =
\left(
\begin{array}{c|ccc|c|ccc|c}
    \frac{d!}{(d-k)!}x_0^{d-k} & 0 & \cdots & 0 & \cdots & 0 & \cdots & 0 & \nabla_kg(\widetilde{x}) \\
    x_0^d & x_0^{d-1}x_1 & \cdots & x_0^{d-1}x_m & \cdots & x_0x_1^d & \cdots & x_0x_m^d & \widetilde{x}^k
\end{array}
\right)^\mT
\]
Observe that the block matrix with rows $\nabla_kg(\widetilde{x})$ and $\widetilde{x}^k$ coincides with $A_k(\widetilde{x},g)$. Since $g\in\C[x_1,\dots,x_m]_d$ is generic and $(m,k)\neq(1,d-1)$, then $\rank M_k(\widetilde{x},g)\le 1$ if and only if $\widetilde{x}=0$. This implies that the unique solution of $\rank M_k(x,f)\le 1$, or the unique $k$th-order eigenvector of $f$, is $x=e_0$. Therefore we have found a point $[f]\in\DL_k(\nu_m^d,Q_\mBW)$ such that $\varphi_{2,k}^{-1}([f])=\{([x_0^d],[f])\}$, hence $\varphi_{2,k}$ is birational over its image $\DL_k(\nu_m^d,Q_\mBW)$. The last part of the statement descends by Proposition \ref{prop: kth order ED degree BW}.

Now we prove the last part of the statement. First, let $(m,k)=(1,d-1)$ with $d\ge 1$ and consider a pair $(\nu_1^d,Q)$ in general $(d-1)$-osculating position. In this case $\image j_{d-1}=\PP^{d-1}(\nu_1^d)=\OO_{\P^1}(1)^{\oplus d}$, hence $\dim\ker j_{d-1}=1$, in particular $\ker j_{d-1}=\OO_{\P^1}(a)$ for some $a\in\Z$. To compute $a$, we use the sequence \eqref{eq: k-jet exact sequence} which in this case simplifies to
\[
0 \to \OO_{\P^1}(a) \to \OO_{\P^1}^{\oplus(d+1)} \to \OO_{\P^1}(1)^{\oplus d} \to 0\,,
\]
from which we obtain the identity $0=c_1(\OO_{\P^1}^{\oplus(d+1)})=c_1(\OO_{\P^1}(a))+c_1(\OO_{\P^1}(1)^{\oplus d})$, giving $a=-d$. As a consequence $\EE_{d-1}(\nu_1^d,Q)=(\ker j_{d-1})^\vee\oplus\OO_{\P^1}(d)=\OO_{\P^1}(d)^{\oplus 2}$, in particular the tautological line bundle $\OO_{\EE_{d-1}(\nu_1^d,Q)^\vee}(1)$, whose sections are given by $H^0(X,(\ker j_{d-1})^\vee\oplus \OO_X(1))=H^0(\P^1,\OO_{\P^1}(d)^{\oplus 2})$, is very ample, therefore the morphism $\varphi_{2,d-1}\colon\P(\EE_{d-1}(\nu_1^d,Q)^\vee)=\DC_{d-1}(\nu_1^d,Q)\to\DL_{d-1}(\nu_1^d,Q)$, induced by $\OO_{\EE_{d-1}(\nu_1^d,Q)^\vee}(1)$, is birational.

Finally we consider the case $(m,k)\neq(1,d-1)$ with $k\le d$. Consider the open subset $\mathcal{U}\subseteq\P(\C[x_0,\dots,x_m]_2)$ of nondegenerate quadratic forms on $\C^{m+1}$ and the incidence variety
\[
\Sigma \coloneqq \overline{\left\{([v],[u],[Q]) \mid \text{$[Q]\in\mathcal{U}$ and $[v]\in \nu_m^d(\P^m)$ is critical of order $k$ for $d_u^2$}\right\}}
\]
together with the morphism $\psi_{23}\colon\Sigma\to\P^n\times\P(\C[x_0,\dots,x_m]_2)$ induced by the projection of $\Sigma$ onto the last two components. For any $[Q]\in\mathcal{U}$ and $[u]\in\DL_k(\nu_m^d,Q)$, we have $\psi_{23}^{-1}([u],[Q])=\varphi_{2,k}^{-1}([u])$. From the first part of the proof we derive that, if $(m,k)\neq (1,d-1)$, then $\psi_{23}^{-1}([u],[Q_\mBW])$ is a point for a generic $[u]\in\DL_k(\nu_m^d,Q_\mBW)$, in particular $\psi_{23}$ is birational. As a consequence, if we define
\[
\mathcal{V}\coloneqq\bigcup_{\substack{Q\in\mathcal{U}\\\text{$\varphi_{2,k}$ birational}}}\DL_k(\nu_m^d,Q)\times\{[Q]\}\subseteq\P^n\times\P(\C[x_0,\dots,x_m]_2)\,,
\]
then $\mathcal{V}\neq\emptyset$ because $\DL_k(\nu_m^d,Q_\mBW)\times\{[Q_\mBW]\}\subseteq\mathcal{V}$. Furthermore $\overline{\mathcal{V}}=\psi_{23}(\Sigma)$, or equivalently $\Sigma=\psi_{23}^{-1}(\overline{\mathcal{V}})$. Now consider the projection $\psi_3\colon\Sigma\to\P(\C[x_0,\dots,x_m]_2)$, which is dominant because $\psi_3^{-1}([Q])\cong\DC_k(\nu_m^d,Q)$ for all $Q\in\mathcal{U}$. Then $\psi_3(\overline{\mathcal{V}})=\P(\C[x_0,\dots,x_m]_2)$, which is equivalent to say that, for a generic $[Q]\in\P(\C[x_0,\dots,x_m]_2)$, the morphism $\varphi_{2,k}$ is birational. The last formula \eqref{eq: generic degree data locus Veronese} for the degree of the data locus $\DL_k(\nu_m^d,Q)$ is obtained from Corollary \ref{corol: generic kth ED degree Veronese}.
\end{proof}

\subsection{Proof of Theorem \ref{thm: generic ED degree general polynomial map}}\label{proof2}

The proof of Theorem \ref{thm: generic ED degree general polynomial map} requires the following technical lemma.

\begin{lemma}\label{lem: technical lemma}
Let $\P(V)\subseteq\P(W)$ be projective spaces of dimensions $n$ and $r$, respectively. Consider the morphism $f=\pi\circ f'\colon X\to\P(V)$, where $f'\colon X\to\P(W)$ has generic $k$-osculating dimension $m_k$ and $\pi \colon \P(W) \dashrightarrow \P(V)$ is a generic projective linear map for some $m_k<n$. The following holds.
\begin{enumerate}
    \item $f$ has generic $k$-osculating dimension $m_k$ and
    \begin{equation}\label{eq: identity gDD f and f'}
    \gDD_k(f) = \gDD_k(f')\,.
    \end{equation}
    \item Let $Q\subseteq\P(V)$ be a quadric hypersurface that intersects $f(X)$ transversally, namely $f(X)\cap Q$ is smooth and disjoint from the singular locus of $X$. Then $Q$ induces a nonsingular quadric hypersurface $Q' \subseteq\P(W)$ that intersects $f'(X)$ transversally and
    \begin{equation}\label{eq: isomorphism between intersection data locus and data locus of projection}
    \DL_k(f, Q) \cong \DL_k(f',Q') \cap \P(V)\,.
    \end{equation}
\end{enumerate}
\end{lemma}
\begin{proof}
Consider the $k$th jet morphisms $j_k \colon V\otimes \OO_X \to \PP^k(f)$ and $j_k' \colon W\otimes\OO_X \to \PP^k(f')$, associated with $f$ and $f'$, respectively. Choose bases $(\sigma_0, \dots, \sigma_n)$ of $V$ and $(\tau_0, \dots, \tau_r)$ of $W$, and write:
\[
f(p) = (\sigma_0(p), \dots, \sigma_n(p))\,,\quad f'(p) = (\tau_0(p), \dots, \tau_r(p))\,.
\]
The generic $k$th osculating dimension of $f$ (and similarly $f'$) corresponds to the rank of the matrices
\begin{align*}
A^{(k)}_p(f) &\coloneqq 
\left(
\begin{array}{c|c|c|c}
j_{k,p}(\sigma_0)^\mT & j_{k,p}(\sigma_1)^\mT & \cdots & j_{k,p}(\sigma_n)^\mT 
\end{array}
\right), \\
A^{(k)}_p(f') &\coloneqq 
\left(
\begin{array}{c|c|c|c}
j_{k,p}(\tau_0)^\mT & j_{k,p}(\tau_1)^\mT & \cdots & j_{k,p}(\tau_r)^\mT 
\end{array}
\right).
\end{align*}
Since $\pi$ is generic, the columns of $A^{(k)}_p(f)$ are generic linear combinations of those of $A^{(k)}_p(f')$. Hence
\[
\rank(A^{(k)}_p(f)) = \min\{n, m_k\} = m_k
\]
for a generic $p\in X$ and $f$ has generic $k$-osculating dimension $m_k$. Furthermore, this shows that at a generic point $p\in X$
\begin{equation}\label{eq: projection tangent}
    \pi(\T_p^k(f'))=\T_p^k(f)\,.
\end{equation}
In \cite{piene2022higher} and \cite{piene1978polar} Piene proved that $\mu_{k,j}(f) = \mu_{k,j}(f')$. This yields the identity \eqref{eq: identity gDD f and f'}.

We now show part $(2)$. Consider the linear map $\Pi \colon W \to V$ associated with $\pi$ and let $\Lambda = \P(\ker\Pi) \subseteq\P(W)$. In particular $\Lambda=\emptyset$ if $n=r$, otherwise $\dim\Lambda=r-n-1$. Furthermore $\Lambda\cap f'(X)=\emptyset$ by genericity of $\pi$.

Let $Q\subseteq\P(V)$ be a nonsingular quadric hypersurface that intersects $f(X)$ transversally. Denote by $\langle\,,\,\rangle_Q\colon V\times V\to\C$ the nondegenerate symmetric bilinear form on $V$ associated with $Q$. We define a nonsingular quadric $Q'\subseteq\P(W)$ whose associated nondegenerate symmetric bilinear form $\langle\,,\,\rangle_{Q'}\colon W\times W\to\C$ extends $\langle\,,\,\rangle_Q$ over $W$ and yields the orthogonal decomposition $W=\ker\Pi\oplus\image\Pi=\ker\Pi\oplus V$, in particular $\langle x,y\rangle_{Q'}=0$ for all $x\in\ker\Pi$ and $y\in V$. We denote by $(\,)^\perp$ and $(\,)^{\perp'}$ the orthogonal complements in $\P(V)$ and $\P(W)$ associated with $Q$ and $Q'$, respectively. Consider a generic point $p\in X$ such that \eqref{eq: projection tangent} holds. We show that
\begin{equation}\label{eq: intersection of normal spaces}
\T^k_p(f')^{\perp'} \cap \P(V) = \T^k_p(f)^{\perp}\,.
\end{equation}
Let $[x]\in \T^k_p(f')^{\perp'} \cap \P(V)$, hence $x\in V$ and $\langle x,y\rangle_{Q'}=0$ for all $[y]\in \T^k_p(f')^{\perp'}$. Writing $y=y'+\Pi(y)$, where $y'\in\ker\Pi$, then
$0 = \langle x,y\rangle_{Q'} = \langle x,y'\rangle_{Q'} + \langle x,\Pi(y)\rangle_{Q'} = \langle x,\Pi(y)\rangle_Q$,
therefore $\langle x,\Pi(y)\rangle_Q=0$ for every $[y]\in \T^k_p(f')^{\perp'}$, or equivalently $[x]\in (\pi(\T_p^k(f')))^\perp=\T_p^k(f)^\perp$. This shows the inclusion $\T^k_p(f')^{\perp'} \cap \P(V) \subseteq \T^k_p(f)^{\perp}$. The other inclusion is based on a dimension count.

The genericity assumptions ensure that, up to modifying $Q$ and the linear map $\pi$, the quadric hypersurface $Q'$ intersects $f'(X)$ transversally. It follows that for generic $p\in X$ the map $\pi$ induces an isomorphism:
\begin{equation}\label{eq: identification normal}
\pi\colon \N_p^k(f',Q') \cap \P(V)\xrightarrow{\sim}\N_p^k(f,Q)\,.    
\end{equation}

Consider the morphism $\varphi_{2,k}(f',Q')\colon\DC_k(f',Q')\to\DL_k(f',Q')$ given in Definition \ref{def: distance locus}.
The generic pointwise identification \eqref{eq: identification normal} lifts canonically to a birational map $\pi^*$ between the closure defining $\DC_k(f,Q)$ and the preimage of $\DL(f',Q')\cap \P(V)$ under $\varphi_{2,k}$:
\[
\pi^*\colon \DC_k(f,Q)\dashrightarrow\varphi_{2,k}(f',Q')^{-1}(\DL(f',Q')\cap \P(V))\,.
\]
 
Let $\alpha\colon \DL(f,Q)\dashrightarrow\DL(f',Q)\cap\P(V)$ be the induced rational map on the images of these projections. 
By tracing the commutative diagram
\[
\begin{tikzcd}
\DC_k(f,Q)\arrow[dashed, r, "\pi^*" above] \arrow[swap,->>]{d}{\varphi_{2,k}(f,Q)} & \varphi_{2,k}(f',Q')^{-1}(\DL(f',Q')\cap \P(V)) \arrow[->>]{d}{\varphi_{2,k}(f',Q')} \\
\DL(f,Q) \arrow[dashed]{r}{\alpha} & \DL(f',Q)\cap\P(V)
\end{tikzcd}
\]
one verifies that  $\alpha$ is a birational map. It follows that $\DL(f,Q)$ and  $\DL(f',Q)\cap\P(V)$ have the same dimension. The birational map $\alpha$ is defined (by construction) by restricting the generic linear projection $\pi.$ Again the genericity of $\pi$ implies that the degrees are preserved. 
\end{proof}

In the following example, we highlight the isomorphism of Lemma \ref{lem: technical lemma}(2).

\begin{example}\label{ex: second order DD quartic}
Let $[t_0:t_1]$ be homogeneous coordinates for $\P^1$. Let $\P(V)\subseteq\P^4$ be the hyperplane of equation $u_4=0$. Consider the morphism $f=[f_0:\cdots:f_3]\colon\P^1\to\P(V)\cong\P^3$ defined over the affine patch $\{t_0\neq 0\}$ with local coordinate $t=\frac{t_1}{t_0}$ as
\[
t\mapsto\left(t(1+t^2+t^3),(1-t)^2(1+t),-(1+t^2)(1-t-t^2),-t(1-t)(1+t)^2\right)\,.
\]
The closure of the image of the previous map is a nonsingular curve of degree $4$ in $\P^3$. Choosing $k=2$, then the right kernel of $A_p^{(2)}(f)$ is one-dimensional and is generated by the column vector
\[
\begin{pmatrix}
(1+t^2)(1-6\,t-10\,t^2+6\,t^3+t^4)\\
-(1-8\,t)(1+t)^2\\
-(1-6\,t-15\,t^2-8\,t^3-3\,t^4+6\,t^5+t^6)\\
(1-6\,t-9\,t^2-8\,t^3+6\,t^4)
\end{pmatrix}\,.
\]
If $Q=Q_\mED$ is the nonsingular quadric threefold in $\P^3$ of equation $\sum_{i=0}^3 u_i^2=0$, then the $2$nd-order normal space $\N_p^2(f,Q_\mED)$ at $p$ is the projective line in $\P^3$ spanned by the column vector above and by $f(p)$. The polynomial components defining $f$ have been chosen generic enough so that $(f,Q_\mED)$ is in general $2$-osculating position. Applying Theorem \ref{thm: generic ED degree general polynomial map} with $d=4$ and $k=2$, we have that $\DL_2(f,Q_\mED)$ is a surface of degree $\sum_{i=0}^1\binom{3}{i}2^i4^{1-i}=4+6=10$ in $\P^3$.

Now consider the projection $\pi\colon\P^4\dasharrow\P(V)$. In particular the vertex of $\pi$ is the point $\Lambda=[0:0:0:0:1]$. Define the new morphism $f'=[f_0':\cdots:f_4']\colon\P^1\to\P^4$ such that $f_i'=f_i$ for all $i\in\{0,\dots,3\}$ and $f_4'(t)=-(1+t)(1-t+t^3)$ with the same choice of $t$ as before. Then $f=\pi\circ f'$ and $f'(\P^1)$ is a rational normal curve in $\P^4$, isomorphic to $f(\P^1)$.
One verifies that the right kernel of $A_p^{(2)}(f')$ is two-dimensional and is generated by the columns of
\[
\begin{pmatrix}
-1+3\,t-9\,t^2-t^3 & -1+9\,t+3\,t^2+t^3\\
1+3t & -1\\
1+6\,t^2+t^3 & -t(6+3\,t+t^2)\\
-1-3\,t^2-t^3 & 3\,t(1+t)\\
t(3-6\,t-t^2) & -1+6\,t+3\,t^2
\end{pmatrix}\,.
\]
Choosing $Q'=Q_\mED'=\V(\sum_{i=0}^4 u_i^2)$, then $Q_\mED'$ extends $Q_\mED$ and $(f',Q_\mED')$ is in general $2$-osculating position. Furthermore $\DL_2(f',Q_\mED')$ is a threefold of degree $10$ in $\P^4$, such that $\DL_2(f',Q_\mED')\cap\P(V)\cong \DL_2(f,Q_\mED)$. The two surfaces $\DL_2(f',Q_\mED')\cap\P(V)$ and $\DL_2(f,Q_\mED)$ are displayed in Figure \ref{fig: isomorphism data loci}.\hfill$\diamondsuit$

\begin{figure}[ht]
\centering
\begin{overpic}[width=0.4\textwidth]{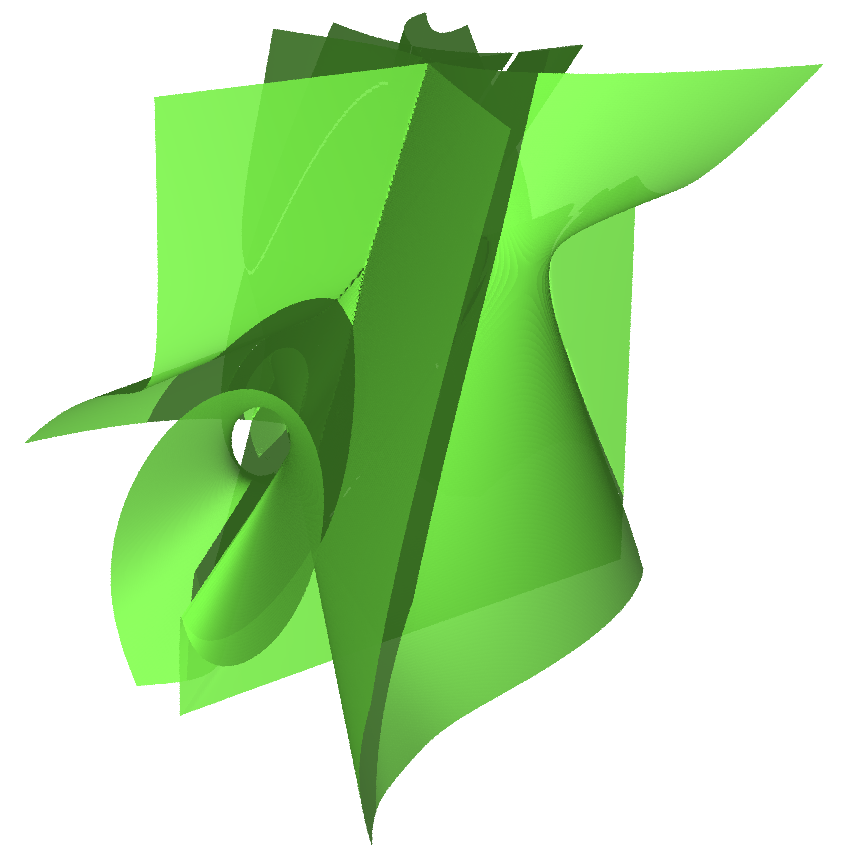}
\end{overpic}
\begin{overpic}[width=0.4\textwidth]{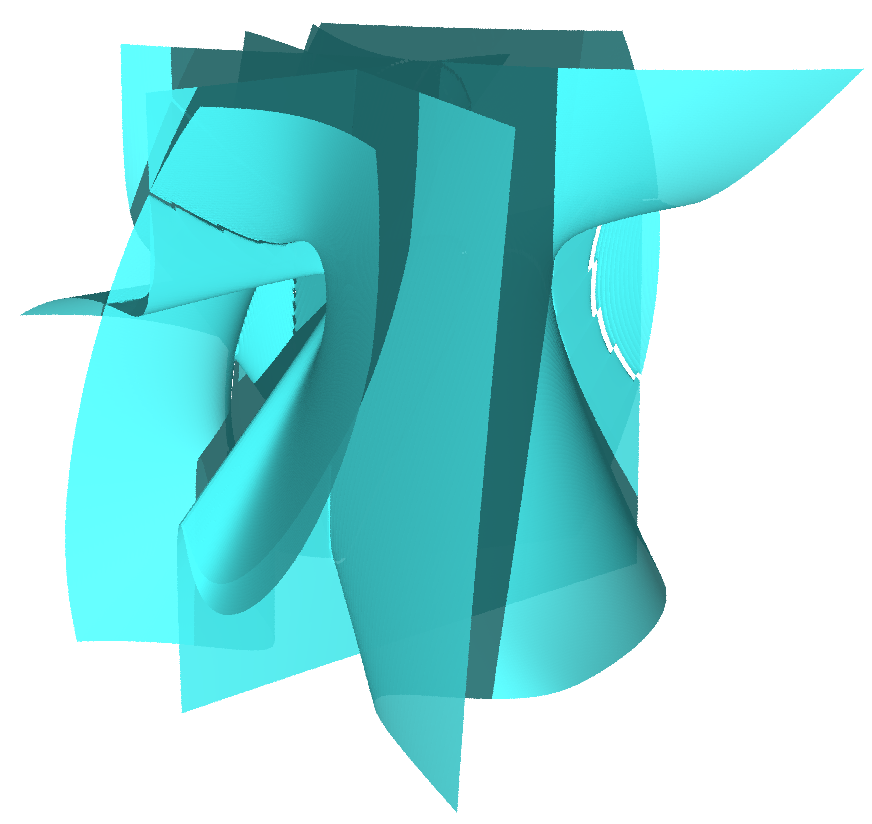}
\end{overpic}
\caption{The projective surfaces $\DL_2(f',Q_\mED')\cap\P(V)$ and $\DL_2(f, Q_\mED)$ in the affine chart $\{u_3=1\}\subseteq\P^3$.}\label{fig: isomorphism data loci}
\end{figure}
\end{example}

We are now ready to prove the main result of this section.

\begin{proof}[Proof of Theorem \ref{thm: generic ED degree general polynomial map}]
Let $f_0,\dots,f_n$ be $n+1$ generic polynomials in $\C[x_0,\dots,x_m]_d$ with $n>\binom{m+k}{k}-1$ and consider a nonnegative integer $k\le d$. 
Consider the morphism $f\colon\P^m\to\P^n$ defined by $f(p)=(f_0(p),\dots,f_n(p)).$ 
Since the polynomials $f_0, \dots, f_n$ are generic, the morphism $f$ can be expressed as the composition $f = \pi \circ \nu_m^d$, where $\nu_m^d \colon \P^m \hookrightarrow \P^{\binom{m+d}{d} - 1}$ denotes the Veronese embedding, and $\pi \colon \P^{\binom{m+d}{d} - 1} \dashrightarrow \P^n$ is a generic linear map.
Applying Lemma \ref{lem: technical lemma}(1) with $f'=\nu_m^d$ and $r=\binom{m+d}{d}-1$, we obtain that $f$ has generic $k$-osculating dimension $\binom{m+k}{k}-1$ and the identity $\gDD_k(f) = \gDD_k(\nu_m^d)$.
Furthermore, by Lemma \ref{lem: technical lemma}(2) there exists a nonsingular quadric $Q' \subseteq\P^{\binom{m+d}{d} - 1}$ that is transversal to $\nu_m^d(\P^m)$ and such that $\DL_k(f,Q) \cong \DL_k(\nu_m^d, Q') \cap \P^n$, in particular $\deg\DL_k(f,Q) = \deg\DL_k(\nu_m^d,Q')$. Consequently, the degrees of the morphisms $\varphi_{2,k}(f,Q) \colon \DC_k(f,Q) \to \DL_k(f,Q)$ and $\varphi_{2,k}(\nu_m^d, Q') \colon \DC_k(\nu_m^d, Q') \to \DL_k(\nu_m^d, Q')$ are equal. The degree of $\varphi_{2,k}(\nu_m^d, Q')$ is one by Theorem \ref{thm: projection data locus Veronese BW is either birational or of degree 2}, where we are assuming that $k\le d$. The final identity follows by \eqref{eq: explicit EDD k polar classes Veronese}.
\end{proof}

\begin{remark}
First, observe that assuming $n=\binom{m+k}{k}-1$ implies that $f$ is isomorphic to the $k$th Veronese embedding of $\P^m$ by Proposition \ref{prop: linearly normal} and $\DL_k(f,Q)=f(X)$.
Secondly, if the hypersurface $Q$ is not transversal to $f(\P^m)$, then the morphism $\varphi_{d,k}$ might fail to be birational, as shown in the second part of the illustrative example in the introduction and at the right-hand side of Figure~\ref{fig: EDD 2nd order}, where we chose the Bombieri-Weyl inner product in the space of real binary cubics.
\end{remark}

\section{Higher-order distance degrees of regular embeddings}\label{sec: formulas for k-regular embeddings}

In this subsection, we apply the formula \eqref{eq: kth order ED degree sum of kth order polar degrees} to compute the generic $k$th-order distance degree of a $k$-regular embedding $f\colon X\hookrightarrow\P^n$. The $k$-regularity of $f$ implies that $\mu_{k,i}(f)$ equals the degree of $c_i(\PP^k(f))$ for all $i$ and $k$.

In Corollary \ref{corol: generic kth ED degree Veronese}, we already computed the generic higher-order distance degree of the Veronese embedding $\nu_m^d\colon\P^n\hookrightarrow\P^{\binom{m+d}{d}-1}$, using that its jet bundles split as a sum of line bundles. More in general, given a $k$-regular embedding $f\colon X\hookrightarrow\P^n$, the short exact sequences
\begin{equation}\label{eq: two short exact sequences}
\begin{gathered}
    0\to \Omega_X\otimes\OO_X(1) \longrightarrow \PP^1(f) \longrightarrow \OO_X(1) \to 0\\
    0\to \mathrm{Sym}^k\Omega_X\otimes\OO_X(1) \longrightarrow \PP^k(f) \longrightarrow \PP^{k-1}(f) \to 0\quad\forall\,k\ge 2
\end{gathered}
\end{equation}
allow us to compute the Chern classes of the vector bundles $\PP^k(f)$ via a recursive formula, provided that the Chern classes of $\mathrm{Sym}^k\Omega_X\otimes\OO_X(1)$ and of $\mathrm{Sym}^k\Omega_X$ are known. Regarding the first ones, the Chern classes of the tensor product between a vector bundle $\FF$ of rank $r$ and a line bundle $\LL$ on a nonsingular variety $X$ are given by the following formula.

\begin{equation}\label{lem: chern classes tensor product with line bundle}
c_i(\FF\otimes\LL) = \sum_{j=0}^i\binom{r-j}{i-j}c_j(\FF)c_1(\LL)^{i-j}\,.
\end{equation}

In the following, we write the total Chern classes of $\TT_X$ and $\OO_X(1)$ as $c(\OO_X(1))=1+L$ and $c(\TT_X)=c(X)=1+\sum_{i=1}^mc_i$. Note that $c_i(\Omega_X)=(-1)^ic_i$ for all $i$. From the first exact sequence in \eqref{eq: two short exact sequences} we get $c(\PP^1(f))=c(\Omega_X\otimes\OO_X(1))\cdot c(\OO_X(1))$. Using Lemma \ref{lem: chern classes tensor product with line bundle} we get for all $i\ge 1$
\begin{align*}
    c_i(\PP^1(f)) &= c_i(\Omega_X\otimes\OO_X(1))+c_{i-1}(\Omega_X\otimes\OO_X(1))\cdot L\\
    &= \sum_{j=0}^i\binom{m-j}{i-j}c_j(\Omega_X)\cdot L^{i-j}+\sum_{j=0}^{i-1}\binom{m-j}{i-1-j}c_j(\Omega_X)\cdot L^{i-j}\\
    &= c_i(\Omega_X)+\sum_{j=0}^{i-1}\left[\binom{m-j}{i-j}+\binom{m-j}{i-1-j}\right]c_j(\Omega_X)\cdot L^{i-j}\\
    &= (-1)^ic_i+\sum_{j=0}^{i-1}\binom{m+1-j}{i-j}(-1)^jc_j\cdot L^{i-j} \\
    &= \sum_{j=0}^i\binom{m+1-j}{i-j}(-1)^jc_j\cdot L^{i-j}\,,
\end{align*}
which corresponds to \cite[\S3]{holme1988geometric}. Then, using the second exact sequence in \eqref{eq: two short exact sequences}, we have $c(\PP^k(f))=c(\mathrm{Sym}^k\Omega_X\otimes\OO_X(1))\cdot c(\PP^{k-1}(f))$, hence
\begin{align*}
    c_i(\PP^k(f)) &= \sum_{\ell=0}^ic_\ell(\mathrm{Sym}^k\Omega_X\otimes\OO_X(1))c_{i-\ell}(\PP^{k-1}(f))\\
    &= \sum_{\ell=0}^i\left[\sum_{j=0}^\ell\binom{\binom{m+k-1}{k}-j}{\ell-j}c_j(\mathrm{Sym}^k\Omega_X)\cdot L^{\ell-j}\right]c_{i-\ell}(\PP^{k-1}(f))\,,
\end{align*}
where we used the fact that $\rank\mathrm{Sym}^k\Omega_X=\binom{m+k-1}{k}$ for all $k\ge 1$. Summing up, we have the following recursive formula:
\begin{equation}\label{eq: recursive formula}
\begin{cases}
    c_i(\PP^1(f)) = \sum_{j=0}^i\binom{m+1-j}{i-j}(-1)^jc_j\cdot L^{i-j}\\
    c_i(\PP^k(f)) = \sum_{\ell=0}^i\left[\sum_{j=0}^\ell\binom{\binom{m+k-1}{k}-j}{\ell-j}c_j(\mathrm{Sym}^k\Omega_X)\cdot L^{\ell-j}\right]c_{i-\ell}(\PP^{k-1}(f))\,.
\end{cases}
\end{equation}

We implemented a \verb|Macaulay2| code \cite{GS} that computes the Chern classes $c_i(\PP^k(f))$ for any input $k$ and $m$ using \eqref{eq: recursive formula}; see \cite{github}. Our code computes the generic $k$th-order distance degree of any $k$-regular embedding $f\colon X\hookrightarrow\P^n$ of a nonsingular projective variety $X$ of dimension $m$.

In the following, we derive closed formulas for the classes $c_i(\PP^k(f))$ for $m\in[3]$, using the following lemma, whose proof is a direct application of the splitting principle \cite[Remark 3.2.3]{fulton1998intersection}. More involved identities can be derived for larger $m$.

\begin{lemma}\label{lem: chern classes symmetric power vector bundle small rank}
Let $\FF$ be a vector bundle on a nonsingular surface $X$. Then for any $k\ge 1$,
\[
c(\mathrm{Sym}^k\FF) =
\begin{cases}
    1+k\,c_1(\FF) & \text{if $\rank\FF=1$}\\[3pt]
    1+\binom{k+1}{2}c_1(\FF)+\binom{k+1}{3}\frac{3k+2}{4}c_1^2(\FF)+\binom{k+2}{3}c_2(\FF) & \text{if $\rank\FF=2$}\\[3pt]
    1+\binom{k+2}{3}c_1(\FF)+\frac{1}{2}\binom{k+3}{3}\binom{k+1}{3}c_1^2(\FF)+\binom{k+3}{4}c_2(\FF) & \\[2pt]
    +\binom{k+3}{5}\frac{5\,k^4+20\,k^3-5\,k^2-50\,k-12}{54}c_1^3(\FF) & \\[2pt]
    +\binom{k+3}{5}\frac{5\,k^2+20\,k+6}{6}c_1(\FF)c_2(\FF)+\binom{k+3}{4}\frac{2\,k+3}{5}c_3(\FF) & \text{if $\rank\FF=3$.}
\end{cases}
\]
\end{lemma}

\begin{proposition}\label{prop: gEDD k-regular curve}
Consider a $k$-regular embedding $f\colon C\hookrightarrow\P^n$ of a nonsingular projective curve $C$. Write $c(X)=1+c_1$ and $c(\OO_X(1))=1+L$. Then
\begin{equation}\label{eq: gEDD k-regular curve}
    \gDD_k(f) = \int_C(k+2)L-\binom{k+1}{2}c_1\,.
\end{equation}
\end{proposition}
\begin{proof}
We claim that, for all $k\ge 1$,
\begin{equation}\label{eq: identity total chern class J_k curve}
c(\PP^k(f)) = 1+(k+1)L-\binom{k+1}{2}c_1\,.
\end{equation}
Assuming the claim is true, then
\[
\gDD_k(f) = \mu_{k,0}(f)+\mu_{k,1}(f) = \int_C L+c_1(\PP^k(f)) = \int_C (k+2)L-\binom{k+1}{2}c_1\,,
\]
which is the desired formula.
To prove the claim, consider the recursive formula \eqref{eq: recursive formula}. The first equation yields $c(\PP^1(f)) = 1 + 2L-c_1$, which is identity \eqref{eq: identity total chern class J_k curve} for $k=1$. Now assume \eqref{eq: identity total chern class J_k curve} true for $k-1$. Using the second exact sequence in \eqref{eq: two short exact sequences}, we obtain that
\begin{align*}
c(\PP^k(f)) &= c(\mathrm{Sym}^k\Omega_C\otimes\OO_C(1))c(\PP^{k-1}(f)) = (1+L+c_1(\mathrm{Sym}^k\Omega_C))\left(1+kL-\binom{k}{2}c_1\right)\\
&= (1+L-kc_1)\left(1+kL-\binom{k}{2}c_1\right) = 1+(k+1)L-\binom{k+1}{2}c_1\,,
\end{align*}
thus proving the identity \eqref{eq: identity total chern class J_k curve} at the step $k$. This completes the proof.
\end{proof}

\begin{example}\label{ex: gDD Veronese embedding projective line}
We apply Proposition \ref{prop: gEDD k-regular curve} for the Veronese embedding $\nu_1^d\colon\P^1\hookrightarrow\P^d$, which is $k$-regular for any $k\le d$. Let $h$ denote the class of a point in $\P^1$. In particular $c(\P^1)=(1+h)^2=1+2h$ and $L=dh$, hence $\gDD_k(\nu_1^d) = \int_{\P^1}(k+2)dh-\binom{k+1}{2}2h = (k+2)d-k(k+1)$.\hfill$\diamondsuit$
\end{example}

\begin{proposition}\label{prop: gEDD k-regular surface}
Consider a $k$-regular embedding $f\colon S\hookrightarrow\P^n$ of a nonsingular projective surface $S$. Write $c(S)=1+c_1+c_2$ and $c(\OO_S(1))=1+L$. Then
\begin{equation}\label{eq: gEDD k-regular surface}
\gDD_k(f) = \int_X\alpha_1\,L^2+\alpha_2\,c_1L+\alpha_3\,c_1^2+\alpha_4\,c_2\,,
\end{equation}
where
\begin{align}\label{eq: coefficients alpha}
\begin{split}
    \alpha_1 &= 1+\binom{k+2}{2}+3\binom{k+3}{4}\,,\quad\alpha_2 = -\binom{k+2}{2}\binom{k+2}{3}\,,\\
    \alpha_3 &= \frac{1}{2}\binom{k+1}{3}\binom{k+3}{3}\,,\quad\alpha_4 = \binom{k+3}{4}\,.
\end{split}
\end{align}
\end{proposition}
\begin{proof}
We claim that, for all $k\ge 1$,
\begin{equation}\label{eq: identity total chern class J_k surface}
\resizebox{.9\textwidth}{!}{$\begin{aligned}c_1(\PP^k(f)) &= \binom{k+2}{2}L-\binom{k+2}{3}c_1\\
c_2(\PP^k(f)) &= \frac{1}{2}\binom{k+1}{2}\binom{k+3}{2}L^2-2k\binom{k+3}{4}c_1L+\frac{1}{2}\binom{k+1}{3}\binom{k+3}{3}c_1^2+\binom{k+3}{4}c_2\,.\end{aligned}$}
\end{equation}
Assuming the claim is true, then
\[
\resizebox{\textwidth}{!}{$\begin{aligned}
\gDD_k(f) &= \mu_{k,0}(f)+\mu_{k,1}(f)+\mu_{k,2}(f) = \int_X L^2+c_1(\PP^k(f))L+c_2(\PP^k(f))\\
&= \int_X \left[1+\binom{k+2}{2}+3\binom{k+3}{4}\right]L^2 -\left[\binom{k+2}{3}+2k\binom{k+3}{4}\right]c_1L+\frac{1}{2}\binom{k+1}{3}\binom{k+3}{3}c_1^2+\binom{k+3}{4}c_2\,,
\end{aligned}$}
\]
which, after simplifying, is equal to the desired formula.
To prove the claim, consider the recursive formula \eqref{eq: recursive formula}. The first equation yields the identities $c_1(\PP^1(f))=3L-c_1$ and $c_2(\PP^1(f))=3L^2-2c_1L+c_2$, thus proving \eqref{eq: identity total chern class J_k surface} for $k=1$. Now assume \eqref{eq: identity total chern class J_k surface} true for $k-1$. Observe that $\mathrm{Sym}^k\Omega_S$ has rank $k+1$ for all $k\ge 1$.
Using Lemma \ref{lem: chern classes symmetric power vector bundle small rank}, we get the Chern classes
\begin{equation}
    c_1(\mathrm{Sym}^k\Omega_S) = -\binom{k+1}{2}c_1\,,\quad c_2(\mathrm{Sym}^k\Omega_S) = \frac{3k+2}{4}\binom{k+1}{3}c_1^2+\binom{k+2}{3}c_2\,.
\end{equation}
Applying Lemma \ref{lem: chern classes tensor product with line bundle}, we obtain that
\[
\resizebox{\textwidth}{!}{$\begin{aligned}
    c(\mathrm{Sym}^k\Omega_S\otimes\OO_S(1)) &= \sum_{i=0}^{k+1}c_i(\mathrm{Sym}^k\Omega_S\otimes\OO_S(1)) = \sum_{i=0}^{k+1}\sum_{j=0}^i\binom{k+1-j}{i-j}c_j(\mathrm{Sym}^k\Omega_S)\cdot L^{i-j}\\
    &= 1+(k+1)L+c_1(\mathrm{Sym}^k\Omega_S)+\binom{k+1}{2}L^2+kc_1(\mathrm{Sym}^k\Omega_S)L+c_2(\mathrm{Sym}^k\Omega_S)\\
    &= 1+(k+1)L-\binom{k+1}{2}c_1+\binom{k+1}{2}L^2-k\binom{k+1}{2}c_1L+\frac{3k+2}{4}\binom{k+1}{3}c_1^2+\binom{k+2}{3}c_2\,.
\end{aligned}$}
\]
Using the second exact sequence in \eqref{eq: two short exact sequences} and the induction step, we obtain that
\[
\resizebox{\textwidth}{!}{$\begin{aligned}
    c(\PP^k(f)) &= c(\mathrm{Sym}^k\Omega_S\otimes\OO_S(1))\cdot c(\PP^{k-1}(f))\\
    &= \left(1+(k+1)L-\binom{k+1}{2}c_1+\binom{k+1}{2}L^2-k\binom{k+1}{2}c_1L+\frac{3k+2}{4}\binom{k+1}{3}c_1^2+\binom{k+2}{3}c_2\right)\\
    &\quad\cdot\left(1+\binom{k+1}{2}L-\binom{k+1}{3}c_1+\frac{1}{2}\binom{k}{2}\binom{k+2}{2}L^2-2(k-1)\binom{k+2}{4}c_1L+\frac{1}{2}\binom{k}{3}\binom{k+2}{3}c_1^2+\binom{k+2}{4}c_2\right)\,.
\end{aligned}$}
\]
Expanding the product, one verifies that
\[
\resizebox{\textwidth}{!}{$\begin{aligned}
    c_1(\PP^k(f)) &= \left[(k+1)+\binom{k+1}{2}\right]L-\left[\binom{k+1}{2}+\binom{k+1}{3}\right]c_1 = \binom{k+2}{2}L-\binom{k+2}{3}c_1\,,\\
    c_2(\PP^k(f)) &= \left[(k+2)\binom{k+1}{2}+\frac{1}{2}\binom{k}{2}\binom{k+2}{2}\right]L^2\\
    &\quad-\left[k\binom{k+1}{2}+\binom{k+1}{2}^2+(k+1)\binom{k+1}{3}+2(k-1)\binom{k+2}{4}\right]c_1L\\
    &\quad+\left[\frac{3k+2}{4}\binom{k+1}{3}+\binom{k+1}{2}\binom{k+1}{3}+\frac{1}{2}\binom{k}{3}\binom{k+2}{3}\right]c_1^2+\left[\binom{k+2}{3}+\binom{k+2}{4}\right]c_2\,,
\end{aligned}$}
\]
which, after simplifying, correspond to the identities in \eqref{eq: identity total chern class J_k curve} at the step $k$.
\end{proof}

We also compute the generic $k$th-order distance degree of any $k$-regular embedding $f\colon X\hookrightarrow\P^n$ of a nonsingular projective threefold $X$. Its proof uses again the recursive formula \eqref{eq: recursive formula}, Lemma \ref{lem: chern classes symmetric power vector bundle small rank}, and is similar to Proposition \ref{prop: gEDD k-regular surface}; therefore, we omit it.

\begin{proposition}\label{prop: gEDD k-regular threefold}
Consider a $k$-regular embedding $f\colon X\hookrightarrow\P^n$ of a nonsingular projective threefold $X$. Write $c(X)=1+c_1+c_2+c_3$ and $c(\OO_X(1))=1+L$. Then
\begin{equation}\label{eq: gEDD k-regular threefold}
    \gDD_k(f) = \int_X \beta_1\,L^3+\beta_2\,c_1L^2+\beta_3\,c_1^2L+\beta_4\,c_2L+\beta_5\,c_1^3+\beta_6\,c_1c_2+\beta_7\,c_3\,,
\end{equation}
where
\begin{align}\label{eq: coefficients beta}
\begin{split}
    \beta_1 &= \frac{(k+4)(k^{2}+2\,k+3)(k^{6}+12\,k^{5}+58\,k^{4}+138\,k^{3}+157\,k^{2}+66\,k+216)}{1296}\\
    \beta_2 &= -\binom{k+3}{4}\frac{k^{6}+12\,k^{5}+58\,k^{4}+138\,k^{3}+157\,k^{2}+66\,k+72}{72}\\
    \beta_3 &= \binom{k+3}{5}\frac{k(k^{2}+6\,k+11)(5\,k^{3}+35\,k^{2}+90\,k+72)}{288}\\
    \beta_4 &= \binom{k+4}{5}\frac{k(k^{2}+6\,k+11)}{6}\\
    \beta_5 &= -\binom{k+3}{5}\frac{5\,k^{7}+65\,k^{6}+355\,k^{5}+931\,k^{4}+816\,k^{3}-1404\,k^{2}-3312\,k-1152}{3456}\\
    \beta_6 &= -\binom{k+4}{6}\frac{k^{3}+7\,k^{2}+18\,k+8}{4}\\
    \beta_7 &= -\binom{k+4}{5}\frac{k+2}{3}\,.
\end{split}
\end{align}
\end{proposition}

\begin{example}
We apply Propositions \ref{prop: gEDD k-regular surface} and \ref{prop: gEDD k-regular threefold} for the Segre-Veronese embedding $\nu_\bm^\bd$ defined in \eqref{eq: def Segre-Veronese embedding} and with $\bm=(1,\dots,1)\in\N^r$ and $r\in\{2,3\}$. Recall that $\nu_\bm^\bd$ is $k$-regular if and only if $k\le\min\bd$ by Corollary \ref{corol: when Segre-Veronese k-regular}. Let $h_i$ be the class of a point in the $i$th factor of $\P^\bm=(\P^1)^{\times r}$. In particular $L=\sum_{i=1}^r d_ih_i$ and $c(\P^\bm)=\prod_{i=1}^r(1+h_i)^2=\sum_{j=0}^r 2^j e_j(h_1,\dots,h_r)$, where $e_j$ is the $j$th elementary symmetric polynomial in its arguments, with $e_0\coloneqq 1$.

For $r=2$, the relations needed are $L^2 = 2d_1d_2h_1h_2$, $c_1L = 2(d_1+d_2)h_1h_2$, $c_1^2 = 8h_1h_2$, and $c_2 = 4h_1h_2$. Plugging these relations in \eqref{eq: gEDD k-regular surface} and simplifying, one verifies that
\[
\resizebox{\textwidth}{!}{$\begin{aligned}\gDD_k(\nu_\bm^\bd) = \frac{(k^2+k+2)(k^2+5k+8)}{4}d_1d_2-2\binom{k+2}{3}\binom{k+2}{2}(d_1+d_2)+\frac{4}{3}(2k^2+1)\binom{k+3}{4}\end{aligned}$}
\]
for any $k\le\min\{d_1,d_2\}$. Instead for $r=3$, using \eqref{eq: gEDD k-regular threefold}, one verifies that
\begin{equation}\label{eq: gDD 3 copies P1}
\gDD_k(\nu_\bm^\bd) =
2\left[3\,\beta_1\,e_3(\bd)+2\,\beta_2\,e_2(\bd)+2(2\,\beta_3+\beta_4)e_1(\bd)+24\,\beta_5+12\,\beta_6+4\,\beta_7\right]
\end{equation}
for any $k\le\min\{d_1,d_2,d_3\}$. For example, if $k=2\le\min\{d_1,d_2,d_3\}$, then $\gDD_2(\nu_\bm^\bd) = 8(132\,e_3(\bd)-115\,e_2(\bd)+108\,e_1(\bd)-110)$.\hfill$\diamondsuit$
\end{example}

\begin{example}\label{ex: formulas for complete intersection varieties}
Let $f\colon X\hookrightarrow\P^n$ be an embedded nonsingular complete intersection variety of dimension $m$. Let $d=(d_1,\dots,d_{n-m})$ be the vector of degrees of the polynomials defining $X$. Then
\[
c(X)=\frac{(1+L)^{n+1}}{(1+d_1L)\cdots(1+d_{n-m}L)} = \sum_{i=0}^m\left(\sum_{j=0}^i(-1)^j\binom{n+1}{i-j}e_j(d)\right)L^i\,,
\]
where $e_j(d)$ is the $j$th elementary symmetric function of $d$.
One can apply Propositions \ref{prop: gEDD k-regular curve} and \ref{prop: gEDD k-regular surface}, or more in general the recursive formula \ref{eq: recursive formula} to derive a formula for the generic $k$th-order distance degree of $f$ as a polynomial in $d$. For example, if $m=1$, then
\[
\gDD_k(f) = d_1\cdots d_{n-1}\left[k+2-\binom{k+1}{2}(n+1-d_1-\cdots-d_{n-1})\right]
\]
whenever $f$ is $k$-regular.\hfill$\diamondsuit$   
\end{example}

The formulas \eqref{eq: gEDD k-regular curve}, \eqref{eq: gEDD k-regular surface}, and \eqref{eq: gEDD k-regular threefold} have an interesting geometric interpretation in toric geometry. In the following, we adopt the notation used in \cite{Higher_duality_and_toric}. For an $m$-dimensional toric embedding $f\colon X\hookrightarrow\P^n$ with corresponding lattice polytope $P=P(f)\subseteq\R^m$, setting $L=c_1(\OO_X(1))$ one has (see \cite[Corollary 11.5]{danilov1978geometry})
\begin{equation}\label{eq: Chern classes toric}
    c_i=\sum_{F\subseteq P,\,\codim(F)=i}[F]\,,\quad \int_X c_iL^{m-i} = \sum_{F\subseteq P,\,\codim(F)=i}\Vol(F)
\end{equation}
for every $i\in\{0,\dots,m\}$, where $[F]$ denotes the class of the invariant subvariety of $X$ associated with the face $F$ of $P$, and $\Vol(F)$ indicates the lattice volume of $F$ measured with respect to the lattice induced by $\Z^m$ in the linear span of $F$, in particular $\Vol(F)$ is equal to $(\dim F)!$ times the Euclidean volume of $F$. Since the following formulas deal with toric varieties of dimension $m\le 3$, we also define
\[
\VV\coloneqq |\{\text{vertices of $P$}\}|\,,\quad\EE\coloneqq \sum_{\text{$\xi$ edge of $P$}
}\Vol(\xi)\,,\quad\FF\coloneqq \sum_{\text{$F$ facet of $P$}}\Vol(F)\,,\quad\PP\coloneqq\Vol(P)\,.
\]

\begin{corollary}\label{corol: gEDD k-regular toric surface}
Let $f\colon X\hookrightarrow\P^n$ be a $k$-regular toric embedding of a nonsingular toric surface $X$ and let $P=P(f)$ be the associated polygon. Then:
\begin{equation}\label{eq: gEDD k-regular toric surface}
\gDD_k(f) = \alpha_1\,\PP+\alpha_2\,\EE+(\alpha_4-\alpha_3)\VV+12\,\alpha_3\,,
\end{equation}
where the coefficients $\alpha_i(k)$ are displayed in \eqref{eq: coefficients alpha}.
\end{corollary}
\begin{proof}
Recall that $L=c_1(\OO_X(1))$. Firstly, using the relations in \eqref{eq: Chern classes toric}, one obtains that $\int_X L^2=\PP$, $\int_X c_1L=\EE,$ and $\int_X c_2=\VV.$ Secondly, Noether's Formula $\chi(\OO_X)=(K_X^2+\chi(X))/12$ \cite[IV,\S1]{griffiths1978principles} gives $\int_X c_1^2=\int_X K_X^2=12\chi(\OO_X)-\chi(X)=12\chi(\OO_X)-\VV$. Applying \cite[Corollary 7.4]{danilov1978geometry}, a consequence of Demazure vanishing for toric varieties \cite[Proposition 6, p. 564]{demazure1970sousgroupes}, we have $\chi(\OO_X)=1$. Hence $\int_Xc_1^2=12-\VV.$ The statement descends by \eqref{eq: gEDD k-regular surface} after plugging in the relations obtained before.
\end{proof}

The formula in \eqref{eq: gEDD k-regular toric surface} coincides with the sum of the degrees of the higher-order polar classes given in \cite[Example 8.2]{piene2022higher}.

\begin{example}
For $k=1$, the formula \eqref{eq: gEDD k-regular toric surface} simplifies to $\gDD(f)=7\,\PP-3\,\EE+\VV$, which agrees with \cite[Corollary 5.11]{DHOST} when $m=2$. We also point out that, for $k=2$, the formula \eqref{eq: gEDD k-regular toric surface} simplifies to $\gDD_2(f) = 22\,\PP-24\,\EE+60$, in particular it does not depend on $\VV$.\hfill$\diamondsuit$
\end{example}

Let $f\colon X\hookrightarrow\P^n$ be the toric embedding of a nonsingular toric threefold with associated lattice polytope $P$. Denote by $K_X=c_1(\bigwedge^3\Omega_X)$ the canonical divisor of $X$. Recall the notation $L=c_1(\OO_X(1))$. If the divisor $K_X+L$ is nef \cite[Definition 1.4.1]{lazarsfeld2017positivity}, then the corresponding polytope is $P^\circ\coloneqq\conv(\operatorname{int}(P)\cap\Z^3)$. We also define
\[
\FF_1\coloneqq \sum_{\text{$F$ facet of $P^\circ$}}\Vol(F)\,,\quad\PP_1\coloneqq\Vol(P^\circ)\,.
\]

We conclude this section with a classification of $k$-regular toric embeddings of nonsingular toric threefolds, with their respective generic $2$nd-order distance degrees. This extends the computation made in \cite[Example 8.3]{piene2022higher} for $2$-regular toric embeddings of nonsingular toric threefolds.

\begin{proposition}\label{corol: gEDD 2-regular toric threefold}
Let $f\colon X\hookrightarrow\P^n$ be a $k$-regular toric embedding of a nonsingular toric threefold. Then one of the following possibilities occurs:
\begin{enumerate}
\item $k=1$ and $\gDD(f)=\gDD_1(f) = 15\PP-7\FF+1/40\EE-\VV$
\item $k\in\{2,3\}$ and $(X,f)=(\P^3,\nu_3^d)$ with $d\in\{2,3\}$. In this case
\begin{equation}
    \gDD_k(f) =
    \begin{cases}
    8 & \text{if $k=2,d=2$}\\
    370 & \text{if $k=2,d=3$.}\\
    27 & \text{if $k=3,d=3$.}
    \end{cases}
\end{equation}

\item $k=2$ and $X=\P(\EE)$ for $\EE=\OO_{\P^1}(a)\oplus\OO_{\P^1}(b)\oplus\OO_{\P^1}(c)$ with $a\ge b\ge c\ge 2$ and $f$ is the embedding defined by $2\,\xi_\EE$. In this case
\begin{equation}\label{eq: gEDD 2-regular toric threefold case 2}
    \gDD_2(f)=162(a+b+c)-154\,.    
\end{equation}

\item Otherwise
\begin{equation}\label{eq: gEDD 2-regular toric threefold case 3}
\gDD_2(f) = \beta_1\,\PP+\beta_2\,\PP_1-\beta_3\,\FF+\beta_4\,\FF_1+\beta_5\,\EE-\beta_6\,\VV-\beta_7\,.
\end{equation}
\end{enumerate}
\end{proposition}
\begin{proof}
Following the list of exceptions in \cite[A1]{polyhedralAdjunction}, we see that the exceptions of $K_X+L$ being nef have $k=1$ or they are as in cases $(1)$ and $(2)$. In particular, simple blow-ups at a fixed point and Cayley sums have $k=1$ as they contain linear spaces of degree one. 

In Case (1), the value $\gDD_2(f)$ is computed applying Proposition \ref{prop: gEDD k-regular threefold} with $k=1$, giving $\gDD(f)=\gDD_1(f) = \int_X 15L^3-7c_1L^2+3c_2L-c_3$. The desired formula is obtained considering the polytope interpretation of Chern classes as in \eqref{eq: Chern classes toric}.

In Case (2), the value of $\gDD_2(f)$ is computed using \eqref{eq: EDD k polar classes Veronese}. 

Case $(3)$ deals with threefolds $\P(\EE)=\P(\OO_{\P^1}(a)\oplus \OO_{\P^1}(b)\oplus\OO_{\P^1}(c))$ with $a\ge b\ge c\ge 2$ and the $2$-regular embedding $f\colon\P(\EE)\to\P^{4a+4b+4c+5}$ defined by $\OO_\EE(2)$. Recall that $c_1=\OO_\EE(3)-(a+b+c-2)F$ and $c_2=3c_1(\OO_{\P(\EE)}(1))^2-2(a+b+c-3)c_1(\OO_{\P(\EE)}(1))F,$ where $F\cong\P^2$ is the divisor class of the fiber of the projection map $\pi\colon\P(\EE)\to\P^1.$ It follows that $\int_X L^3 = 8(a+b+c)$, $\int_X c_1L^2 = 8(a+b+c+1)$, $\int_X c_1^2L = 6(a+b+c+4)$, $\int_X c_2L = 2(a+b+c+6)$, $\int_X c_1^3 = 54$, $\int_X c_1c_2 = 24$, and $\int_X c_3 = 6$. Plugging these values in the identity \eqref{eq: gEDD k-regular threefold} for $k=2$, that is
\begin{equation}\label{eq: gEDD 2-regular threefold}
\gDD_2(f) = \int_X 176\,L^3-230\,c_1L^2+81\,c_1^2L-7\,c_1^3+54\,c_2L-20\,c_1c_2-8\,c_3\,,
\end{equation}
we derive the desired identity \eqref{eq: gEDD 2-regular toric threefold case 2}.

Finally, we assume that we are not in cases $(1)$ and $(2)$; therefore, the divisor $K_X+L$ is nef.
Recall that $L=c_1(\OO_X(1))$. Firstly, using the relations in \eqref{eq: Chern classes toric}, one obtains that $\int_X L^3=\PP$, $\int_X c_1L^2=\FF$, $\int_X c_2L=\EE$, and $\int_X c_3=\VV$.
It remains to compute the degrees of $c_1c_2$, $c_1^2L$, and $c_1^3$.
For a nonsingular threefold, the Riemann-Roch theorem gives $\frac{1}{24}\chi(\OO_X)=\int_X c_1c_2.$ Since $\chi(\OO_X)=1$ (see the proof of Corollary \ref{corol: gEDD k-regular toric surface}), we conclude that $\int_X c_1c_2=24.$ Since $c_1=-K_X$, it follows that $\PP_1 = \int_X(L-c_1)^3$ and $\FF_1 = \int_X c_1(L-c_1)^2$, from which we derive the identities $\int_X c_1^3 = 2(\PP_1-\PP)+3(\FF_1+\FF)$ and $\int_X c_1^2L = \PP_1-\PP+\FF_1+2\FF$. Plugging these values in \eqref{eq: gEDD 2-regular threefold}, we obtain \eqref{eq: gEDD 2-regular toric threefold case 3}.
\end{proof}

\begin{example}
We apply \eqref{eq: gEDD 2-regular toric threefold case 3} in a slightly more general version of \cite[Example 3.8]{Higher_duality_and_toric}. Given an integer $a\ge 2$, we consider the $2$-regular toric threefold given by the Segre-Veronese embedding of $(\P^1)^{\times 3}$ with the line bundle $\OO_{(\P^1)^{\times 3}}(a,a,a)$. The lattice polytope associated with the embedding is a cube with edges of length $a$, hence $\PP=3!\cdot a^3=6\,a^3$, $\FF=6\cdot 2!\cdot a^2=12\,a^2$, $\EE=12\,a$, and $\VV=8$. Furthermore $P^\circ$ is also a cube with edges of length $a-2$, hence $\PP_1=6(a-2)^3$ and $\FF_1=12(a-2)^2$. Applying \eqref{eq: gEDD 2-regular toric threefold case 3}, after simplification we obtain that $\gDD_2(\nu_\bm^\bd)=1056\,a^3-2760\,a^2+2592\,a-880$. The result is confirmed by the formula \eqref{eq: gDD 3 copies P1} for $k=2$ and $\bd=(a,a,a)$.\hfill$\diamondsuit$
\end{example}

\section{Tropical geometry of distance optimization}\label{sec: tropical}

In this chapter, we present a new combinatorial framework for studying distance optimization, based on tropical geometry. To enable this analysis, we introduce an additional parameter $t$ to the coefficients of our variety $X$, and replace our base field of complex numbers $\C$ with the valued field of complex Puiseux series $\C\{\{t\}\}$. 
The limiting behavior of $X$ as $t$ goes to $0$ is captured by the tropicalization $\Trop X$, a polyhedral complex that encodes many of the geometric properties of $X$. 
Studying this limiting behavior allows us to characterize both higher distance degrees and the sensitivity of critical points of the distance function with respect to the parameter $t$.

We begin by fixing some notation.
Throughout, $\CC$ denotes an algebraically closed field with a nontrivial, non-Archimedean valuation $\nu \colon \CC^* \to \R$.
We denote by $\CC^{\circ} \coloneqq \big\{ x \in \CC \mid  \nu(x) \ge 0 \big\}$  and by $\CC^{\circ \circ} \coloneqq \big\{ x \in \CC \mid  \nu(x) > 0 \big\}$ the valuation ring of $\CC$ and its unique maximal ideal, respectively. The residue field $\CC^\circ/ \CC^{\circ\circ}$ is denoted by~$\widetilde{\CC}$.

We recall that a nonzero {\em formal Puiseux series} is of the form
\[
x(t)=\sum_{k=k_0}^{\infty} c_k t^{\frac{k}{N}} \text { for some } k_0 \in \Z,\, N \in \N,\, c_k \in \C, \, c_{k_0} \neq 0\,.
\]
The valuation $\nu(x(t))= \frac{k_0}{N}$ is defined as the smallest exponent of $t$ in $x(t)$ with nonzero coefficient, it is a measure of the degree to which $x(t)$ depends on $t$.

From now on, let $f\colon X \to \P^n = \P(\CC)$ denote the closed embedding of an $m$-dimensional variety.
We denote by $\Trop X$ the tropicalization of the cone over $f(X)$.
\begin{align*}
\Trop X \coloneqq \overline{\{ \nu( x) \ \colon \  x \in (\CC^*)^{n+1} \mid \ [x] \in f(X)  \}} \subseteq \R^{n+1},
\end{align*}
This is the closure of the image under the element-wise valuation map, taken in Euclidean topology.

Recall the notation given in the introduction and at the beginning of Section \ref{sec: higher-order normal bundles}. Consider $q\in\mathrm{Sym}^2V^*$, the nonsingular quadric hypersurface $Q = \V(q)\subseteq\P(V)=\P^n$, and the squared distance function $d_u^2(x) = q(u-x)$, seen as a function $d_u^2\colon V\to\C$.
Throughout most of this section, we assume $Q=Q_\mED = \V(x_0^2 + \cdots + x_n^2)$, which is most relevant for Euclidean distance optimization. 
Note that, identifying the projective space $\P^n$ and its dual $(\P^n)^\vee$ through the reciprocity map \eqref{eq: polarity}, induces a closed embedding $\mathrm{id} \times \partial_q^{-1} \colon W_k(f) \subseteq \P^n \times \P^n$ of the conormal variety, and denote the embedded variety by $W_k(f,Q) \coloneqq \mathrm{id} \times \partial_q^{-1} (W_k(f))$.
The varieties $W_k(f,Q)$ and $W_k(f,Q_\mED)$ are related through a linear change of coordinates induced by $M_Q$.
We now show that the tropicalization $\Trop W_k(f,Q)$ contains crucial information on the $ k$th-order polar degrees of $f$ for arbitrary $k$.

\subsection{Proof of Theorem \ref{theorem: tropical description of higher-order multidegrees}}\label{proof5}
\begin{proof}[Proof of Theorem \ref{theorem: tropical description of higher-order multidegrees}]
We can express higher polar degrees of $f$ in terms of the tropicalized conormal variety $\Trop W_k(f,Q)$ purely combinatorially.
To this end, we denote for $0\le h \le n+1$ by $M^h_{n+1}$ the uniform matroid of rank $h$ on $n+1$ elements, and by $\operatorname{Berg}(M^h_{n+1})$ the corresponding Bergman fan (see \cite[\S 4.2]{maclagan2015introduction}, and for several characterizations see \cite[Proposition 3.6]{Feichtner2005}).
We claim that for every $0\le j\le m+n-m_k-1$, the {\em $j$th multidegree $\delta_{k,j}(f)$ of order $k$} is equal to the degree of the stable intersection $\Trop W_k(f,Q) \cdot \operatorname{Berg}(M^{n-j+1}_{n+1}\times M^{j+m_k-m+2}_{n+1})$.

The tropicalization of $L_1 \times L_2$ is equal to the product of the Bergman fans of the uniform matroids $M^{n-j+1}_{n+1}$ and $M^{j+m_k-m+2}_{n+1}$, respectively \cite[Example 5.2.7]{maclagan2015introduction}.
As in equation \eqref{eq: interpretation multidegrees delta_i}, the multidegree
$\delta_{k,j}(f)$ is equal to the cardinality of the finite intersection $W_k(f,Q)\cap(L_1\times L_2)$, where $L_1, L_2 \subseteq \P^n$ are generic linear spaces of dimension
$n-j$ and $j+m_k-m+1$ respectively.
It follows from \cite[Theorem 5.3.3]{osserman2013lifting} that the cardinality of $W_k(f,Q)\cap(L_1\times L_2)$ is equal to the degree of the stable intersection of $\Trop W_k(f,Q) \cdot \Trop(L_1 \times L_2)$.
Alternatively, note that the multidegree $\delta_{k,j}(f)$ is the degree of the cycle $[W_k(f,Q)] \cdot [L_1\times L_2] $ in the Chow ring of $\P^n \times \P^n$. To show equality
with the degree of the tropical cycle $\Trop W_k(f,Q) \cdot \Trop(L_1 \times L_2)$ we employ
\cite[Theorem 3.1]{sturmfels1997intersection}, together with \cite[Theorem 4.4]{tropIntersTheory}. We obtain
\[\delta_{k,j}(f)=\Trop W_k(f,Q) \cdot \operatorname{Berg}(M^{n-j+1}_{n+1}\times M^{j+m_k-m+2}_{n+1}).\]  Finally, the equality 
\[
\gDD_k(f) = \sum_{j = 0}^{ m+n-m_k-1 }
\int \Trop W_k(f,Q) \cdot \operatorname{Berg}(M^{n-j+1}_{n+1}\times M^{j+m_k-m+2}_{n+1})
\]
follows from Proposition \ref{prop: properties kth polar degrees} and Proposition \ref{prop: kth order ED degree sum of kth order polar degrees}.
\end{proof}

Going beyond Theorem \ref{theorem: tropical description of higher-order multidegrees}, for $k=1$ we can not only compute the generic Euclidean distance degree of $f$, but we can also determine the tropical critical points on $X$.

\begin{definition}
We call $w \in \R^n$ a {\em tropical critical point} of $X$ with respect to $u$, if it holds $w = \nu(x(t))$ for some critical point $x(t)$ of $d_u^2$ and $[x(t)] \in X(\CC)$.
\end{definition}

\begin{proposition}\label{prop: critical points w.r.t. u.}
Assume that the pair $(f,Q_\mED)$ is in general $1$-osculating position. Define $Z_u \coloneqq\{(x,y) \in \CC^{n+1} \times \CC^{n+1} \mid x+y = u\}$. The set of critical points of $X$ with respect to $u$ is the image of the set
\[
W_1(f,Q_\mED) \cap Z_u
\]
under the natural projection $\C^{n+1} \times \C^{n+1} \to \C^{n+1}$ to the first component $\C^{n+1}$.
\end{proposition}
\begin{proof}
It follows from \cite[Lemma 2.8]{DHOST} that for every critical point $x$ we can write the data vector $u$ as a sum $u = x + y$ for some pair $(x,y) $ in the cone over the conormal variety $W_1(f,Q_\mED)$.
This immediately implies the desired statement.
\end{proof}

The above statement has the following tropical analogue. However, we need the slightly stronger genericity assumption that $f$ is in general position under the torus action.
\begin{corollary}\label{corollary: tropical critical points from the tropical conormal variety}
Let $t \in \CC^{n+1}$ be a generic element of the algebraic torus with element-wise vanishing valuation: $\nu(t) = 0$. Then the set of tropical critical points of $t \cdot X$ with respect to $u$ is the image of the stable intersection
\[
\Trop W_1(t\cdot f,Q_\mED) \cdot \Trop Z_u
\]
under the natural projection $\R^{n+1}\times\R^{n+1}\to\R^{n+1}$ to the first component $\R^{n+1}$.
\end{corollary}
\begin{proof}
By Proposition \ref{prop: critical points w.r.t. u.}, it suffices to show the equality
\[
\Trop W_1(t \cdot f,Q_\mED) \cdot \Trop Z_u = \Trop(W_1( t \cdot f,Q_\mED) \cap Z_u)\,.
\]
The desired equality follows from \cite[Theorem 5.3.3]{osserman2013lifting}.
\end{proof}

\begin{example}\label{example: rational projective curve tropically}
As a simple example, we consider the rational projective curve
\[
f \colon\P^1 \longrightarrow \P^3\,,\ t \mapsto [1 : t^2 : t^3 : t^4]\,
\]
Here we are working in the affine patch $\{t_0\neq 0\}$ and are using the local coordinate $t=\frac{t_1}{t_0}$.
The projective dual variety is the hypersurface
\[
X^\vee = \V(4\,x_1^3x_2^2 + 27\,x_0x_2^4 - 16\,x_1^4x_3  - 144\,x_0x_1x_2^2x_3 + 128\,x_0x_1^2x_3^2 - 256\,x_0^2x_3^3)\subseteq(\P^3)^\vee\,.
\]
In Figure \ref{fig:test1} we show a green curve $u(t) = (1, t^{2}, t^{3}, t^{4})$ of data points, where $(t_0,t_1)=(1,t)$, and a red curve $x(t) \in X$ of points with a minimal Euclidean distance to $u(t)$.
We display the image of the cone over $X$ under the projection $(x_0, x_1, x_2, x_3) \mapsto(x_1, x_2, x_3)$ in $3$-dimensional space.
Drawn on logarithmic paper, Figure \ref{fig:test1} becomes approximately linear. We show the tropicalization
\[
\Trop X =  \rowspan \begin{pmatrix}
1 &1 &1 &1 \\
0 &2 &3 &4 \\
\end{pmatrix}
\]
in Figure \ref{fig:test2}.
In this example, the tropical conormal variety $\Trop W_1(f,Q_\mED)$ is of dimension $4$, and the union of the four maximal cones
\[
\sigma_1 = \R_+\cdot e_5 + L, \
\sigma_2 = \R_+\cdot e_6 + L, \
\sigma_3 = \R_+\cdot e_7 + L, \
\sigma_4 = \R_+\cdot e_8 + L,
\]
where $L$ denotes the three-dimensional lineality space
\[
L = \rowspan
\begin{pmatrix}
1 & 1& 1& 1& 0& 0 &0 & 0\\
0 & 0& 0& 0& 1& 1 &1 & 1\\
0 & 2& 3& 4& 0& -2 &-3 & -4\\
\end{pmatrix}.
\]
The tropicalization $Z_u$ is a four-dimensional polyhedral fan comprising $80$ maximal cones.
We check that the stable intersection of $\Trop Z_u$ and $\Trop W_k(f,Q_\mED)$ contains the point $(0, 2, 3, 4, 0, 4, 3, 2)$, indicating that there exists a critical point $x(t)$ with valuation $\nu(x(t)) = (0, 2, 3, 4)$. 
\begin{figure}
\centering
\begin{minipage}[t]{.5\textwidth}
\centering
\begin{overpic}[width=0.8\textwidth]{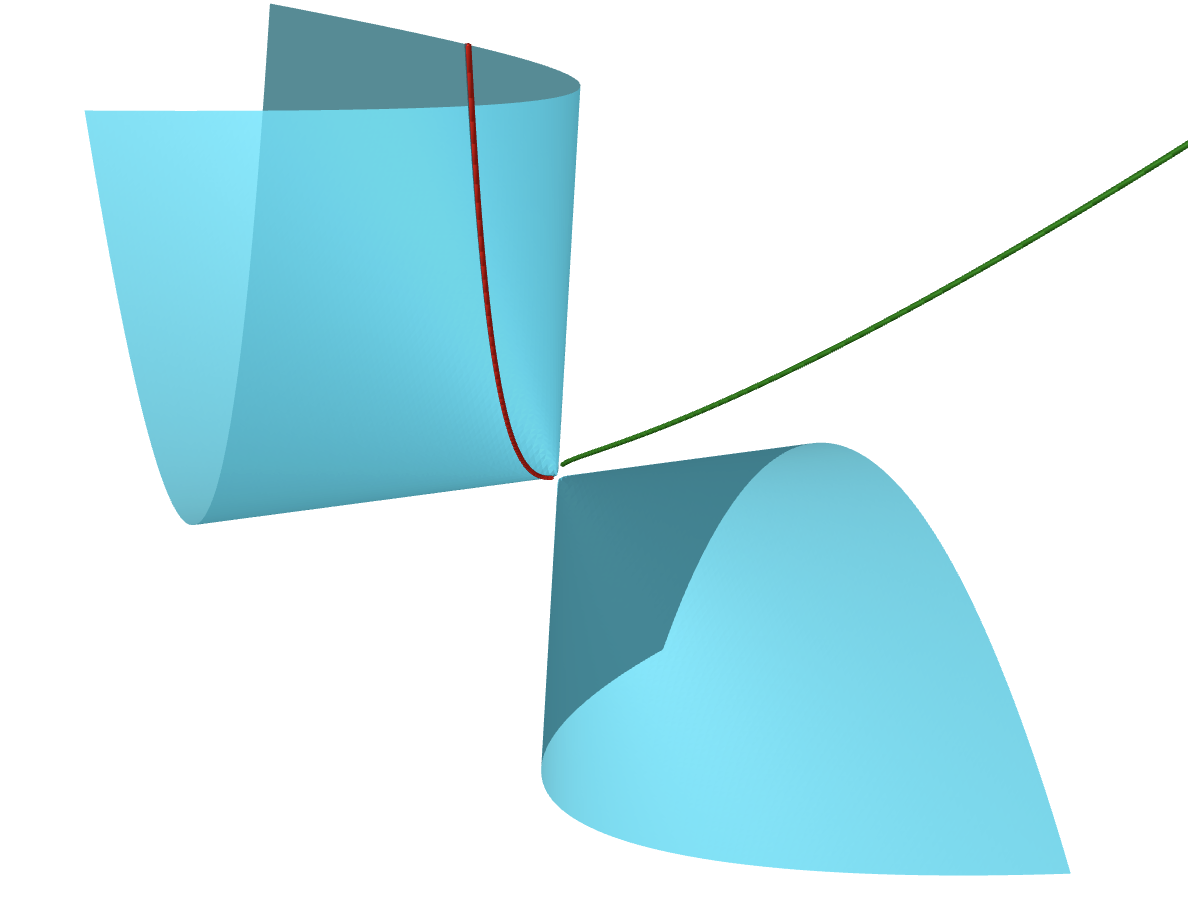}
\put (70,55) {$u(t)$}
\put (32,78) {$x(t)$}
\end{overpic}
\captionof{figure}{A one parameter family of data vectors $u(t)$ and critical points $x(t)$.}
\label{fig:test1}
\end{minipage}%
\begin{minipage}[t]{.5\textwidth}
\centering
\begin{overpic}[width=0.45\textwidth]{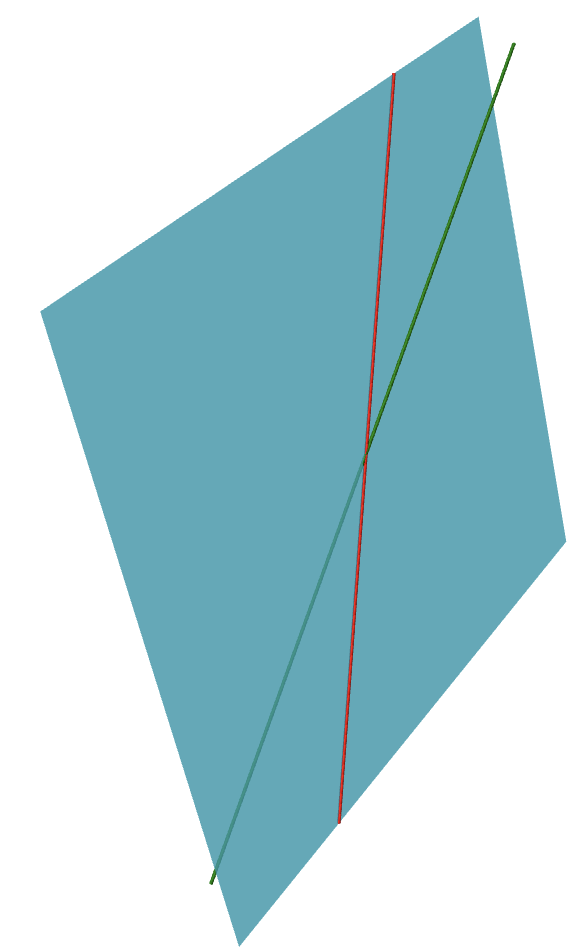}
\put (25,95) {$\nu(x(t))$}
\put (45,100) {$\nu(u(t))$}
\put (5,41) {$\Trop X$}
\end{overpic}
\captionof{figure}{The tropical picture.}
\label{fig:test2}
\end{minipage}
\end{figure}
\end{example}

\subsection{Toric varieties}

In this chapter, we characterize the generic higher-order distance degrees and polar degrees of toric varieties.
Similar to \cite{Tropical_Discriminants}, we use implicitization techniques to
tropicalize higher-order conormal varieties, based on an analogue of the Horn uniformization map.

Recall the notation given in Section \ref{subsec: gallery of toric examples}. In particular $f\colon X\to\P^n$ is a nonsingular, globally $k$-osculating, torus-equivariant embedding of a toric variety $X$ of dimension $m$. 
Again, $f$ is defined on the algebraic torus by an integer matrix $A$ of full rank.
In this section, by abuse of notation we now denote by the same letter $A$, the matrix
$\begin{pmatrix}
    \bone \\
    A
\end{pmatrix}$,
which contains the additional row $\bone=(1,\dots,1)$.
In particular, $A \in \Z^{(m+1)\times(n+1)}$.
We denote by $A^{(k)}$ the matrix $A_{\bone}^{(k)}(f)$ defined in \eqref{eq: matrix A}.
As before, the $k$th osculating space of $X$ at $\bone$, $\T_{\bone}^k(f)$, is equal to the row span of $A^{(k)}$.
The following is analogous to the Horn parametrization of the dual variety $X^\vee$, see, for example, \cite[Proposition 4.1]{Tropical_Discriminants}.
In fact, the $k$th-order dual variety is the projection of $W_k(f)$ onto the second factor. Our exposition differs in that we allow $k$ to be larger than one.

\begin{proposition}
\label{prop: parametrization for conormal variety}
The $k$th-order conormal variety $W_k(f,Q_\mED)$ is the closure of the image of the map
\begin{align*}
    \gamma\colon (\CC^*)^m \times \P(\ker(A^{(k)} )) &\longrightarrow \P^{n} \times \P^{n}\\
    (t, u) & \longmapsto (t^A \cdot \bone, t^{-A} \cdot u).
\end{align*}
\end{proposition}
\begin{proof}
Both $\overline{\image\gamma}$ and $W_k(f,Q_\mED)$ are irreducible varieties of dimension $n-m_k+m-1$, and we are left with showing the inclusion $\overline{\image\gamma} \subseteq W_k(f,Q_\mED)$.
We first note that, by the construction of $W_k(f)$ we have the inclusion $\{\bone\}\times \T_{\bone}^k(f)^\perp \subseteq W_k(f,Q_\mED)$.
By \cite[Lemma 5.2]{Higher_duality_and_toric} it holds $\T_{\bone}^k(f) = \P(\rowspan A^{(k)}),$ and one checks directly that, following \eqref{eq: perp}, the orthogonal complement $\P(\rowspan A^{(k)})^\perp$ is equal to the kernel $\P(\ker(A^{(k)} ))$.
Together, we obtain the inclusion
\[
\{\bone\} \times \P(\ker(A^{(k)} )) \subseteq W_k(f,{Q_\mED})\,
\]
Now $W_k(f,Q_\mED)$ is stable under the $(\CC^*)^m$-action $t \cdot (x, T) =  (t^A \cdot x , t^{-A} \cdot T )$ for all $(x, T) \in W_k(f,Q_\mED)$. This finishes the proof.
\end{proof}

Before providing an algorithm for computing higher-order polar degrees, we recall basic results on the tropicalization of certain unirational varieties.
Let $C$ and $D$ be integer matrices of size $r \times d$ and $s \times r$ respectively. The rows of $D$ are $w_1,\dots,w_s \in \Z^r$.
We denote by $\lambda_C$ the map defined by $C$, and by $\mu_D$ the monomial map specified by $D$:
\begin{multicols}{2}
\noindent
\begin{align*}
\lambda_C \colon (\CC^*)^d &\dashrightarrow (\CC^*)^r \\
v &\longmapsto C\,v
\end{align*}
\begin{align}\label{eq: parametrization of YUV}
\begin{split}
\mu_D \colon (\CC^*)^r &\longrightarrow (\CC^*)^s \\
x &\longmapsto (x^{w_1},\dots,x^{w_s}).
\end{split}
\end{align}
\end{multicols}

The composition of these maps gives the unirational variety
$Y_{C,D} = \overline{\image( \mu_D \circ \lambda_C )}$ in $(\CC^*)^s$.
Its tropicalization $\Trop Y_{C,D}$ is obtained by tropicalizing the map $\mu_D \circ \lambda_C$.
The tropical linear space $\Trop(\image\lambda_C)$ is computed
purely combinatorially, as the {\em Bergman fan} of the matroid of $U$. An effective algorithm for computing this object was introduced in \cite{RINCON201386}. An implementation can be found in \texttt{polymake} \cite{Gawrilow2000}.
The monomial map $\mu_D$ tropicalizes to the linear map $V\colon\R^r \to \R^s$.
The following result is from \cite[Theorem 3.1]{Tropical_Discriminants} and
\cite[Theorem 5.5.1]{maclagan2015introduction}.
Linear projections of balanced polyhedral complexes, and the underlying weights, are described in \cite[Lemma 3.6.3]{maclagan2015introduction}.

\begin{lemma}
\label{lemma: tropicalizing monomial maps}
The tropical variety $\Trop Y_{C,D}$ is the image, as a balanced fan, of the Bergman fan $\Trop(\image\lambda_C)$ under the linear map $\R^r \to \R^s$ given by $V$.
\end{lemma}
\noindent
In particular, Lemma \ref{lemma: tropicalizing monomial maps} allows us to compute not only the support of $\Trop Y_{C,D}$, but also the multiplicities of its maximal cones \cite[Definition 3.4.3]{algorithmsInInvariantTheory}.

The following result is a direct consequence of Proposition \ref{prop: parametrization for conormal variety} together with Lemma \ref{lemma: tropicalizing monomial maps}.
\begin{proposition}   
\label{prop: description of conormal variety as Y_{C,D}}
Fix an $(n+1)\times(n+1-m_k)$ matrix $B$ whose columns span the kernel of $A^{(k)} $, called the {\em Gale dual} of $A^{(k)}$. 
Then the affine cone over $W_k(f,Q_\mED)$ is
$Y_{C,D}$, where
\begin{equation}
\label{eq: Matrices U, V}
U = \begin{pmatrix}
0 & B \\
I_{m+1} & 0  \\
\end{pmatrix}, \quad
V = \begin{pmatrix}
-A^t & 0  \\
A^t & I_{n+1} \\
\end{pmatrix}.
\end{equation}
Here $Y_{C,D} = \overline{\image( \mu_D \circ \lambda_C )}$ and $ \lambda_C$ and $\mu_D$ are defined as in equation \eqref{eq: parametrization of YUV} with $d = n+m+2-m_k$, $r = n+m+2$, and $s = 2n+2$.

Furthermore, the tropical variety $\Trop W_k(f,Q_\mED) $ is supported
on the image of $\Trop Y_{C,D}$ in $\R^{n+1}\times \R^{n+1}$. This is the Minkowski sum
\[
\Trop W_k(f,Q_\mED) = \{0\} \times \Trop(\ker A^{(k)}) + \rowspan
\begin{pmatrix}
-A & A
\end{pmatrix}\,.
\]
\end{proposition}

\begin{remark}
Compare the description of $\Trop W_k(f,Q_\mED)$ in Proposition \ref{prop: description of conormal variety as Y_{C,D}} to the description of the tropicalized dual variety $\Trop X^{(k)}$, given in \cite[Corollary 4]{Tropical_Discriminants} 
for $k=1$, and for general $k$ in \cite[Theorem 5.3]{Higher_duality_and_toric}.
We recover both results by projecting $\Trop W_k(f,Q_\mED)\subseteq \R^{n+1}\times \R^{n+1}$ onto the second factor.
\end{remark}

Proposition \ref{prop: description of conormal variety as Y_{C,D}} leads to Algorithm \ref{alg:tropConormal} below for computing $\Trop W_k(f,Q_\mED)$. 
Algorithm \ref{alg:tropConormal} returns a list of pairs $(m_\tau,\tau)$, where $\tau \subseteq \R^{2n}$ is a polyhedral cone, and $m_\tau$ is a positive integer. The tropical variety $\Trop W_k(f,Q_\mED)$ is the union of all cones $\tau$, and the multiplicity of $\Trop W_k(f,Q_\mED)$ at a generic point $x$ is the sum $\sum_{x\in\tau}m_\tau$.
We note that, although the union of all cones $\tau$ forms the support of a fan, the collection of cones itself is generally not a fan. 
Together with Theorem \ref{theorem: tropical description of higher-order multidegrees}, Algorithm \ref{alg:tropConormal} allows us to compute higher-order polar degrees of toric varieties.
We provide an experimental implementation in the \verb|Julia| package \texttt{TropicalImplicitization} \cite{Rose2025}. It can be found at the supplementary website
\url{https://github.com/kemalrose/TropicalImplicitization.jl}.

\begin{example}\label{example: code demonstration}
To recover the second-order polar degrees $(\mu_{2,0}(\nu_1^3),\mu_{2,1}(\nu_1^3))=(3,3)$ of the cubic Veronese embedding $\nu_1^3 \colon \P^1 \hookrightarrow \P^3$, defined in coordinates as $[t_1 : t_2] \mapsto [t_1^3 : t_1^2 t_2 : t_1 t_2^2 : t_2^3]$, we download the \verb|Julia| software package
\texttt{TropicalImplicitization} from the source and we run the following commands:
\begin{verbatim}
A = [1 1 1 1; 0 1 2 3]
cone_list, weight_list = get_tropical_conormal_variety(A, 2)
extract_polar_degrees(cone_list, weight_list)
\end{verbatim}
This allows us to recover the 4-dimensional tropical conormal variety $\Trop W_2(f,Q_\mED)$ as the union of four maximal cones
\[
\sigma_1 = \R_+\cdot e_5 + L\,,\quad
\sigma_2 = \R_+\cdot e_6 + L\,,\quad
\sigma_3 = \R_+\cdot e_7 + L\,,\quad
\sigma_4 = \R_+\cdot e_8 + L\,.
\]
Here $L$ denotes the lineality space.
We can also just run the command 
\begin{verbatim}
compute_polar_degrees(A,2)
\end{verbatim}
The output then reads:
\begin{verbatim}
The toric variety is of degree 3.
The generic distance degree of order 2 is 6.
The dual variety of order 2 is of degree 3 and of codimension 2.
The polar degrees of order 2 are [3, 3].
\end{verbatim}
\end{example}

\begin{example}\label{ex: polar degrees O(1,1,2) with k=2}
Consider the triples $\bm=(1,1,1)$ and $\bd=(1,1,2)$. As a more challenging example, we compute the polar degrees of order $2$ of the Segre-Veronese embedding $\nu_\bm^\bd\colon\P^\bm\hookrightarrow\P^{11}$, see \eqref{eq: def Segre-Veronese embedding}. This toric embedding is defined by the $4\times 12$ matrix
\[
A = 
\begin{pmatrix}
1& 1& 1& 1& 1& 1& 1& 1& 1& 1& 1& 1\\
0& 1& 0& 1& 0& 1& 0& 1& 0& 1& 0& 1\\
0& 0& 1& 1& 0& 0& 1& 1& 0& 0& 1& 1\\
0& 0& 0& 0& 1& 1& 1& 1& 2& 2& 2& 2
\end{pmatrix}\,.
\]
We note that $\nu_\bm^\bd$ is $2$-osculating, but not $2$-regular, thanks to Corollary \ref{corol: when Segre-Veronese k-regular}.
Our software reveals that $(\mu_{2,0}(\nu_\bm^\bd),\dots,\mu_{2,3}(\nu_\bm^\bd))=(12,28,36,28)$, hence $\gDD_2(f)=\sum_{i\ge 0}\mu_{2,i}(\nu_\bm^\bd)=88$. 
Note that this number cannot be computed using \eqref{eq: gDD 3 copies P1}.
Based on Proposition \ref{prop: compare k-order defect and number of vanishing higher polar degrees}, we see that 
the second-order dual variety is of dimension $6$ and of degree $28$.
$$\codim(X_2^\vee) = m_2-m+\defect_2(X)+1 = 7-3+0+1 = 5.$$
\end{example}

Algorithm \ref{alg:tropConormal} is based on Proposition \ref{prop: description of conormal variety as Y_{C,D}} together with Lemma \ref{lemma: tropicalizing monomial maps}.
In particular, $\Trop W_k(f,Q_\mED)$ is a linear projection of the
tropicalized column span of $U = \begin{pmatrix}
B & 0 \\
0 & I_{m+1}
\end{pmatrix}$ under
$V = \begin{pmatrix}
0 & -A^t \\
I_{n+1} & A^t \\
\end{pmatrix}.$

To describe the underlying weights of $\Trop W_k(f,Q_\mED)$, let $y \in \Trop W_k(f,Q_\mED)$ be a generic point inside a top-dimensional cone $\tau \subseteq \Trop W_k(f,Q_\mED)$. Following \cite[Theorem 3.12]{TevelevSturmfels}, we can express the multiplicity $m_{\tau}$ of $\tau$ as a sum of lattice indices.
Here the sum runs over all (finitely many) points $x$ in the tropicalized column span of $U$, which $V$ maps to $y$. We denote by $\sigma_x$ a top-dimensional cone in $\Trop(\image\lambda_C)$ containing $x$, and by $\mathbb{L}_\tau$ and $\mathbb{L}_\sigma$ the linear span of
$\tau-y$ and $\sigma-x$ respectively.
The multiplicity $ m_{\tau} $ can be expressed as:
\begin{equation*}
 m_{\tau} \, = \,  \sum_{Vx = y} \text{index}
 \bigl(\,\mathbb{L}_\tau \cap \Z^{2n+2} : V(\mathbb{L}_{\sigma_x} \cap \Z^{ n+m+2 })\, \bigr). 
\end{equation*}
See also the chapter ``Tropical Implicitization Revisited'' in \cite{OSCAR-book}.

\IncMargin{1em}
\begin{algorithm}
\SetKwData{Left}{left}\SetKwData{This}{this}\SetKwData{Up}{up}
\SetKwFunction{Union}{Union}\SetKwFunction{FindCompress}{FindCompress}
\SetKwInOut{Input}{Input}\SetKwInOut{Output}{Output}
\Input{An integer matrix $A \in \Z^{(m+1)\times(n+1)}$, of full rank containing the all ones vector $(1,1,\dots,1)$ in its row span.}
\Output{The tropical variety $\Trop W_k(f,Q_\mED) \subseteq \R^{n+1} \times \R^{n+1}$}
\BlankLine
$B \to \text{Gale dual of }A^{(k)}$ \label{line: GaleDual}\\
$U \to \begin{pmatrix}
0 & B \\
I_{m+1} & 0  \\
\end{pmatrix}$\\
$V \to \begin{pmatrix}
-A^t & 0 \\
A^t & I_{n+1} \\
\end{pmatrix}$\\
$M \to \text {matroid of } U$ \label{line: 1 tropical linear space}\\
$\Trop(\image\lambda_C) \to  \text{Bergman fan of } M $ \label{line: 2 tropical linear space}\\
$\Trop W_k(f,Q_\mED) \to  \emptyset $\\
\For{$(m_\sigma, \sigma) \in \Trop(\image\lambda_C)$} {
\label{line: project trop U}
$\tau \to V   \sigma$\\
$m_{\rm lattice}  \to {\rm index}(\mathbb{L}_\tau \cap \Z^{2n+2}: V(\mathbb{L}_\sigma \cap \Z^{n+m+2}))$\\
$\Trop W_k(f,Q_\mED) \to \Trop W_k(f,Q_\mED) \cup \{(m_\sigma \cdot m_{\rm lattice}, \tau)\}$\\
}
\caption{Tropicalizing higher conormal varieties.}
\label{alg:tropConormal}
\end{algorithm}\DecMargin{1em}

\section{Higher-order distance degrees of affine morphisms}\label{sec: affine}

Keeping in mind the notations used in Section~\ref{sec: osculating}, we now investigate higher-order distance degrees and distance loci of $m$-dimensional nonsingular irreducible varieties in the affine space $\A^n$ over $\C$.
Let $\A^m$ and $\A^n$ be two affine spaces over $\C$, and let $t=(t_1,\dots,t_m)$ be a coordinate system in $\A^m$. Consider an algebraic morphism $f\colon\A^m\to\A^n$. In particular, there exist $n$ polynomials $f_1,\dots,f_n$ in $\C[t_1,\dots,t_n]$ such that $f=(f_1,\dots,f_n)$. The image $f(\A^m)$ is an irreducible variety of $\A^n$.
In Section~\ref{sec: osculating}, we defined the $k$th-order osculating spaces of projective morphisms using the jet bundles $\PP^k(\OO_X(1))$ and the morphisms $j_k$ in \eqref{eq: morphism jk}. One may repeat the same construction for affine morphisms $f\colon\A^m\to\A^n$. This is equivalent to considering the $\binom{m-1+k}{k}\times n$ matrices
\begin{equation}\label{eq: affine matrices Apk}
A_p^{(k)}(f) \coloneqq \left(\frac{1}{|\alpha|}\frac{\partial^{|\alpha|}f_i}{\partial t^\alpha}(p)\right)_{\substack{\alpha\in\N^m\\\ 1\le|\alpha|\le k\\1\le i\le n}}\,,\quad p\in\A^m\,,
\end{equation}
where, for the sake of brevity, we adopted the same notation as in \eqref{eq: matrix A}.
Notice that, differently from the projective case, the first row of $A_p^{(k)}(f)$ is not the vector of components of $f$ evaluated at $p$. This is because $f(p)$ belongs to the affine tangent space of $f$ at $f(p)$ only when $f(\A^m)$ is an affine cone in $\A^n$.
\begin{definition}\label{def: affine osculating space}
For every point $p\in\A^m$, the {\em $k$th osculating space of $f$ at $f(p)\in f(\A^m)$} is
\[
T_p^k(f) \coloneqq \operatorname{rowspan}A_p^{(k)}(f)\,.
\]
Furthermore, we denote by $U_k$ the dense open subset of $\A^m$ of points $p$ such that $\rank A_p^{(k)}(f)$ is constant. We define the {\em (affine) generic $k$-osculating dimension} of $f$ as $m_k\coloneqq\rank A_p^{(k)}(f)$, and we say that the morphism $f$ is {\em globally $k$-osculating} if $U_k=\A^m$. In this case, $m_k$ is referred to as the {\em (affine) $k$-osculating dimension} of $f$.
\end{definition}
Notice that, by construction, the $k$th osculating space $T_p^k(f)$ always passes through the origin, not necessarily through $f(p)$. Indeed, the affine $k$th osculating space of $f$ at $f(p)\in f(\A^m)$ is $f(p)+T_p^k(f)$. Furthermore, to avoid confusion, we stress that the invariant $m_k$ given in Definition \ref{def: affine osculating space} differs by one from the generic $k$-osculating dimension of Definition \ref{def: globally k-osculating}, also denoted by $m_k$.

In the following, we consider the compactification $\P^n=\A^n\cup H_\infty$, in particular, $H_\infty$ is the hyperplane at infinity of $\A^n$. For any affine subspace $L\subseteq\A^n$, we define $L_\infty\coloneqq\overline{L}\cap H_\infty$, where $\overline{L}$ is the Zariski closure of $L$ in $\P^n$.
Suppose that $\A^n$ is the complexification of a real $n$-dimensional affine space equipped with a positive-definite quadratic form $q$. We can associate with $q$ a unique nonsingular quadric hypersurface $Q\subseteq H_\infty$. The notion of affine orthogonality in $\A^n$ is then given by the notion of polarity in $H_\infty$, using the construction given at the beginning of Section~\ref{sec: higher-order normal bundles}. More precisely, given an affine space $L\subseteq\A^n$ and a point $z\in\A^n$, the {\em orthogonal space to $L$ passing through $z$} is
\[
z+L^\perp \coloneqq \langle z, (L_\infty)^\perp\rangle\cap\A^n\,,
\]
where $(L_\infty)^\perp$ is defined as in \eqref{eq: perp}, after replacing $\P^n$ with $H_\infty$, and the span is taken in $\P^n$.

\begin{definition}\label{def: affine kth order normal space}
Consider a morphism $f\colon\A^m\to\A^n$, a nonsingular quadric hypersurface $Q\subseteq H_\infty$, and let 
$p\in U_k.$
The {\em $k$th-order normal space of $(f,Q)$ at $f(p)$} is
\[
    N_p^k(f,Q) \coloneqq (T_p^k(f))^\perp \subseteq \A^n\,.
\]
When $k=1$, we call it the {\em normal space of $(f,Q)$ at $p$} and we denote it by $N_p(f,Q)$.
\end{definition}

Similarly as in the introduction, given a point $u\in\A_{\mR}^n$, we denote by $d_u$ the distance function from $u$, defined by $d_u(z)\coloneqq\sqrt{q(u-z)}$ for all $z\in\A^n$. We say that $z=f(p)\in f(\A^m)$ with $p\in U_k$ is {\em critical of order $k$} for the squared distance function $d_u^2$ if
\begin{equation}\label{eq: kth order critical affine}
    \nabla d_u^2(z)\in N_p^k(f,Q)\,.
\end{equation}

\begin{definition}\label{def: affine kth distance correspondence}
Consider a morphism $f\colon\A^m\to\A^n$ and a nonsingular quadric hypersurface $Q\subseteq H_\infty$.
The {\em (affine) $k$th-order distance correspondence of $(f,Q)$} is
\[
\DC_k(f,Q) \coloneqq \overline{\left\{(z, u) \in \A^n \times \A^n \,\middle|\, z \in f(\A^m) \text{ is critical of order } k \text{ for } d_u^2\right\}}\subseteq\A^n \times \A^n\,.
\]
\end{definition}

Denote by $\pr_1$ and $\pr_2$ the projections of $\A^n\times\A^n$ onto the first and second factor, respectively. Similarly as in \cite[Theorem 4.1]{DHOST} and in Lemma \ref{lem: higher-order proj ED correspondence is irreducible}, one shows that $\pr_1$ is locally trivial over $f(U_k)$ with fibers of rank $n-m_k$, hence $\DC_k(f,Q)$ is irreducible of dimension $m+n-m_k$ in $\A^n\times\A^n$.

\begin{definition}\label{def: affine kth oder distance locus}
Consider a morphism $f\colon\A^m\to\A^n$ and a nonsingular quadric hypersurface $Q\subseteq H_\infty$.
The {\em (affine) $k$th-order distance locus of $(f,Q)$} is
\[
\DL_k(f,Q) \coloneqq \pr_2(\DC_k(f,Q)) = \overline{\bigcup_{p\in U_k} (f(p) + N_p^k(f,Q))}\,.
\]
\end{definition}

Finally, we denote by $\varphi_{1,k}\colon\DC_k(f,Q)\to f(\A^n)$ and $\varphi_{2,k}\colon\DC_k(f,Q)\to\DL_k(f,Q)$ the surjective morphisms induced by the projections $\pr_1$ and $\pr_2$. The next definition is a higher-order version of the EDD of an affine variety.

\begin{definition}
Consider an affine morphism $f\colon\A^m\to\A^n$ and a nonsingular quadric hypersurface $Q\subseteq H_\infty$. Assume that the morphism $\varphi_{2,k}$ is generically finite. The {\em (affine) $k$th-order distance degree of $(f,Q)$} is
\[
\DD_k(f,Q)\coloneqq\deg\DL_k(f,Q)\cdot\deg\varphi_{2,k}\,,
\]
where $\deg\varphi_{2,k}=\deg\varphi_{2,k}^{-1}(u)$ for a generic $u\in\DL_k(f,Q)$.    
\end{definition}

The following fact is an immediate consequence of the previous constructions.

\begin{corollary}\label{corollary: generically finite}
If $\varphi_{2,k}$ is generically finite, then $\DL_k(f,Q)$ is an irreducible variety of dimension $m+n-m_k$ in $\A^n$.
\end{corollary}

In the following, we consider an affine version of the Illustrative Example given in the introduction.

\begin{example}
Consider the restriction of the morphism $f\colon\P^1\to\P^3$ in the Illustrative Example to the affine patch $\A^1=\{t_0\neq 0\}$. We use the coordinate $t=\frac{t_1}{t_0}$, hence $H_\infty=\{u_0=0\}\cong\P^2$. The corresponding affine morphism is then $f(t)\coloneqq(t,t^2,t^3)$. In the affine setting, the matrix
\[
    A_p^{(2)}(f) = 
    \begin{pNiceMatrix}[last-col=4]
    1&2\,t&3\,t^2&\frac{\partial f}{\partial t}\\[4pt]
    0&1&3\,t&\frac{1}{2}\frac{\partial^2 f}{\partial t^2}
    \end{pNiceMatrix}
\]
equals the bottom $2\times 3$ block of the corresponding matrix $A_p^{(2)}(f)$ in the Illustrative Example.
On the one hand $T_p^2(f)=\rowspan A_p^{(2)}(f))$ is $2$-dimensional.
On the other hand, the right kernel of $A_p^{(2)}(f)$ is one-dimensional and is generated by the vector $\eta=(3\,t^2,-3\,t,1)^\mT$.
If we consider the standard Euclidean quadric $Q_\mED\subseteq H_\infty$, then the ideal of $\DC_2(f,Q_\mED)$ is simply
\[
I(\DC_2(f,Q_\mED)) =
\left\langle
\text{$2\times 2$ minors of }
\begin{pmatrix}
3\,t^{2} & -3\,t & 1\\
t-u_1&t^{2}-u_2&t^{3}-u_3
\end{pmatrix}\right\rangle\,.
\]
Eliminating the variable $t$ from the previous ideal, one verifies that the affine second-order distance locus $\DL_2(f,Q_\mED)$ is the surface of degree $5$ in $\A^3$ cut out by the polynomial
\begin{align*}
&9\,u_2^5-27\,u_1u_2^3u_3-27\,u_1^4-45\,u_1^2u_2^2-36\,u_2^4+81\,u_1^3u_3+108\,u_1u_2^2u_3-81\,u_1^2u_3^2-9\,u_2^2u_3^2\\
&+27\,u_1u_3^3+39\,u_1^2u_2+48\,u_2^3-84\,u_1u_2u_3+9\,u_2u_3^2-4\,u_1^2-24\,u_2^2+12\,u_1u_3+4\,u_2\,.
\end{align*}
In particular, using the description of $\DL_2(f,Q_\mED)$ given in Definition \ref{def: affine kth oder distance locus}, a generic point $u\in\DL_2(f,Q_\mED)$ can be written as $u = (t_0,t_0^2,t_0^3)+\lambda(3\,t_0^2,-3\,t_0,1)$ for some $(t_0,\lambda)\in\C^2$. Plugging in this relation in $I(\DC_2(f,Q_\mED))$ and computing the $2\times 2$ minors of the above matrix, one gets the primary decomposition
\[
\langle t-t_0\rangle\cap\langle t_0\,t-2\,t_0^{2}+2,\,t^{2}+2\,t_0^{2}-1\rangle\,,
\]
where the first component gives the expected solution $t=t_0$, while eliminating the variable $t$ from the second component gives the relation $6\,t_0^4-9\,t_0^2+4=0$, which is not satisfied for a generic $t_0$. This proves that a generic $u\in\DL_2(f,Q)$ has only one second-order critical point, hence
\[
\DD_2(f,Q_\mED) = \deg\DL_2(f,Q_\mED)\cdot\deg\varphi_{2,2} = 5\cdot 1=5\,.
\]
We conclude this example comparing the surface $\DL_2(f,Q_\mED)$ with the branch locus of the second projection of $\DC_1(f,Q_\mED)\to\A^3$, namely the {\em ED discriminant} of $(f,Q)$ (see \cite[Section 7]{DHOST}), denoted by $\Sigma(f,Q)$. It corresponds to the locus of data points $u$ whose locus of first-order critical points of $d_u^2$ on $f(\A^1)$ is not a finite set of EDD pairwise distinct points. The complement of the real zero locus of $\Sigma(f,Q)$ divides the affine real space into chambers where the number of real distinct critical points of $d_u^2$ is constant. In this case, the ED discriminant of $(f,Q)$ is the sextic surface in $\A^3$ cut out by the polynomial
\begin{align*}
&26244\,u_2^2u_3^4-78732\,u_1u_3^5+73728\,u_2^5-345600\,u_1u_2^3u_3+364500\,u_1^2u_2u_3^2\\
&+62208\,u_2^3u_3^2-204120\,u_1u_2u_3^3-26244\,u_2u_3^4-84375\,u_1^4-144000\,u_1^2u_2^2\\
&-159744\,u_2^4+202500\,u_1^3u_3+437760\,u_1u_2^2u_3-271350\,u_1^2u_3^2-92160\,u_2^2u_3^2\\
&+98604\,u_1u_3^3+6561u_3^4+100800\,u_1^2u_2+137216\,u_2^3-185472\,u_1u_2u_3\\
&+45504\,u_2u_3^2-17856\,u_1^2-58368\,u_2^2+26496\,u_1u_3-7488\,u_3^2+12288\,u_2-1024\,.
\end{align*}
The two surfaces $\DL_2(f,Q_\mED)$ and $\Sigma(f,Q)$ are displayed in Figure~\ref{fig: affine 2nd order data locus and ED discriminant}.\hfill$\diamondsuit$
\begin{figure}[ht]
\centering
\begin{overpic}[width=0.45\textwidth]{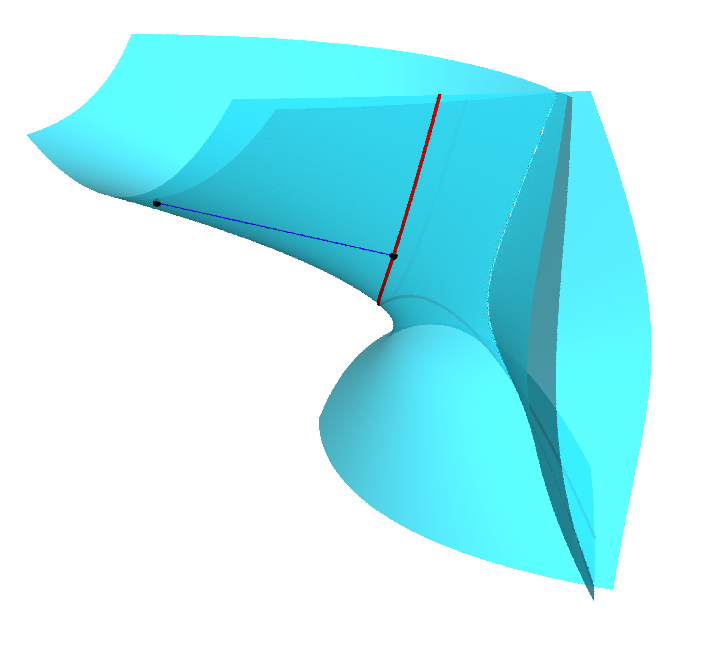}
\put (47,20) {$\DL_2(f,Q_\mED)$}
\put (44,70) {$f(\A^1)$}
\put (57,54) {\small{$f(p)$}}
\put (18,60) {\small{$u$}}
\put (19,54) {\scriptsize{$f(p)+N_{p}^2(f,Q_\mED)$}}
\end{overpic}
\begin{overpic}[width=0.45\textwidth]{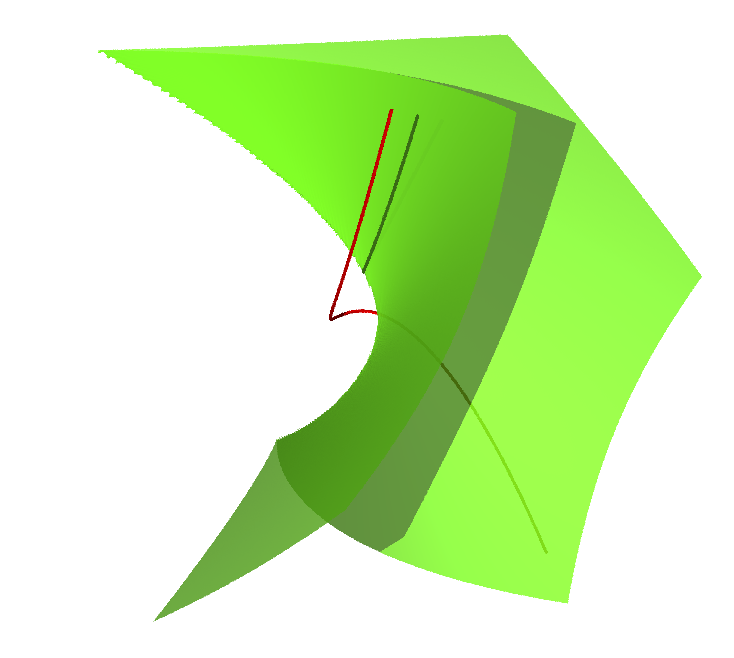}
\put (38,70) {$f(\A^1)$}
\put (47,20) {$\Sigma(f,Q_\mED)$}
\end{overpic}
\caption{The affine second-order distance locus $\DL_2(f,Q_\mED)$ (left) and the ED discriminant $\Sigma(f,Q_\mED)$ (right) of the affine twisted cubic in $\A^3$, with respect to the standard Euclidean quadric.}\label{fig: affine 2nd order data locus and ED discriminant}
\end{figure}
\end{example}

More in general, we computed the $k$th-order distance degree of the map $f(t)=(t,t^2,\ldots,t^d)\in\A^d$ for several values of $k$ and $d$, with $k\le d$, with respect to the standard Euclidean quadric $Q_\mED\subseteq H_\infty$ and for other sufficiently general quadrics $Q$. We conjecture that $\deg\varphi_{2,k}=1$ for all $k\le d$ and that
\[
\DD_k(f,Q) = \deg\DL_k(f,Q) = (k+1)d-k^2
\]
for a generic $Q\subseteq H_\infty$. It is interesting to observe that the above value is always smaller than the generic $k$th-order distance degree of the Veronese embedding $\nu_1^d\colon\P^1\hookrightarrow\P^d$, which is equal to $(k+2)d-k(k+1)$, see Example \ref{ex: gDD Veronese embedding projective line}. The difference between the two values is $d-k$. We leave for future research a systematic study of the relation between the generic $k$th-order distance degree of a projective embedding $f\colon X\hookrightarrow\P^n$ and the ``generic'' affine $k$th-order distance degree of a dehomogenization of $f$ in $\A^n$. So far, conditions yielding an equality between the two numbers are known only in the classical case $k=1$, see \cite[Theorem 6.11]{DHOST}. These conditions involve the intersections $\overline{f(\A^m)}\cap H_\infty$ and $\overline{f(\A^m)}\cap Q$. Despite this, a general formula that quantifies the possible discrepancy between the two metric invariants is still missing even for $k=1$.

A natural numerical approach to distance optimization over a parametrized variety is to apply Newton's method to the objective function $D_u = d_u^2 \circ f$.
Given an iterate $p_j$, the Newton step $p_{j+1}$ is obtained by solving the linear equation system
\[
    \nabla^2 D_u(p_j)(p_{j+1} - p_j) =  - \nabla D_u(p_j)\,.
\]
Despite the strong local convergence guarantees of Newton's method, it is common in applications to replace the Hessian by the approximation $\nabla^2 D_u(p) \approx 2J_f(p)^\mT\,J_f(p)$, leading to the {\em Gauss-Newton method} (see \cite[Chapter 10.2]{dennis1996numerical}), which is often faster and numerically more stable.
In the following proposition, we show that the second-order criticality condition in \eqref{eq: kth order critical affine} implies the equality $\nabla^2 D_u(p) = 2J_f(p)^\mT\,J_f(p)$. From this perspective, second-order distance degrees quantify the prevalence of data configurations for which distance optimization admits solutions that are not only critical but also amenable to fast, numerically stable algorithms.
\begin{proposition}\label{prop: hessian}
For $z = f(p)$, the second-order criticality condition $\nabla d_u^2(z)\in N_p^2(f,Q)$ implies $\nabla^2D_u(p) = 2\,J_f(p)^\mT\,Q\,J_f(p)$.
\end{proposition}
\begin{proof}
We regard the vectors $p,z,u$ as row vectors in $\A^n$. Recall that $d_u^2(z)=(u-z)\,Q\,(u-z)^\mT$, hence $\nabla d_u^2 = 2(u-z)\,Q$ and $\nabla^2 d_u^2 = 2\,Q$.
Using the chain rule, we compute the Hessian
\begin{align*}
    \nabla^2D_u(p) &= J_f(p)^\mT\,\nabla^2 d_u^2\bigl(f(p)\bigr)\,J_f(p)+\sum_{i=1}^n\frac{\partial d_u^2}{\partial z_i}\bigl(f(p)\bigr)\,\nabla^2 f_i(p)\\
    &= 2\,J_f(p)^\mT\,Q\,J_f(p)+2\sum_{i=1}^n (u_i-f_i(p))\,Q\,\nabla^2 f_i(p)\,.
\end{align*}
Since the symmetric form $\nabla^2 f(p)$ generates a subspace of the osculating space $T_p^2(f)$, the second summand at the right-hand side of the previous expansion vanishes if the second-order osculating condition $\nabla d_u^2(z)\in N_p^2(f,Q)$ holds.
\end{proof}
We conclude this section by revisiting \cite[Example 3.3]{DHOST} on the {\em $n$-view triangulation} problem, which consists of recovering the location of a $3$D point from its noisy $2$D projections in $n$ distinct camera images.
In particular, we compute the $2$nd-order distance locus of a special multiview variety.
By Proposition \ref{prop: hessian}, for any data point $u$ on the $2$nd-order distance locus and for any second-order critical
solution $z=f(p)$ of the distance problem with respect to $u$, the Gauss-Newton approximation becomes {\em exact} at $p$.
Thus, the explicit defining equations of the $2$nd-order distance locus provide an algebraic {\em certificate} for data configurations where distance optimization over the multiview variety admits solutions that can be solved by Gauss-Newton.

\begin{example}\label{example: 4-view variety}
Working in projective coordinates, we fix a collection of $n\ge 1$ camera matrices $A = (A_1, \dots, A_n)$, $A_i \in \R^{3 \times 4}$, with $\rank A_i=3$ for all $i\in[n]$. Multiplication by $A_i$ is well-defined away from the focal point $\ker A_i$ and models the image formation process of the $i$-th camera.
The {\em multiview variety} $Y_{n,A}$ consists of all image measurements arising from points in projective three-space. It is defined as the Zariski closure of the image of the rational map
\[
\varphi_{n,A} \colon \P^3 \dashrightarrow (\P^2)^n, \quad
y \mapsto (A_1 y, \dots, A_n y)\,.
\]
We work on the affine chart $[1:\,\cdot:\,\cdot\,]\in\P^2$ in each factor of $(\P^2)^n$ and denote by $X_{n,A} \coloneqq Y_{n,A} \cap \R^{2n}$ the set of measurements away from infinity. We also call $f_{n,A}$ an affine parametrization of $X_{n,A}$ obtained from $\varphi_{n,A}$. We will be more specific about $f_{n,A}$ in the upcoming case study.
Given a noisy measurement $u\in \R^{2n}$, the problem of {\em $n$-view triangulation} from computer vision is to find a point $z\in X_{A,n}$ with minimal Euclidean distance to $u$. The Euclidean distance degree of the affine multiview variety $X_{n,A}$ governs the algebraic complexity of this problem. This invariant was computed in \cite[Eq. (4.1)]{maxim2020euclidean} for a sufficiently generic collection of cameras $A$:
\begin{equation}\label{eq: EDD multiview variety}
    \DD(f_{n,A},Q_\mED) = \frac{9}{2}n^3-\frac{21}{2}n^2+8n-4\,.
\end{equation}
We consider the case $n=4$ and the collection $A'=(A_1',\dots,A_4')$ given in \cite[Section 4]{aholt2013hilbert}:
\[
A_1' = 
\begin{pmatrix}
0 & 1 & 0 & 0\\
0 & 0 & 1 & 0\\
0 & 0 & 0 & 1
\end{pmatrix}, \quad
A_2' = 
\begin{pmatrix}
1 & 0 & 0 & 0\\
0 & 0 & 1 & 0\\
0 & 0 & 0 & 1
\end{pmatrix}, \quad
A_3' = 
\begin{pmatrix}
1 & 0 & 0 & 0\\
0 & 1 & 0 & 0\\
0 & 0 & 0 & 1
\end{pmatrix}, \quad
A_4' = 
\begin{pmatrix}
1 & 0 & 0 & 0 \\
0 & 1 & 0 & 0 \\
0 & 0 & 1 & 0 
\end{pmatrix}.
\]
In this case, the variety $Y_{4,A'}$ is toric.
On the one hand, thanks to \cite[Remark 4.4]{aholt2013hilbert}, the collection $A'$ is {\em universal} in the sense that every multiview variety $Y_{4,A}$ with respect to a collection $A$ of four cameras in linearly general position (meaning that the four focal points $\ker A_1,\dots,\ker A_4\in\P^3$ are not coplanar, and no three of them are collinear) is isomorphic to the toric variety $Y_{4,A'}$. On the other hand, since the collection $A'$ is special, we can only conclude that $\DD(f_{4,A'},Q_\mED)\le 148$ using \eqref{eq: EDD multiview variety}. Indeed, we verified symbolically that $\DD(f_{4,A'},Q_\mED)=5$, while $\DD(f_{A'},Q)=9$ for a sufficiently general quadric $Q\subseteq H_\infty$.

A defining binomial prime ideal of $Y_{4,A'}$ has been presented in \cite[Proposition 4.1]{aholt2013hilbert}. From this implicit description, we derive the following parametric description of $X_{4,A'}$ as the Zariski closure of the image of the monomial map
\begin{align*}
\varphi \colon (\C^*)^3 \to (\P^2)^4\,,\quad (t_1, t_2, t_3)&\mapsto
\left(\left[1:\frac{1}{t_1}:\frac{t_2}{t_1}\right], \left[1:\frac{t_3}{t_2}:t_3\right], \left[1:\frac{t_1t_3}{t_2}:t_3\right], \left[1:\frac{t_1t_3}{t_2}:\frac{t_3}{t_2}\right]\right)\,.
\end{align*}
Dehomogenizing the monomial map above with respect to the affine chart $[1:\,\cdot:\,\cdot\,]\in\P^2$ in each factor of $(\P^2)^4$, we obtain the affine morphism $f_{4,A'}\colon (\C^\ast)^3 \to\C^8$, $f(t_1,t_2,t_3)=(u_1,\dots,u_8)$ where
\[
(u_1,u_2,u_3,u_4,u_5,u_6,u_7,u_8)
=
\left(\frac{1}{t_1},\frac{t_2}{t_1},\frac{t_3}{t_2},t_3,\frac{t_1t_3}{t_2},t_3,\frac{t_1t_3}{t_2},\frac{t_3}{t_2}\right)\,.
\]
Observe that the image of $f_{4,A'}$ is contained in the $5$-dimensional linear subspace
\[
L \coloneqq\{u_4=u_6,\ u_5=u_7,\ u_3=u_8\}\subseteq \A^8\,,
\]
in particular $T^2_p(f_{4,A'})\subseteq L$ for all $p\in(\C^\ast)^3$.
We now show that $T^2_p(f_{4,A'})=L$ for all $p\in(\C^\ast)^3$, in particular the affine generic $2$-osculating dimension $m_2$ is equal to $5$, see Definition~\ref{def: affine osculating space}. Recall that $T^2_p(f_{4,A'})=\rowspan A^{(2)}_p(f_{4,A'})$, where $A^{(k)}_p(f_{4,A'})$ is defined in \ref{eq: affine matrices Apk}.
Consider the $5\times 5$ minor of $A^{(2)}_p(f_{4,A'})$
\[
    M = 
    \begin{pNiceMatrix}[last-col=6,first-row=2]
    u_1&u_2&u_3&u_4&u_5&\\[4pt]
    -\frac{1}{t_1^2}&-\frac{t_2}{t_1^2}&0&0&\frac{t_3}{t_2}&\frac{\partial f}{\partial t_1}\\[4pt]
    0&\frac{1}{t_1}&-\frac{t_3}{t_2^2}&0&-\frac{t_1t_3}{t_2^2}&\frac{\partial f}{\partial t_2}\\[4pt]
    0&0&\frac{1}{t_2}&1&\frac{t_1}{t_2}&\frac{\partial f}{\partial t_3}\\[4pt]
    \frac{1}{t_1^3}&\frac{t_2}{t_1^3}&0&0&0&\frac{1}{2}\frac{\partial^2 f}{\partial t_1^2}\\[4pt]
    0&0&\frac{t_3}{t_2^3}&0&\frac{t_1t_3}{t_2^3}&\frac{1}{2}\frac{\partial^2 f}{\partial t_2^2}
    \end{pNiceMatrix}\,.
\]
One verifies that $\det M = -\frac{t_3^2}{t_1^4t_2^4}\neq 0$ for all $p=(t_1,t_2,t_3)\in(\C^\ast)^3$.
Therefore $\rank A^{(2)}_p(f_{4,A'})\ge 5$ for all $p\in(\C^\ast)^3$, and since $\rank A^{(2)}_p(f_{4,A'})\le\dim L=5$, we conclude that the second-order osculating spaces are globally constant, namely $T^2_p(f_{4,A'})=L$ for all $p\in(\C^\ast)^3$, and that $m_2=5$.

Let $Q=Q_\mED\subseteq H_\infty$ be the standard Euclidean quadric. By Definition \ref{def: affine kth order normal space}, $N^2_p(f_{4,A'},Q_\mED)=(T^2_p(f_{4,A'}))^\perp=L^\perp\subseteq \A^8$, which is independent of $p$ and equals the $3$-plane $L^\perp=\langle e_4-e_6,\ e_5-e_7,\ e_3-e_8\rangle$.
Consequently, by Definition~\ref{def: affine kth oder distance locus}, the second-order data locus admits the simple geometric description
\[
\DL_2(f_{4,A'},Q_\mED)=\overline{\bigcup_{p\in U_2}\left(f_{4,A'}(p)+N^2_p(f_{4,A'},Q_\mED)\right)}
=\overline{\bigcup_{p\in(\C^\ast)^3}\left(f_{4,A'}(p)+L^\perp\right)}\,,
\]
i.e., $\DL_2(f_{4,A'},Q_\mED)$ is a ruled variety obtained by translating the fixed $3$-plane $L^\perp$ along the multiview image $X_{4,A'}\subseteq L$.
Equivalently, $u\in \DL_2(f_{4,A'},Q_\mED)$ if and only if there exist $p=(t_1,t_2,t_3)\in(\C^\ast)^3$ and
$(\lambda_1,\lambda_2,\lambda_3)\in\C^3$ such that
\begin{equation}\label{eq: parametrize data locus L}
u=f_{4,A'}(t_1,t_2,t_3)+\lambda_1(e_4-e_6)+\lambda_2(e_5-e_7)+\lambda_3(e_3-e_8)\,.
\end{equation}
Since the map $f_{4,A'}$ contains denominators in $t_1$ and $t_2$, we consider additional variables $v_1$ and $v_2$ such that $v_1t_1=v_2t_2=1$.
Eliminating the $t_i$'s, $v_i$'s and $\lambda_i$'s from the ideal generated by these relations and the equations in \eqref{eq: parametrize data locus L} yields the prime ideal of $\DL_2(f_{4,A'},Q_\mED)\subseteq\A^8$, which is minimally generated by the polynomials
\[
u_3+u_8-u_1(u_5+u_7)\,,\quad
u_4+u_6-u_2(u_5+u_7)\,,\quad u_1(u_4+u_6)-u_2(u_3+u_8)\,.
\]
We verified that $\codim\DL_2(f_{4,A'},Q_\mED)=2$ and $\deg\DL_2(f_{4,A'},Q_\mED)=3$, therefore $\DL_2(f_{4,A'},Q_\mED)$ is not a complete intersection.
Moreover, on the dense open subset where $u_1u_2\neq 0$ one reconstructs uniquely
\[
t_1=\frac1{u_1},\qquad t_2=\frac{u_2}{u_1},\qquad t_3=\frac{u_4+u_6}{2},
\]
and then $\lambda_1,\lambda_2,\lambda_3$ are uniquely determined by the coordinate differences. Therefore the map $\varphi_{2,2}\colon\DC_2(f_{4,A'},Q_\mED)\to\DL_2(f_{4,A'},Q_\mED)$ is generically one-to-one, and the hypothesis of Corollary \ref{corollary: generically finite} applies. This confirms that $\DL_2(f_{4,A'},Q_\mED)$ is irreducible of dimension $6$ and
\[
\DD_2(f_{4,A'},Q_\mED) = \deg\DL_2(f_{4,A'},Q_\mED)\cdot\deg\varphi_{2,2} = 3\cdot 1=3\,.
\]
Denote by $\bar{f}_{4,A'}$ the projective morphism associated with $f_{4,A'}$ whose image is the projective closure $\overline{X_{4,A'}}\subseteq\P^8$ of $X_{4,A'}$.
It is interesting to compare $\DD_k(f_{4,A'},Q_\mED)$ with $\DD_k(\bar{f}_{4,A'},Q_\mED)$ for $k\in\{1,2\}$. We verified symbolically that $\DD_1(\bar{f}_{4,A'},Q_\mED)=2$ and $\DD_2(\bar{f}_{4,A'},Q_\mED)=3$. In particular, notice that in this case $\DD_1(\bar{f}_{4,A'},Q_\mED)=2<5=\DD_1(f_{4,A'},Q_\mED)$, a phenomenon similar to \cite[Eq. (6.3)]{DHOST}.
\hfill$\diamondsuit$
\end{example}

\bibliographystyle{alphaurl}
\bibliography{biblio}

\newcommand{\etalchar}[1]{$^{#1}$}
\begin{thebibliography}{DRHNP13}

\bibitem[ACGH85]{arbarello1985geometry}
E.~Arbarello, M.~Cornalba, P.~A. Griffiths, and J.~Harris.
\newblock {\em Geometry of algebraic curves. {V}ol. {I}}, volume 267 of {\em Grundlehren der mathematischen Wissenschaften}.
\newblock Springer-Verlag, New York, 1985.
\newblock \href {https://doi.org/10.1007/978-1-4757-5323-3} {\path{doi:10.1007/978-1-4757-5323-3}}.

\bibitem[AK24]{amari2024adversarial}
E.~Aamari and A.~Knop.
\newblock Adversarial manifold estimation.
\newblock {\em Found. Comput. Math.}, 24(1):1--97, 2024.
\newblock \href {https://doi.org/10.1007/s10208-022-09588-2} {\path{doi:10.1007/s10208-022-09588-2}}.

\bibitem[AMR19]{araujo2019defectivity}
C.~Araujo, A.~Massarenti, and R.~Rischter.
\newblock On non-secant defectivity of {S}egre-{V}eronese varieties.
\newblock {\em Trans. Amer. Math. Soc.}, 371(4):2255--2278, 2019.
\newblock \href {https://doi.org/10.1090/tran/7306} {\path{doi:10.1090/tran/7306}}.

\bibitem[AST13]{aholt2013hilbert}
C.~Aholt, B.~Sturmfels, and R.~Thomas.
\newblock A {H}ilbert scheme in computer vision.
\newblock {\em Canad. J. Math.}, 65(5):961--988, 2013.
\newblock \href {https://doi.org/10.4153/CJM-2012-023-2} {\path{doi:10.4153/CJM-2012-023-2}}.

\bibitem[Ban38]{banach1938uber}
S.~Banach.
\newblock {\"Uber homogene Polynome in $(L^2)$}.
\newblock {\em Studia Math.}, 7:36--44, 1938.

\bibitem[BCGI07]{bernardi2007osculating}
A.~Bernardi, M.~V. Catalisano, A.~Gimigliano, and M.~Id\`a.
\newblock Osculating varieties of {V}eronese varieties and their higher secant varieties.
\newblock {\em Canad. J. Math.}, 59(3):488--502, 2007.
\newblock \href {https://doi.org/10.4153/CJM-2007-021-6} {\path{doi:10.4153/CJM-2007-021-6}}.

\bibitem[BF03]{ballico2003secant}
E.~Ballico and C.~Fontanari.
\newblock On the secant varieties to the osculating variety of a {V}eronese surface.
\newblock {\em Cent. Eur. J. Math.}, 1(3):315--326, 2003.
\newblock \href {https://doi.org/10.2478/BF02475212} {\path{doi:10.2478/BF02475212}}.

\bibitem[BH75]{BrysonHo}
A.~E. Bryson, Jr. and Y.~C. Ho.
\newblock {\em Applied optimal control}.
\newblock Hemisphere Publishing Corp., Washington, DC; distributed by Halsted Press [John Wiley \& Sons, Inc.], New York-London-Sydney, 1975.
\newblock Optimization, estimation, and control, Revised printing.

\bibitem[BKS24]{breiding2024metric}
P.~Breiding, K.~Kohn, and B.~Sturmfels.
\newblock {\em Metric algebraic geometry}, volume~53 of {\em Oberwolfach Seminars}.
\newblock Birkh\"auser/Springer, Cham, [2024] \copyright 2024.
\newblock \href {https://doi.org/10.1007/978-3-031-51462-3} {\path{doi:10.1007/978-3-031-51462-3}}.

\bibitem[BRW25]{breiding2025critical}
P.~Breiding, K.~Ranestad, and M.~Weinstein.
\newblock Critical curvature of algebraic surfaces in three-space.
\newblock {\em Acta Univ. Sapientiae Math.}, 17(1), 2025.
\newblock \href {https://doi.org/10.1007/s44426-025-00001-3} {\path{doi:10.1007/s44426-025-00001-3}}.

\bibitem[BW24]{brandt2024voronoi}
M.~Brandt and M.~Weinstein.
\newblock Voronoi cells in metric algebraic geometry of plane curves.
\newblock {\em Math. Scand.}, 130(1):59--85, 2024.

\bibitem[CGG02]{catalisano2002secant}
M.~V. Catalisano, A.~V. Geramita, and A.~Gimigliano.
\newblock On the secant varieties to the tangential varieties of a {V}eronesean.
\newblock {\em Proc. Amer. Math. Soc.}, 130(4):975--985, 2002.
\newblock \href {https://doi.org/10.1090/S0002-9939-01-06251-7} {\path{doi:10.1090/S0002-9939-01-06251-7}}.

\bibitem[CP05]{cazals2005estimating}
F.~Cazals and M.~Pouget.
\newblock Estimating differential quantities using polynomial fitting of osculating jets.
\newblock {\em Comput. Aided Geom. Design}, 22(2):121--146, 2005.
\newblock \href {https://doi.org/10.1016/j.cagd.2004.09.004} {\path{doi:10.1016/j.cagd.2004.09.004}}.

\bibitem[CRSW22]{cifuentes2022voronoi}
D.~Cifuentes, K.~Ranestad, B.~Sturmfels, and M.~Weinstein.
\newblock Voronoi cells of varieties.
\newblock {\em J. Symbolic Comput.}, 109:351--366, 2022.
\newblock \href {https://doi.org/10.1016/j.jsc.2020.07.009} {\path{doi:10.1016/j.jsc.2020.07.009}}.

\bibitem[CS13]{cartwright2013number}
D.~Cartwright and B.~Sturmfels.
\newblock The number of eigenvalues of a tensor.
\newblock {\em Linear Algebra Appl.}, 438(2):942--952, 2013.
\newblock Tensors and Multilinear Algebra.
\newblock URL: \url{https://www.sciencedirect.com/science/article/pii/S0024379511004629}, \href {https://doi.org/10.1016/j.laa.2011.05.040} {\path{doi:10.1016/j.laa.2011.05.040}}.

\bibitem[Dan78]{danilov1978geometry}
V.~I. Danilov.
\newblock The geometry of toric varieties.
\newblock {\em Uspekhi Mat. Nauk}, 33(2(200)):85--134, 247, 1978.

\bibitem[DDRP14]{Higher_duality_and_toric}
A.~Dickenstein, S.~Di~Rocco, and R.~Piene.
\newblock Higher order duality and toric embeddings.
\newblock {\em Ann. Inst. Fourier (Grenoble)}, 64(1):375--400, 2014.
\newblock \href {https://doi.org/10.5802/aif.2851} {\path{doi:10.5802/aif.2851}}.

\bibitem[DDRP24]{dickenstein2024interpolation}
A.~Dickenstein, S.~Di~Rocco, and R.~Piene.
\newblock Interpolation of toric varieties.
\newblock {\em New York J. Math.}, 30:1498--1516, 2024.

\bibitem[DEF{\etalchar{+}}25]{OSCAR-book}
W.~Decker, C.~Eder, C.~Fieker, M.~Horn, and M.~Joswig, editors.
\newblock {\em The {C}omputer {A}lgebra {S}ystem {OSCAR}: {A}lgorithms and {E}xamples}, volume~32 of {\em Algorithms and {C}omputation in {M}athematics}.
\newblock Springer, 1 edition, 2025.
\newblock URL: \url{https://link.springer.com/book/9783031621260}, \href {https://doi.org/10.1007/978-3-031-62127-7} {\path{doi:10.1007/978-3-031-62127-7}}.

\bibitem[Dem70]{demazure1970sousgroupes}
M.~Demazure.
\newblock Sous-groupes alg\'ebriques de rang maximum du groupe de {C}remona.
\newblock {\em Ann. Sci. \'Ecole Norm. Sup. (4)}, 3:507--588, 1970.
\newblock URL: \url{http://www.numdam.org/item?id=ASENS_1970_4_3_4_507_0}.

\bibitem[DFS07]{Tropical_Discriminants}
A.~Dickenstein, E.~M. Feichtner, and B.~Sturmfels.
\newblock Tropical discriminants.
\newblock {\em J. Amer. Math. Soc.}, 20(4):1111--1133, 2007.
\newblock \href {https://doi.org/10.1090/S0894-0347-07-00562-0} {\path{doi:10.1090/S0894-0347-07-00562-0}}.

\bibitem[DHO{\etalchar{+}}16]{DHOST}
J.~Draisma, E.~Horobe\c{t}, G.~Ottaviani, B.~Sturmfels, and R.~R. Thomas.
\newblock The {E}uclidean distance degree of an algebraic variety.
\newblock {\em Found. Comput. Math.}, 16(1):99--149, 2016.
\newblock \href {https://doi.org/10.1007/s10208-014-9240-x} {\path{doi:10.1007/s10208-014-9240-x}}.

\bibitem[DR99]{dirocco1999generation}
S.~Di~Rocco.
\newblock Generation of {$k$}-jets on toric varieties.
\newblock {\em Math. Z.}, 231(1):169--188, 1999.
\newblock \href {https://doi.org/10.1007/PL00004722} {\path{doi:10.1007/PL00004722}}.

\bibitem[DRGS26]{dirocco2024Relative}
S.~Di~Rocco, L.~Gustafsson, and L.~Sodomaco.
\newblock Conditional euclidean distance optimization via relative tangency.
\newblock {\em Math. Comp.}, 95:477--524, 2026.
\newblock URL: \url{https://doi.org/10.1137/19M1265776}, \href {https://doi.org/10.1090/mcom/4047} {\path{doi:10.1090/mcom/4047}}.

\bibitem[DRHNP13]{polyhedralAdjunction}
S.~Di~Rocco, C.~Haase, B.~Nill, and A.~Paffenholz.
\newblock Polyhedral adjunction theory.
\newblock {\em Algebra Number Theory}, 7(10):2417--2446, 2013.
\newblock \href {https://doi.org/10.2140/ant.2013.7.2417} {\path{doi:10.2140/ant.2013.7.2417}}.

\bibitem[DRRS25]{github}
S.~Di~Rocco, K.~Rose, and L.~Sodomaco.
\newblock {O}sculating {G}eometry and {H}igher-{O}rder {D}istance {L}oci - {G}it{H}ub repository, 2025.
\newblock Available at \url{https://github.com/kemalrose/Higher-order-distance-degrees}.

\bibitem[DRS01]{dirocco2001line}
S.~Di~Rocco and A.~J. Sommese.
\newblock Line bundles for which a projectivized jet bundle is a product.
\newblock {\em Proc. Amer. Math. Soc.}, 129(6):1659--1663, 2001.
\newblock \href {https://doi.org/10.1090/S0002-9939-00-05875-5} {\path{doi:10.1090/S0002-9939-00-05875-5}}.

\bibitem[DS96]{dennis1996numerical}
J.~E. Dennis, Jr. and Robert~B. Schnabel.
\newblock {\em Numerical methods for unconstrained optimization and nonlinear equations}, volume~16 of {\em Classics in Applied Mathematics}.
\newblock Society for Industrial and Applied Mathematics (SIAM), Philadelphia, PA, 1996.
\newblock \href {https://doi.org/10.1137/1.9781611971200} {\path{doi:10.1137/1.9781611971200}}.

\bibitem[dSL08]{desilva2008tensor}
V.~de~Silva and L.-H. Lim.
\newblock Tensor rank and the ill-posedness of the best low-rank approximation problem.
\newblock {\em SIAM J. Matrix Anal. Appl.}, 30(3):1084--1127, 2008.
\newblock \href {https://doi.org/10.1137/06066518X} {\path{doi:10.1137/06066518X}}.

\bibitem[EY36]{eckart1936approximation}
C.~Eckart and G.~Young.
\newblock The approximation of one matrix by another of lower rank.
\newblock {\em Psychometrika}, 1(3):211--218, 1936.
\newblock \href {https://doi.org/10.1007/BF02288367} {\path{doi:10.1007/BF02288367}}.

\bibitem[FIK{\etalchar{+}}18]{fefferman2018fitting}
C.~Fefferman, S.~Ivanov, Y.~Kurylev, M.~Lassas, and H.~Narayanan.
\newblock Fitting a putative manifold to noisy data.
\newblock In S\'ebastien Bubeck, Vianney Perchet, and Philippe Rigollet, editors, {\em Proceedings of the 31st Conference On Learning Theory}, volume~75 of {\em Proceedings of Machine Learning Research}, pages 688--720. PMLR, 06--09 Jul 2018.
\newblock URL: \url{https://proceedings.mlr.press/v75/fefferman18a.html}.

\bibitem[FMN16]{fefferman2016testing}
C.~Fefferman, S.~Mitter, and H.~Narayanan.
\newblock Testing the manifold hypothesis.
\newblock {\em J. Amer. Math. Soc.}, 29(4):983--1049, 2016.
\newblock \href {https://doi.org/10.1090/jams/852} {\path{doi:10.1090/jams/852}}.

\bibitem[FO14]{friedland2014number}
S.~Friedland and G.~Ottaviani.
\newblock The number of singular vector tuples and uniqueness of best rank-one approximation of tensors.
\newblock {\em Found. Comput. Math.}, 14(6):1209--1242, 2014.
\newblock \href {https://doi.org/10.1007/s10208-014-9194-z} {\path{doi:10.1007/s10208-014-9194-z}}.

\bibitem[FS97]{sturmfels1997intersection}
W.~Fulton and B.~Sturmfels.
\newblock Intersection theory on toric varieties.
\newblock {\em Topology}, 36(2):335--353, 1997.
\newblock \href {https://doi.org/10.1016/0040-9383(96)00016-X} {\path{doi:10.1016/0040-9383(96)00016-X}}.

\bibitem[FS05]{Feichtner2005}
E.~M. Feichtner and B.~Sturmfels.
\newblock Matroid polytopes, nested sets and {B}ergman fans.
\newblock {\em Port. Math. (N.S.)}, 62(4):437--468, 2005.

\bibitem[Ful98]{fulton1998intersection}
W.~Fulton.
\newblock {\em Intersection theory}, volume~2 of {\em Ergebnisse der Mathematik und ihrer Grenzgebiete. 3. Folge.}
\newblock Springer-Verlag, Berlin, second edition, 1998.
\newblock \href {https://doi.org/10.1007/978-1-4612-1700-8} {\path{doi:10.1007/978-1-4612-1700-8}}.

\bibitem[GH78]{griffiths1978principles}
P.~Griffiths and J.~Harris.
\newblock {\em Principles of algebraic geometry}.
\newblock Pure and Applied Mathematics. Wiley-Interscience [John Wiley \& Sons], New York, 1978.

\bibitem[GJ00]{Gawrilow2000}
E.~Gawrilow and M.~Joswig.
\newblock polymake: a framework for analyzing convex polytopes.
\newblock In {\em Polytopes---combinatorics and computation ({O}berwolfach, 1997)}, volume~29 of {\em DMV Sem.}, pages 43--73. Birkh\"auser, Basel, 2000.
\newblock \href {https://doi.org/10.1007/978-3-0348-8438-9_2} {\path{doi:10.1007/978-3-0348-8438-9_2}}.

\bibitem[Gro67]{grothendieck1967elements}
A.~Grothendieck.
\newblock {\'E}l\'ements de g\'eom\'etrie alg\'ebrique. {IV}. \'etude locale des sch\'emas et des morphismes de sch\'emas {IV}.
\newblock {\em Inst. Hautes \'Etudes Sci. Publ. Math.}, 32:5--361, 1967.
\newblock URL: \url{https://www.numdam.org/item/PMIHES_1967__32__5_0/}, \href {https://doi.org/10.1007/BF02732123} {\path{doi:10.1007/BF02732123}}.

\bibitem[GS97]{GS}
D.~Grayson and M.~Stillman.
\newblock Macaulay 2--a system for computation in algebraic geometry and commutative algebra, 1997.

\bibitem[Hol88]{holme1988geometric}
A.~Holme.
\newblock The geometric and numerical properties of duality in projective algebraic geometry.
\newblock {\em Manuscripta Math.}, 61(2):145--162, 1988.
\newblock \href {https://doi.org/10.1007/BF01259325} {\path{doi:10.1007/BF01259325}}.

\bibitem[Hor17]{horobet2017data}
E.~Horobe\c{t}.
\newblock The data singular and the data isotropic loci for affine cones.
\newblock {\em Comm. Algebra}, 45(3):1177--1186, 2017.
\newblock \href {https://doi.org/10.1080/00927872.2016.1172632} {\path{doi:10.1080/00927872.2016.1172632}}.

\bibitem[Hor24]{horobet2024critical}
E.~Horobe\c{t}.
\newblock The critical curvature degree of an algebraic variety.
\newblock {\em J. Symbolic Comput.}, 121:Paper No. 102259, 12, 2024.
\newblock \href {https://doi.org/10.1016/j.jsc.2023.102259} {\path{doi:10.1016/j.jsc.2023.102259}}.

\bibitem[HR22]{horobet2022data}
E.~Horobe\c{t} and J.~I. Rodriguez.
\newblock Data loci in algebraic optimization.
\newblock {\em J. Pure Appl. Algebra}, 226(12):Paper No. 107144, 15, 2022.
\newblock \href {https://doi.org/10.1016/j.jpaa.2022.107144} {\path{doi:10.1016/j.jpaa.2022.107144}}.

\bibitem[HS18]{SturmfelsHelmer}
M.~Helmer and B.~Sturmfels.
\newblock Nearest points on toric varieties.
\newblock {\em Math. Scand.}, 122(2):213--238, 2018.
\newblock \href {https://doi.org/10.7146/math.scand.a-101478} {\path{doi:10.7146/math.scand.a-101478}}.

\bibitem[HW19]{horobet2019offset}
E.~Horobe\c{t} and M.~Weinstein.
\newblock Offset hypersurfaces and persistent homology of algebraic varieties.
\newblock {\em Comput. Aided Geom. Design}, 74:101767, 14, 2019.
\newblock \href {https://doi.org/10.1016/j.cagd.2019.101767} {\path{doi:10.1016/j.cagd.2019.101767}}.

\bibitem[Kat12]{tropIntersTheory}
E.~Katz.
\newblock Tropical intersection theory from toric varieties.
\newblock {\em Collect. Math.}, 63(1):29--44, 2012.
\newblock \href {https://doi.org/10.1007/s13348-010-0014-8} {\path{doi:10.1007/s13348-010-0014-8}}.

\bibitem[Kle86]{kleiman1986tangency}
S.~L. Kleiman.
\newblock Tangency and duality.
\newblock In {\em Proceedings of the 1984 {V}ancouver conference in algebraic geometry}, volume~6 of {\em CMS Conf. Proc.}, pages 163--225. Amer. Math. Soc., Providence, RI, 1986.

\bibitem[KWW25]{kiani2025hardness}
B.~T. Kiani, J.~Wang, and M.~Weber.
\newblock Hardness of learning neural networks under the manifold hypothesis.
\newblock In {\em Proceedings of the 38th International Conference on Neural Information Processing Systems}, NIPS '24, Red Hook, NY, USA, 2025. Curran Associates Inc.

\bibitem[Laz04]{lazarsfeld2017positivity}
R.~Lazarsfeld.
\newblock {\em Positivity in algebraic geometry. {I}}, volume~48 of {\em Results in Mathematics and Related Areas. 3rd Series. A Series of Modern Surveys in Mathematics}.
\newblock Springer-Verlag, Berlin, 2004.
\newblock Classical setting: line bundles and linear series.
\newblock \href {https://doi.org/10.1007/978-3-642-18808-4} {\path{doi:10.1007/978-3-642-18808-4}}.

\bibitem[Lim05]{lim2005singular}
L.-H. Lim.
\newblock Singular values and eigenvalues of tensors: a variational approach.
\newblock In {\em 1st IEEE International Workshop on Computational Advances in Multi-Sensor Adaptive Processing, 2005.}, pages 129--132. IEEE, 2005.

\bibitem[LM99]{lanteri1999higher}
A.~Lanteri and R.~Mallavibarrena.
\newblock Higher order dual varieties of projective surfaces.
\newblock {\em Comm. Algebra}, 27(10):4827--4851, 1999.
\newblock \href {https://doi.org/10.1080/00927879908826733} {\path{doi:10.1080/00927879908826733}}.

\bibitem[MM10]{mordohai2010dimensionality}
P.~Mordohai and G.~Medioni.
\newblock Dimensionality estimation, manifold learning and function approximation using tensor voting.
\newblock {\em J. Mach. Learn. Res.}, 11:411--450, 2010.

\bibitem[MR19]{massarenti2019defectivity}
A.~Massarenti and R.~Rischter.
\newblock Non-secant defectivity via osculating projections.
\newblock {\em Ann. Sc. Norm. Super. Pisa Cl. Sci. (5)}, 19(1):1--34, 2019.

\bibitem[MRW20a]{maxim2020defect}
L.~G. Maxim, J.~I. Rodriguez, and B.~Wang.
\newblock Defect of {E}uclidean distance degree.
\newblock {\em Adv. in Appl. Math.}, 121:102101, 22, 2020.
\newblock \href {https://doi.org/10.1016/j.aam.2020.102101} {\path{doi:10.1016/j.aam.2020.102101}}.

\bibitem[MRW20b]{maxim2020euclidean}
L.~G. Maxim, J.~I. Rodriguez, and B.~Wang.
\newblock Euclidean distance degree of the multiview variety.
\newblock {\em SIAM J. Appl. Algebra Geom.}, 4(1):28--48, 2020.
\newblock \href {https://doi.org/10.1137/18M1233406} {\path{doi:10.1137/18M1233406}}.

\bibitem[MS15]{maclagan2015introduction}
D.~Maclagan and B.~Sturmfels.
\newblock {\em Introduction to tropical geometry}, volume 161 of {\em Graduate Studies in Mathematics}.
\newblock American Mathematical Society, Providence, RI, 2015.
\newblock \href {https://doi.org/10.1090/gsm/161} {\path{doi:10.1090/gsm/161}}.

\bibitem[OO13]{oeding2013eigenvectors}
L.~Oeding and G.~Ottaviani.
\newblock Eigenvectors of tensors and algorithms for {W}aring decomposition.
\newblock {\em J. Symbolic Comput.}, 54:9--35, 2013.
\newblock \href {https://doi.org/10.1016/j.jsc.2012.11.005} {\path{doi:10.1016/j.jsc.2012.11.005}}.

\bibitem[OP13]{osserman2013lifting}
B.~Osserman and S.~Payne.
\newblock Lifting tropical intersections.
\newblock {\em Doc. Math.}, 18:121--175, 2013.

\bibitem[OS20]{ottaviani2020distance}
G.~Ottaviani and L.~Sodomaco.
\newblock The distance function from a real algebraic variety.
\newblock {\em Computer Aided Geometric Design}, 82:101927, oct 2020.
\newblock URL: \url{https://doi.org/10.1016%2Fj.cagd.2020.101927}, \href {https://doi.org/10.1016/j.cagd.2020.101927} {\path{doi:10.1016/j.cagd.2020.101927}}.

\bibitem[OSS14]{ottaviani2014exact}
G.~Ottaviani, P.-J. Spaenlehauer, and B.~Sturmfels.
\newblock Exact solutions in structured low-rank approximation.
\newblock {\em SIAM J. Matrix Anal. Appl.}, 35(4):1521--1542, 2014.
\newblock \href {https://doi.org/10.1137/13094520X} {\path{doi:10.1137/13094520X}}.

\bibitem[Ott13]{ottaviani2013five}
G.~Ottaviani.
\newblock Five lectures on projective invariants.
\newblock {\em Rend. Semin. Mat. Univ. Politec. Torino}, 71(1):119--194, 2013.

\bibitem[Pie78]{piene1978polar}
R.~Piene.
\newblock Polar classes of singular varieties.
\newblock {\em Ann. Sci. \'{E}cole Norm. Sup. (4)}, 11(2):247--276, 1978.
\newblock URL: \url{http://www.numdam.org/item?id=ASENS_1978_4_11_2_247_0}.

\bibitem[Pie83]{Higherdual}
R.~Piene.
\newblock A note on higher order dual varieties, with an application to scrolls.
\newblock In {\em Singularities, {P}art 2 ({A}rcata, {C}alif., 1981)}, volume~40 of {\em Proc. Sympos. Pure Math.}, pages 335--342. Amer. Math. Soc., Providence, RI, 1983.
\newblock \href {https://doi.org/10.1090/pspum/040.2/713259} {\path{doi:10.1090/pspum/040.2/713259}}.

\bibitem[Pie22]{piene2022higher}
R.~Piene.
\newblock Higher order polar and reciprocal polar loci.
\newblock In {\em Facets of algebraic geometry. {V}ol. {II}}, volume 473 of {\em London Math. Soc. Lecture Note Ser.}, pages 238--253. Cambridge Univ. Press, Cambridge, 2022.

\bibitem[Qi05]{Qi05}
L.~Qi.
\newblock Eigenvalues of a real supersymmetric tensor.
\newblock {\em J. Symbolic Comput.}, 40(6):1302--1324, 2005.
\newblock \href {https://doi.org/10.1016/j.jsc.2005.05.007} {\path{doi:10.1016/j.jsc.2005.05.007}}.

\bibitem[Rin13]{RINCON201386}
F.~Rinc\'on.
\newblock Computing tropical linear spaces.
\newblock {\em J. Symbolic Comput.}, 51:86--98, 2013.
\newblock \href {https://doi.org/10.1016/j.jsc.2012.03.008} {\path{doi:10.1016/j.jsc.2012.03.008}}.

\bibitem[RST25]{Rose2025}
K.~Rose, B.~Sturmfels, and S.~Telen.
\newblock Tropical {I}mplicitization {R}evisited.
\newblock In {\em The computer algebra system {OSCAR}}, pages 429--450. Springer, Cham, [2025] \copyright 2025.
\newblock URL: \url{https://doi.org/10.1007/978-3-031-62127-7_17}, \href {https://doi.org/10.1007/978-3-031-62127-7\_17} {\path{doi:10.1007/978-3-031-62127-7\_17}}.

\bibitem[Sev02]{severi1902intersezioni}
F.~Severi.
\newblock Sulle intersezioni delle variet\`a algebriche e sopra i loro caratteri e singolarit\`a proiettive.
\newblock {\em Mem. R. Acc. Sc. Torino}, II, 52:61--118, 1902.

\bibitem[ST08]{TevelevSturmfels}
B.~Sturmfels and J.~Tevelev.
\newblock Elimination theory for tropical varieties.
\newblock {\em Math. Res. Lett.}, 15(3):543--562, 2008.
\newblock \href {https://doi.org/10.4310/MRL.2008.v15.n3.a14} {\path{doi:10.4310/MRL.2008.v15.n3.a14}}.

\bibitem[Stu93]{algorithmsInInvariantTheory}
B.~Sturmfels.
\newblock {\em Algorithms in invariant theory}.
\newblock Texts and Monographs in Symbolic Computation. Springer-Verlag, Vienna, 1993.
\newblock \href {https://doi.org/10.1007/978-3-7091-4368-1} {\path{doi:10.1007/978-3-7091-4368-1}}.

\bibitem[SW89]{seber1989nonlinear}
G.~A.~F. Seber and C.~J. Wild.
\newblock {\em Nonlinear regression}.
\newblock Wiley Series in Probability and Mathematical Statistics: Probability and Mathematical Statistics. John Wiley \& Sons, Inc., New York, 1989.
\newblock \href {https://doi.org/10.1002/0471725315} {\path{doi:10.1002/0471725315}}.

\bibitem[Tod37]{todd1937arithmetical}
J.~A. Todd.
\newblock The {A}rithmetical {I}nvariants of {A}lgebraic {L}oci.
\newblock {\em Proc. London Math. Soc. (2)}, 43(3):190--225, 1937.
\newblock \href {https://doi.org/10.1112/plms/s2-43.3.190} {\path{doi:10.1112/plms/s2-43.3.190}}.

\bibitem[Zha02]{zha2002optimal}
X.~F. Zha.
\newblock Optimal pose trajectory planning for robot manipulators.
\newblock {\em Mech. Mach. Theory}, 37(10):1063--1086, 2002.
\newblock \href {https://doi.org/10.1016/S0094-114X(02)00053-8} {\path{doi:10.1016/S0094-114X(02)00053-8}}.

\end{thebibliography}
\end{document}